\documentclass[11pt, reqno]{amsart}
\usepackage[margin=3.5cm]{geometry}
\usepackage{mathabx}
\usepackage{cancel}

\usepackage{some_macros}

\let\oldtocsection=\tocsection
\let\oldtocsubsection=\tocsubsection

\newcommand{\num}{\#}

\newcommand{\numtil}{\wt{\#}}

\newcommand{\aut}{\op{Aut}}
\newcommand{\orb}{\mathfrak{o}}
\newcommand{\aug}{\epsilon}

\newcommand{\shuf}{\op{Sh}}
\newcommand{\shufbar}{\ovl{\shuf}}
\newcommand{\sht}{\mathfrak{s}}
\newcommand{\lng}{\mathfrak{l}}

\newcommand{\vecv}{{\vec{v}}}
\newcommand{\veca}{{\vec{a}}}
\newcommand{\vecb}{{\vec{b}}}

\newcommand{\mc}{\mathfrak{m}}

\newcommand{\di}{\diamondsuit}

\newcommand{\setm}{\;\setminus\;}

\newcommand{\Ga}{\Gamma}

\newcommand{\dg}{\frak{q}}
\newcommand{\pt}{{pt}}
\newcommand{\mult}{\op{mult}}
\newcommand{\ev}{\op{ev}}
\newcommand{\pd}{\op{PD}}

\newcommand{\calS}{\mathcal{S}}
\newcommand{\Tcount}{\mathbf{T}}
\newcommand{\wtTcount}{\wt{\Tcount}}
\newcommand{\infcount}{{\bf J}}
\newcommand{\jump}{\mathbb{J}}
\newcommand{\fib}{\op{Fib}}
\newcommand{\tame}{\op{tame}}

\newcommand{\fl}{\op{fl}}
\newcommand{\vecz}{\vec{z}}
\newcommand{\rr}{\mathfrak{r}}
\newcommand{\vfa}{{\vec{\mathfrak{a}}}}
\newcommand{\cQ}{\check{Q}}
\newcommand{\cPhi}{\check{\Phi}}
\newcommand{\bd}{{\boldsymbol{\delta}}}
\newcommand{\bdmax}{\bd_{\op{max}}}
\newcommand{\ovlbd}{\ovl{\bd}}

\renewcommand{\im}{\op{im}}

\newcommand{\vk}{{\varkappa}}
\newcommand{\pn}{p_{\vec{1}}}
\newcommand{\Ddiv}{\mathbf{D}}
\newcommand{\calU}{\mathcal{U}}

\newcommand{\nn}{\mathbf{n}}
\newcommand{\intE}{{\mathring{E}}}

\renewcommand{\tocsection}[2]{\hspace{0em}\oldtocsection{#1}{#2}}
\renewcommand{\tocsubsection}[2]{\hspace{1em}\oldtocsubsection{#1}{#2}}

\tikzset{node distance=3cm, auto}

\makeatletter
\def\@secnumfont{\bfseries}
\def\section{\@startsection{section}{1}%
  \z@{.7\linespacing\@plus\linespacing}{.5\linespacing}%
  {\normalfont\Large\bfseries}}
\def\subsection{\@startsection{subsection}{2}%
  \z@{.75\linespacing\@plus.7\linespacing}{-.5em}%
  {\normalfont\large\bfseries}}
\def\subsubsection{\@startsection{subsubsection}{3}%
  \z@{.75\linespacing\@plus.7\linespacing}{-.5em}%
  {\normalfont\bfseries}}
\makeatother

\newtheorem{thm}{Theorem}[subsection]
\newtheorem{thmlet}{Theorem}

\newtheorem{lemma}[thm]{Lemma}
\newtheorem{lemma-notation}[thm]{Lemma/Notation}
\newtheorem{prop}[thm]{Proposition}
\newtheorem{cor}[thm]{Corollary}

\newtheorem{claim}[thm]{Claim}
\newtheorem{sublemma}[thm]{Sublemma}

\newtheorem{conjecture}[thm]{Conjecture}
\newtheorem{definition}[thm]{Definition}
\newtheorem{notation}[thm]{Notation}

\theoremstyle{remark}
\numberwithin{equation}{subsection}

\newtheoremstyle{customremark}
{6pt}
{6pt}
{}
{}
{\bfseries}
{.}
{.5em}
{}
\theoremstyle{customremark}
\newtheorem{rmk_no_diamond}[thm]{Remark}
\newenvironment{rmk}{\begin{rmk_no_diamond} } {\hfill$\Diamond$ \end{rmk_no_diamond}}
\newtheorem{example_no_diamond}[thm]{Example}
\newenvironment{example}{\begin{example_no_diamond} } {\hfill$\Diamond$ \end{example_no_diamond}}

\begin{document}

\makeatother

\date{\today}

\title{Ellipsoidal superpotentials and stationary descendants} 
\date{\today}
\pagestyle{plain}

\begin{abstract}
We compute stationary gravitational descendants in symplectic ellipsoids of any dimension, and use these to derive a number of new recursive formula for punctured curve counts in symplectic manifolds with ellipsoidal ends. 
Along the way we develop a framework in which punctured curve counts can be explicitly computed using the standard complex structure on affine space. Finally, we initiate the study of ``infinitesimal symplectic cobordisms'', which serve as elementary building blocks for symplectic cobordisms between ellipsoids.
\end{abstract}

\author{Grigory Mikhalkin and Kyler Siegel}
\thanks{G.M. is partially supported by SNSF grants 200400 and 204125. K.S. is partially supported by NSF grant DMS-2105578.}

\maketitle

\tableofcontents

\section{Introduction}

\subsection{Prelude}\label{sec:prelude}

Given a positive real number $a$, we define a lattice path 
\begin{align*}
\Gamma^a = (\Gamma_0^a,\Gamma_1^a,\Gamma_2^a,\dots)
\end{align*}
in $\Z_{\geq 0}^2$
as follows.
Let $L_a$ denote the line in $\R^2$ passing through $(0,-1)$ and $(a,0)$.
Then $\Gamma^a$ is the unit step lattice path which starts at $(0,0)$, lies on or above $L_a$, and steps to the right whenever possible, otherwise it steps up.

\begin{example}
For $a = \frac{3}{2}$, we have 
\begin{center}
\begin{tabular}{c|c|c|c|c|c|c|c|c|c} 
 $k$ & 0 & 1 & 2 & 3 & 4 & 5 & 6 & 7 & 8\\ 
 \hline
 $\Gamma^a_k$ & (0,0) & (1,0) & (1,1) & (2,1) & (3,1) & (3,2) & (4,2) & (4,3) & (5,3),\\
\end{tabular}
\end{center}
etc. See Figure \ref{fig:plot1}.
\end{example}

\begin{figure}[h]
\caption{The lattice path $\Gamma^{3/2}$, along with the lines $L_{3/2}$ and $L_{3/2}+(-1,1)$.}
\centering
\includegraphics[width=0.9\textwidth]{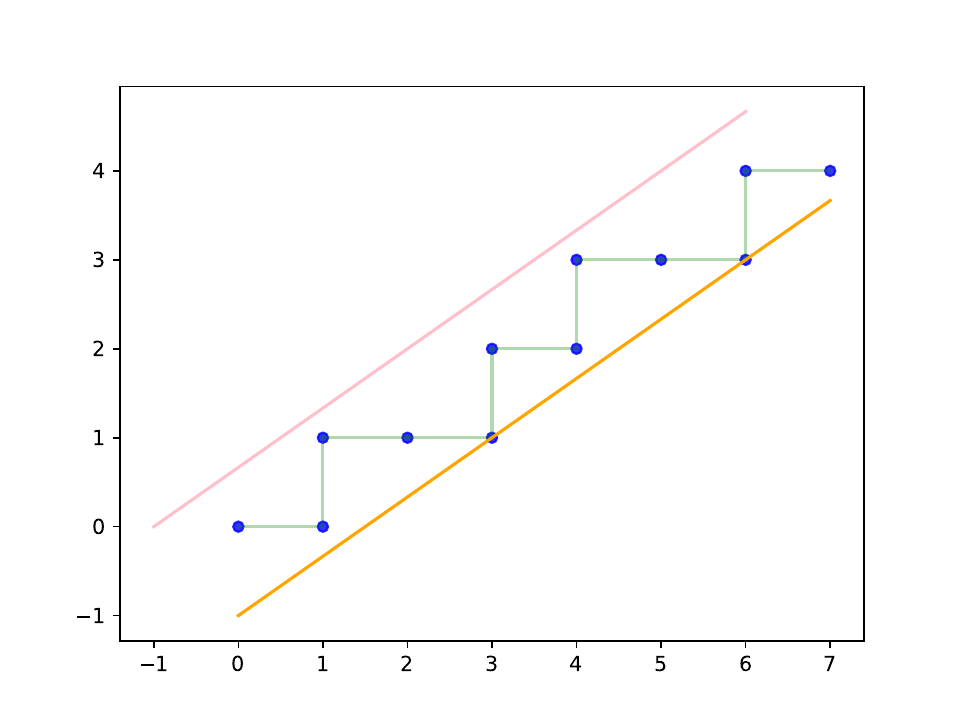}
\label{fig:plot1}
\end{figure}

The lattice path $\Ga^a$ lies on or above $L_a$ and strictly below the line $L_a + (-1,1)$ passing through $(-1,0)$ and $(a-1,1)$, and in fact it is the {\em unique} unit step lattice path with this property. 
Incidentally $\Ga^a$ lies strictly between $L_a$ and $L_a + (-1,1)$ when $a$ is irrational.

Observe that $\tfrac{1}{a}$ is the infimal slope over all lines joining $(0,-1)$ to a point in $\Gamma_a$.
In particular, this means that if $a,b$ are positive real numbers such that $\Gamma^a = \Gamma^b$, i.e. $\Gamma^a_k = \Gamma^b_k$ for all $k \in \Z_{\geq 0}$, then we have $a = b$.
The reader can skip ahead to Figure~\ref{fig:lattice_path_movie} in \S\ref{sec:inf_cobs} for an idea of how $\Ga^a$ changes as we vary $a$.
The higher dimensional analogue of $\Ga^a$ appears in \S\ref{subsec:lattice_path}.

\subsection{Setup and motivation}\label{subsec:setup}

The main object of study in this paper is the invariant $\Tcount_{M,A}^\veca \in \Q$, which we dub the {\bf ellipsoidal superpotential}\footnote{The name is explained in Remark~\ref{rmk:name_expl}.}, and which depends on:
\begin{itemize}
  \item a closed symplectic manifold $M^{2n}$
  \item an integral homology class $A \in H_2(M)$
  \item a vector $\veca = (a_1,\dots,a_n) \in \R_{>0}^n$.
\end{itemize}
The definition is roughly as follows. We first choose a symplectic embedding ${\iota: E(\eps\veca) \hooksymp M}$ for some small $\eps > 0$, where $E(\veca)$ denotes the symplectic ellipsoid ${\{\pi \sum_{i=1}^n |z_i|^2/a_i \leq 1\}}$ in $\C^n$.
Let $M_\veca$ denote the complement of the (interior of) the image of $\iota$, and let $\wh{M}_\veca$ denote its symplectic completion, i.e. the result after attending the seminfinite cylindrical end $\R_{\leq 0} \times \bdy E(\eps \veca)$.
Then $\Tcount_{M,A}^\veca$ is the count of rigid pseudoholomorphic planes in $\wh{M}_\veca$ in homology class $A$. 
As we explain in \S\ref{sec:ell_superpot}, $\Tcount_{M,A}^\veca$ is actually independent of the choice of $\eps$ and $\iota$, as well as any other auxiliary data such as a choice of almost complex structure.

Our goal is to understand the counts $\Tcount_{M,A}^\veca$ to whatever extent possible.
Among other things, we will find that the lattice path 
$\Ga^\veca$
plays a central role in this story.
However, before stating our main results, let us give some motivation for studying ellipsoidal superpotentials.

\subsubsection{Motivation from stabilized symplectic embeddings}\label{subsubsec:motiv_emb}

Studying when one shape symplectically embeds into another of the same dimension has been a central theme in symplectic geometry since Gromov's discovery of the nonsqueezing theorem \cite{gromov1985pseudo}.
The problem is already very rich when the domain is an ellipsoid $E(\veca)$ and the target a closed symplectic manifold $M$. 
Powerful tools such as embedded contact homology have been developed for obstructing such embeddings in dimension four, but much less is known in higher dimensions.  

A recent angle for approaching higher dimensional symplectic embeddings is to look at the so-called stabilized problem $E(\veca) \times \C^N \hooksymp M \times \C^N$ for $N \geq 1$.
Before explaining the connection with ellipsoidal superpotentials, we introduce some notation.
Assume that the components of $\veca$ are rationally independent, so that there are $n$ simple Reeb orbits $\nu_1,\dots,\nu_n$ in $\bdy E(\veca)$, where $\nu_i$ is the intersection of $\bdy E(\veca)$ with the $i$th complex axis. For $k \in \Z_{\geq 1}$, the action (i.e. period) of the $k$-fold iterate $\nu_i^k$ of $\nu_i$ is given by $\calA(\nu_i^k) = ka_i$.
However, it will often be more useful to enumerate the Reeb orbits in $\bdy E(\veca)$ in order of increasing action, i.e. we have the list $\orb^\veca_1,\orb^\veca_2,\orb^\veca_3,\dots$, where $\orb^\veca_k$ denotes the Reeb orbit of $k$th smallest action.
From this point of view, $\calA(\orb^\veca_k)$ is the $k$th smallest element of the multiset $\{ja_i\;|\; i \in \{1,\dots,n\}, j \in \Z_{\geq 1}\}$, which turns out to be equivalent to $\calA(\orb^\veca_k) = \min\limits_{\substack{(v_1,\dots,v_n) \in \Z_{\geq 0}^n \\ v_1+\cdots+v_n = k}} \max\limits_{1 \leq i \leq n} a_iv_i$.

Generalizing recent work \cite{HK,CGH,Ghost,Mint,HSC,chscI,irvine2019stabilized,cristofaro2022higher,mcduff2021symplectic} on stabilized symplectic embeddings, we have:
\begin{prop}\label{prop:stab_emb_obs}
Let $M$ be a closed symplectic manifold and $A \in H_2(M)$ a homology class.
 If $\Tcount_{M,A}^\veca \neq 0$, then any symplectic embedding $E(c\veca) \times \C^N \hooksymp M \times \C^N$ with $N \geq 0$ 
 must satisfy 
 $c \leq \frac{[\omega_M] \,\cdot\, A}{\calA(\orb^\veca_{c_1(A)-1})}$. 
\end{prop}
\NI Here $c_1(A)$ denotes the first Chern number of $A$, and $[\omega_M] \in H^2(M;\R)$ is the cohomology class of the symplectic form of $M$.
\begin{rmk}
  By definition $\Tcount_{M,A}^\veca$ is invariant under scalings of $\veca$.
  In dimension four we will often normalize the vector $(a_1,a_2)$ so that $a_1=1$.
\end{rmk}

In the stable case (i.e. $N \geq 1$), Proposition~\ref{prop:stab_emb_obs} is to our knowledge the only known mechanism for producing obstructions for this problem, so it is natural to ask whether the invariants $\Tcount_{M,A}^\veca$ encode the full answer.
The above references give lots of evidence in the affirmative, at least in the case $M = \CP^2$. 

Many concrete questions in this area remain currently out of reach, such as:
\begin{conjecture}\label{conj:p_q_d}
 For $p,q,d \in \Z_{\geq 1}$ with $p+q = 3d$, $\gcd(p,q) = 1$, and $(p-1)(q-1) \leq (d-1)(d-2)$, we have $\Tcount_{\CP^2,d[L]}^{(q,p)} \neq 0$.\footnote{Here $[L] \in H_2(\CP^2)$ is the line class.}
\end{conjecture}
\NI Together with known results, an affirmative answer to Conjecture~\ref{conj:p_q_d} would give a complete description of the stable analogue of the McDuff--Schlenk ``Fibonacci staircase'' function \cite{McDuff-Schlenk_embedding}.
However, determining when $\Tcount_{M,A}^\veca$ is nonzero is surprisingly delicate, even in the much-studied case $M = \CP^2$.

\subsubsection{Motivation from singular curves}\label{subsubsec:motiv_sing}

A longstanding problem in the study of algebraic curves asks to classify the possible singularities for a plane curve of given degree and genus. For example, the main result in \cite{fernandez2006classification} lists all combinatorial possibilities for a degree $d \in \Z_{\geq 1}$ rational algebraic curve in $\CP^2$ with one $(p,q)$ cusp singularity for some relatively prime $p,q \in \Z_{\geq 1}$ (i.e. the singularity modeled on $\{x^p + y^q = 0\}$, which is topologically the cone over a $(p,q)$ torus knot) and no other singularities. To our knowledge, most extensions of this result (say multiple singularities, multiple Puiseux pairs, other target spaces, higher genus) are at best partially understood. We refer the reader to e.g. \cite{fernandez2006classification,greuel2021plane,orevkov2002rational,moe2015rational,golla2019symplectic} and the references therein for background on this subject.

It turns out that the counts $\Tcount_{M,A}^\veca$ encode lots of information about singular rational curves in $M$. 
Heuristically, the idea is that we should be able to trade closed singular curves in $M$ for asymptotically cylindrical punctured curves in $\wh{M}_\veca$.
The resulting asymptotic Reeb orbits should depend on $\veca$ and the initial singularities,
 and we can try to choose $\veca$ carefully to ``detect'' a given singularity.
Rather than attempting a comprehensive discussion, here we give just one concrete manifestation of this phenomenon:

\begin{prop}
Let $M^4$ be a four-dimensional closed symplectic manifold, let $A \in H_2(M)$ be a homology class, and let $p,q \in \Z_{\geq 1}$ satisfy $p+q = c_1(A)$ and $\gcd(p,q)=1$.
Then $\Tcount_{M,A}^{(q,p)} \neq 0$ if and only if there exists a singular rational symplectic curve in $M$ in homology class $A$ with a $(p,q)$ cusp singularity and otherwise positively immersed.
\end{prop}
\NI Here {\em singular symplectic curve} is the symplectic analogue of a singular algebraic curve (c.f. \cite[Def. 2.5]{golla2019symplectic}). Note that the curve may have other singularities, which by perturbing we can assume are all ordinary double points (i.e. modeled on $\{x^2 = y^2\}$).
Since singular algebraic curves are in particular singular symplectic curves, we have:
\begin{cor}\label{cor:alg_curve_imp}
Suppose that $p,q,d \in \Z_{\geq 1}$ satisfy $p+q+3d$ and $\gcd(p,q) = 1$.
If there exists a degree $d$ rational algebraic curve in $\CP^2$ with a $(p,q)$ cusp singularity, then $\Tcount_{\CP^2,d[L]}^{(q,p)} \neq 0$. 
\end{cor}

For instance, Orevkov \cite{orevkov2002rational} constructed an infinite sequence of rational curves corresponding to ratios of odd index Fibonacci numbers (this is family (d) in \cite{fernandez2006classification}), and these satisfy the hypotheses of Corollary~\ref{cor:alg_curve_imp}, so the corresponding invariants $\Tcount_{\CP^2,d_i[L]}^{(q_i,p_i)}$ are all nonzero. In fact, the obstructions these give via Proposition~\ref{prop:stab_emb_obs} correspond precisely to the outer corners of the Fibonacci staircase.

\begin{rmk}
For $\veca = (a_1,\dots,a_n)$ with $a_2,\dots,a_n \gg a_1$, $\Tcount_{M,A}^\veca$ recovers the count of rational curves satisfying a maximal order local tangency constraint at a specified point in $M$ -- see \cite{McDuffSiegel_counting}.
\end{rmk}

\subsection{Main results}

\subsubsection{Recursive formula}
In \S\ref{sec:formulas}, we prove a recursion formula which computes the ellipsoidal superpotential $\Tcount_{M,A}^\veca$ for any $M,A,\veca$.
It turns to be convenient to put ${\wtTcount_{M,A}^\veca := \mult(\orb^\veca_{c_1(A)-1}) \,\Tcount_{M,A}^\veca}$, where $\mult(\ga)$ denotes the covering multiplicity of the Reeb orbit $\ga$ (i.e. $\mult(\nu_i^k) = k$ for $k \in \Z_{\geq 1}$ and $i = 1,\dots,n$).
For expository purposes in this introduction we restrict to the case $M = \CP^2$ and write $\wtTcount_d^a$ as a shorthand for $\wtTcount_{\CP^2, d[L]}^{(1,a)}$.
For a tuple $\vecv = (v_1,\dots,v_n) \in \Z_{\geq 0}^n$ we put $\vecv\,! := v_1! \cdots v_n!$.

\begin{thmlet}[= Corollary~\ref{cor:rec_CP2}]\label{thm:main_rec_intro}
  For any $d \in \Z_{\geq 1}$ we have:
\begin{align}\label{eq:wtTcount_intro}
\wtTcount_d^a = \left(\Ga^a_{3d-1}\right)!\left( (d!)^{-3} - \sum_{\substack{k \geq 2\\d_1,\dots,d_k \in \Z_{\geq 1}\\d_1 + \cdots + d_k = d}}  \frac{\wtTcount_{d_1}^a \cdot \cdots \cdot \wtTcount_{d_k}^a}{k!(\sum_{s=1}^k \Ga^a_{3d_i-1} )!}\right).
\end{align}
\end{thmlet}

\NI The astute reader may recognize the term $(d!)^{-3}$ as the degree $d$ genus zero stationary descendant Gromov--Witten invariant for $\CP^2$.
The general recursion (see Theorem~\ref{thm:main_recursion}) computes $\Tcount_{M,A}^\veca$ in terms of the stationary descendant Gromov--Witten invariant $N_{M,A}\lll \psi^{c_1(A)-2} \pt \rrr$, which is well-studied (see \S\ref{subsubsec:closed_curves} for the case of Fano toric manifolds).

\begin{rmk}
In the case $M = \CP^2$, a distinct recursive formula is given in \cite{chscI} via ``fully rounding'' convex toric domains. We discuss connections between these two approaches in \S\ref{sec:fully_rounding}), and we refer the reader to \cite{chscI} for many sample computations of $\Tcount_d^a$, all of which are consistent with Conjecture~\ref{conj:p_q_d}.
\end{rmk}

\subsubsection{Stationary descendants in ellipsoids}

A key geometric input to proving Theorem~\ref{thm:main_rec_intro} is to study punctured curve stationary descendants in ellipsoids, which enjoy a surprisingly simple formula in terms of the higher dimensional lattice path $\Ga^\veca$.
Let $\wh{E}(\veca) = E(\veca) \cup (\R_{\geq 0} \times \bdy E(\veca))$ denote the symplectic completion of the ellipsoid $E(\veca)$.
Roughly, $N_{E(\veca)}(\ga_1,\dots,\ga_k) \lll \psi^{m}\pt \rrr$ will denote the count of rational curves in $\wh{E}(\veca)$ with $k$ puctures asymptotic  to Reeb orbits $\ga_1,\dots,\ga_k$ in $\bdy E(\veca)$, and carrying the stationary descendant constraint $\lll \psi^m\pt\rrr$ (see \S\ref{subsubsec:aug_from_des}).
\begin{thmlet}[= Theorem~\ref{thm:main_descendant}]\label{thm:main_descendant_intro}
For any $\veca \in \R_{> 0}^n$ and $i_1,\dots,i_k \in \Z_{\geq 1}$, we have
\begin{align*}
N_{E(\veca)}(\orb^{\veca}_{i_1},\dots,\orb^\veca_{i_k}) \lll \psi^{i_1+\cdots+i_k+k-2}\pt \rrr = 
\frac{1}{(\Ga^\veca_{i_1} + \cdots + \Ga^\veca_{i_k})!}.
  \end{align*}
\end{thmlet}
An important technical ingredient in the proof of Theorem~\ref{thm:main_descendant_intro} is the result, established in \S\ref{sec:J_std}, that SFT moduli spaces for ellipsoids can be computed using the standard complex structure on affine space. The equivalence, which holds for both the symplectic completion of an ellipsoid and the symplectization of its boundary, is given by a nontrivial diffeomorphism. In \S\ref{subsec:asymp_corr} we formulate a precise correspondence theorem involving asymptotic Reeb orbits with winding numbers.

\subsubsection{Infinitesimal cobordisms}

As we explain in \S\ref{sec:inf_cobs}, the counts $\Tcount_{M,A}^a$ can be decomposed via SFT neck stretching into elementary pieces of the form $E^{a^-}_{a^+}$, for $a \in \R_{>0}$, which we call {\em infinitesimal symplectic cobordisms}. Here $E^{a^-}_{a^+}$ is roughly the complementary cobordism of a symplectic embedding (up to scaling) of $E(1,a^+)$ into $E(1,a^-)$, where $a^\pm$ denotes $a \pm \delta$ for $\delta > 0$ sufficiently small.
We view $E^{a^-}_{a^+}$ as the simplest possible nontrivial symplectic cobordism.

In \S\ref{subsec:jump_formulas} we give a general recursive formula which enumerates rational punctured curves in $\wh{E}^{a^-}_{a^+}$ with many positive ends and one negative end, for any $a \in \R_{> 0}$.
In turns out that these counts can only be nontrivial for $a$ lying in a small finite set of rational values,
 and this implies that $\Tcount_{M,A}^a$ is a piecewise constant function with only certain allowable jumps.
\begin{thmlet}[= Corollary~\ref{cor:jumps_of_wtTcount}]\label{thm:jump_intro}
For $d \in \Z_{\geq 1}$, $\wtTcount_d^a$ is piecewise constant as a function of $a$, with discontinuity points contained in   
$\bigcup\limits_{i=1}^d \jump_{3i-1}$, where $\jump_k := \left\{\tfrac{k}{1},\tfrac{k-1}{2},\,\dots,\tfrac{1}{k}\right\}$.
\end{thmlet}
\NI In examples (see e.g. \S\ref{subsec:jump_examples}), $\wtTcount_d^a$ seems to have many fewer jumps than suggested by Theorem~\ref{thm:jump_intro}. It remains an interesting question to determine at precisely which values jumps occur, and to find a more conceptual understanding of what exactly these counts encode.

\section{Ellipsoidal superpotentials}\label{sec:ell_superpot}

In this section we elaborate on the definition of the ellipsoidal superpotential $\Tcount_{M,A}^\veca$.

\subsection{Spaces with ellipsoidal negative ends}

Let $M^{2n}$ be a closed symplectic manifold, and fix $\veca \in \R_{>0}^n$.
We put
\begin{align*}
M_\veca :=  M \setm \iota(\intE(\eps \veca)),
\end{align*}
where $\iota: E(\eps \veca) \hooksymp M$ is a symplectic embedding for some $\eps > 0$. Here $\intE(\veca)$ denotes the interior of $E(\veca)$, i.e. the open ellipsoid $\{\pi \sum_{i=1}^n |z_i|^2/a_i < 1\}$.
Note that $M_\veca$ is a compact symplectic cobordism with negative contact boundary $\bdy E(\eps\veca)$ and no positive boundary.
We can essentially ignore the choice of embedding $\iota$ thanks to the  following fact, which is an easy consequence of the ``extension after restriction'' principle (see e.g. \cite[\S4.4]{Schlenk_old_and_new}):
\begin{lemma}
  Given two symplectic embeddings $\iota_0,\iota_1: E(\eps \veca) \hooksymp M$, we can find $\eps' >0$ sufficiently small so that the restrictions
   $\iota_0|_{E(\eps' \veca)}$ and $\iota_1|_{E(\eps' \veca)}$ differ by postcomposing  with a Hamiltonian diffeomorphism of $M$.
\end{lemma}

We denote by $\wh{M}_\veca$ the symplectic completion of $M_\veca$.
As a smooth manifold, this is given by attaching the negative half-infinite cylindrical end $\R_{\leq 0} \times \bdy E(\eps \veca)$, after using the flow of a locally defined Liouville vector field to trivialize a collar neighborhood of the boundary of $M_\veca$.
The symplectic form $\omega_M$ of $M$ extends smoothly to the cylindrical end as $d(e^r\alpha)$, where $r$ denotes the coordinate on $\R_{\leq 0}$, and $\alpha := \la_\std|_{\bdy E(\eps \veca)}$ denotes the contact form on $\bdy E(\eps\veca)$. 
Here $\la_\std = \tfrac{1}{2}\sum_{i=1} (x_idy_i - y_idx_i)$ is the standard Liouville one-form on $\R^{2n}$.
Similarly, $\wh{E}(\veca)$ denotes the symplectic completion of the compact ellipsoid $E(\veca)$, which is given by attaching the {\em positive} end $\R_{\geq 0} \times \bdy E(\veca)$.

\subsection{SFT admissible almost complex structures}

For $M$ a closed symplectic manifold with symplectic form $\omega_M$, let $\calJ(M) := \calJ(M,\omega_M)$ denote the space of $\omega_M$-compatible almost complex structures on $M$.

If $(Y^{2n-1},\beta)$ is a strict contact manifold, we denote by $\calJ(Y) := \calJ(Y,\beta)$ the space of SFT admissible almost complex structures on the symplectization $\R \times Y$, defined as follows.
 Let $R_\beta$ denote the Reeb vector field on $Y$, which is characterized by the conditions $\beta(R_\beta,-) = 0$ and $\beta(R_\beta) = 1$.
Let $r$ denote the coordinate on the first factor of the symplectization $\R \times Y$.
Then $\calJ(Y)$ is the space of almost complex structures $J$ on $\R \times Y$ which satisfy:
\begin{itemize}
  \item $J$ is $d\beta$-compatible on the contact distribution $\xi := \ker\beta$
  \item $J(\bdy_r) = R_\beta$
  \item $J$ is invariant under translations in the $\R$ factor.
\end{itemize}
  
With few exceptions, all contact manifolds considered in this paper will be ellipsoids.  By default we equip $\bdy E(\eps\veca)$ with the contact form $\alpha = \la_\std|_{\bdy E(\eps \veca)}$, and we equip $\bdy M_\veca$ with its pushforward under $\iota$.
Let $\calJ(M_\veca)$ denote the set of almost complex structures $J$ on $\wh{M}_\veca$ such that 
\begin{itemize}
  \item $J$ is compatible with the symplectic form on the compact piece $M_\veca$
  \item $J$ agrees with the restriction of an element of $\calJ(\bdy M_\veca)$ on the cylindrical end.
\end{itemize}
We define $\calJ_\tame(M_\veca)$ in the same way, except that we only require $J$ to be tame the symplectic form on $M_\veca$. In fact $\calJ_\tame(M_\veca)$ is sufficient for our purposes since e.g. the SFT compactness theorem holds equally well, and the added flexibility will come in handy in \S\ref{sec:J_std}. We will refer to elements in either $\calJ(M_\veca)$ or $\calJ_\tame(M_\veca)$ as SFT admissible almost complex structures.
The spaces $\calJ(E(\veca)),\calJ_\tame(E(\veca))$ of SFT admissible almost complex structures on $\wh{E}(\veca)$ are defined analogously.

\subsection{Moduli spaces of punctured curves}\label{subsec:moduli_spaces}

Fix $\veca = (a_1,\dots,a_n) \in \R_{>0}^n$ with rationally independent components. Recall that $\orb_k^\veca$ denotes the Reeb orbit in $\bdy E(\veca)$ with $k$th smallest action.
The Conley--Zehnder index of $\orb_k^\veca$ as measured by a global trivialization of the contact distribution is given by $\cz(\orb_k^\veca) = n-1+2k$.

All of our moduli spaces of pseudoholomorphic curves in noncompact target spaces will be asymptotically cylindrical in the sense of symplectic field theory.
We also sometimes refer to these informally as ``punctured curves''.

For instance, for $k \in \Z_{\geq 1}$ and $J \in \calJ_\tame(M_\veca)$, we denote by $\calM_{M_\veca,A}^J(\orb_k^\veca)$ the moduli space of $J$-holomorphic maps from a once-punctured Riemann sphere (i.e. a plane) to $\wh{M}_\veca$ which are negatively asymptotic to the Reeb orbit $\orb_k^\veca$ and lie in homology class $A$.
Note that such a curve does indeed give a well-defined homology class in $H_2(M)$ since $H_1(E(\veca))$ and $H_2(E(\veca))$ are both trivial.
The (real) Fredholm index of this moduli space is $n-3 + 2c_1(A) - \cz(\orb_k^\veca)$.

Similarly, given $J \in \calJ_\tame(E(\veca))$ and $i_1,\dots,i_k \in \Z_{\geq 1}$ for some $k \in \Z_{\geq 1}$ , we denote by $\calM_{M_\veca}^J(\orb^\veca_{i_1},\dots,\orb^{\veca}_{i_k})$ the moduli space of $J$-holomorphic maps from a Riemann sphere with $k$ ordered punctures to $\wh{E}(\veca)$ which are positively asymptotic at the punctures to $\orb^\veca_{i_1},\dots,\orb^\veca_{i_k}$ respectively.
Unless stated otherwise, the conformal structure on the domain is unspecified. The punctures (or marked points) will always be ordered unless otherwise stated.
The (real) Fredholm index of this moduli space is $(n-3)(2-k) + \sum\limits_{s=1}^k \cz(\orb^\veca_{i_s})$.

For each $\veca \in \R_{>0}^n$, we have a preferred choice $J_{\bdy E(\veca)} \in \calJ(\bdy E(\veca))$ of almost complex structure on the symplectization $\R \times \bdy E(\veca)$ (see Definition~\ref{def:J_E_etc}).
We will usually assume that our almost complex structures on $\wh{M}_\veca$ or $\wh{E}(\veca)$ agree with $J_{\bdy E(\veca)}$ on the cylindrical end.
By construction $J_{\bdy E(\veca)}$ is diffeomorphic to the standard complex structure $i$ on $\C^n \setm \{\vec{0}\}$. In particular it is preserved by a circle action which corresponds to the diagonal circle action on $\C^n \setm \{\vec{0}\}$.
Together with the $\R$ action by translations in the first factor this gives a $\C^*$ action on $\R \times \bdy E(\veca)$ which preserves $J_{\bdy E(\veca)}$.

We have a natural compactification $\ovl{\calM}_{M_\veca}^J(\orb^\veca_{k})$ of $\calM_{M_\veca}^J(\orb^\veca_{k})$ by stable pseudoholomorphic buildings \cite{BEHWZ}, which consist of 
\begin{itemize}
  \item a main $J$-holomorphic level in $\wh{M}_{\veca}$
  \item some number of $J_{\bdy E(\veca)}$-holomorphic levels in $\R \times \bdy M_\veca$,
\end{itemize}
subject to asymptotic matching conditions at the consecutive levels and a stability condition.
Usually one considers each symplectization level only modulo translations in the $\R$ factor.
Here we will instead consider symplectization levels only up to the aforementioned action by $\C^*$.
This compactification has the virtue that boundary strata in $\ovl{\calM}_{M_\veca}^J(\orb^\veca_{k})$ have (real) expected codimension at least two. 
Similar considerations will apply to the compactification $\ovl{\calM}_{E(\veca)}(\orb^{\veca}_{i_1},\dots,\orb^{\veca}_{i_k})$ and other symplectic manifolds with ellipsoidal ends. Note that this compactification is closely analogous to the the space of stable maps relative to a smooth global divisor which is used to define relative Gromov--Witten invariants, and it simplifies the analysis since we seek only a virtual fundamental class rather than a virtual fundamental chain. In particular, as we explain in the next subsection, in favorable situations we can define a fundamental class by elementary classical techniques. This will indeed be the case for our main computation in \S\ref{sec:descendants}.

\subsection{Thin compactifications}

Here we describe a situation in which fundamental classes of compactified moduli spaces can be defined by elementary point set topology.
Although there are many more sophisticated treatments in the literature, we give a simple criterion which is sufficient for our purposes.

Recall that the $k$th Borel--Moore homology group of a topological space $X$ is defined by
\begin{align*}
H^\op{BM}_k(X) := \varprojlim_{K \subset X}(X,X \setm K),
\end{align*}
where the inverse limit is over compact subsets $K \subset X$. In particular, we have $H^\op{BM}_k(X) = H_k(X)$ if $X$ is compact.
If $X$ is an oriented $n$-dimensional topological manifold there is a canonical Poincar\'e duality isomorphism $H_k^\op{BM}(X) \cong H^{n-k}(X)$.

Let $\ovl{X}$ be a Hausdorff topological space, and let $X \subset \ovl{X}$ be an open subset which is a connected oriented $n$-dimensional topological manifold.  
Put $\bdy \ovl{X} := \ovl{X} \setm X$.
Note that for $k \in \Z_{\geq 0}$ there is a natural map
$\mathfrak{j}: H_k(\ovl{X},\bdy \ovl{X}) \ra H_k^\op{BM}(X)$ which is induced by the composition 
\begin{align}\label{eq:BM_diag}
H_k(\ovl{X},\bdy \ovl{X}) \ra H_k(\ovl{X}, \ovl{X} \setm K) \ra H_k(X,X \setm K)  
\end{align}
for each compact $K \subset X$ (the second map in \eqref{eq:BM_diag} is the excision isomorphism).

\begin{definition}\label{def:thin_comp}
We will say that $\ovl{X}$ is a {\bf thin compactification} of $X$ if
\begin{itemize}
  \item $H_k(\bdy X) = 0$ for all $k \geq n-1$
  \item the map $\mathfrak{j}: H_n(\ovl{X},\bdy \ovl{X}) \ra H_n^\op{BM}(X)$ is an isomorphism.
\end{itemize}
\end{definition}

If $\ovl{X}$ is a thin compactification of $X$, we define the virtual fundamental class $[\ovl{\calM}] \in H_n(\ovl{X})$ via
the composed isomorphism 
\[
\begin{tikzcd}
H_n(\ovl{X}) & H_n(\ovl{X},\bdy \ovl{X}) & H^\op{BM}_n(X) & H^0(X) & \Z
\arrow["\cong",from=1-1,to=1-2] 
\arrow["\cong","\mathfrak{j}"',from=1-2,to=1-3] 
\arrow["\cong",from=1-3,to=1-4] 
\arrow["\cong",from=1-4,to=1-5] 
\arrow[swap,bend right=15,"\cong",from=1-1,to=1-5]
\end{tikzcd},
\]
where the first isomorphism is induced by the long exact sequence for the pair $(\ovl{X},\bdy \ovl{X})$.

The second condition in Definition~\ref{def:thin_comp} is well-known to hold in many situations, for instance it is enough for the inclusion map $\bdy \ovl{X} \subset \ovl{X}$ to be cellular.
In our applications, $X$ will always be a moduli space of pseudoholomorphic curves, either closed or punctured, and $\ovl{X}$ will be its compactification by either stable maps or pseudoholomorphic buildings.
In favorable cases $\bdy\ovl{X}$ will be a finite union of manifolds of dimension at most $n-2$, which guarantees the first condition in Definition~\ref{def:thin_comp}.

\subsection{Ellipsoidal superpotentials}

In general, we denote the (virtual, signed) count of points in an index zero compactified moduli space $\ovl{\calM}$ by $\num\ovl{\calM}$.
In the presence of punctures it is typically convenient to incorporate certain Reeb orbit multiplicities in the count, as these dictate the a priori number of ways to glue along common Reeb orbits.
We denote by $\numtil\ovl{\calM}$ the product of $\num \ovl{\calM}$ with the multiplicities of all Reeb orbits appearing as {\em negative} ends.
For example, we put $\numtil\ovl{\calM}_{M_\veca}^J(\ga) := \mult(\ga) \cdot \num\ovl{\calM}_{M_\veca}^J(\ga)$.
We sometimes suppress the almost complex structure from the notation if the choice of $J$ is implicit or immaterial.
\begin{definition}
For $M$ a closed symplectic manifold, $A \in H_2(M)$ a homology class, and $\veca \in \R_{>0}^n$, we put 
\begin{align*}
\Tcount_{M,A}^\veca := \num \ovl{\calM}_{M_\veca,A}(\orb^\veca_{c_1(A)-1})
\end{align*} 
Similarly, put
\begin{align*}
\wtTcount_{M,A}^\veca := \numtil \ovl{\calM}_{M_\veca,A}(\orb^\veca_{c_1(A)-1}) = \mult(\orb^\veca_{c_1(A)-1}) \cdot \Tcount_{M,A}^\veca.
\end{align*}
\end{definition}
\NI We view the count $\Tcount_{M,A}^\veca$ as more geometrically fundamental, while the count $\wtTcount_{M,A}^\veca$ is often more convenient for proving formulas.

\begin{rmk}
 In various situations of interest, one can show that for generic $J \in \calJ(M_\veca)$ the compactified moduli space $\ovl{\calM}_{M_\veca,A}^J(\orb_k^\veca)$ agrees with the uncomcompactified one $\calM_{M_\veca,A}^J(\orb_k^\veca)$ and consist only of regular and somewhere injective curves, and in such cases $\Tcount_{M,A}^\veca$ is simply naive defined signed count.
For example, this is the case for the moduli spaces relevant for Conjecture~\ref{conj:p_q_d}.
We refer the reader to \cite{CSEN} for a more detailed discussion.
\end{rmk}

\begin{rmk}\label{rmk:name_expl}
  The name is motivated by the analogue with the superpotential of a Lagrangian torus $L$ in a symplectic manifold $M^{2n}$, which encodes the count rigid pseudoholomorphic disks with boundary on $L$.
By surrounding $L$ by a small Weinstein neighborhood $U$ and stretching the neck, this can also be reformulated as the count of rigid planes in the symplectic completion of $M \setm U$.
Here $\ovl{U}$ is symplectomorphic to the unit disk cotangent bundle $D^*\mathbb{T}^n$ of the torus (for some choice of metric).
In our case we have no Lagrangian, but the ellipsoid $E(\veca)$ takes the place of $D^*\mathbb{T}^n$. 
We emphasize, however, the ellipsoidal superpotential depends crucially on the continuous parameter $\veca$.
\end{rmk}

\section{$\Li$ formalism}\label{sec:Li}

In this section we discuss $\Li$ structures, which will provide an effective formalism for deriving various enumerative formulas.
In \S\ref{subsec:Linf} we first give a brief self-contained discussion of (filtered) $\Li$ algebras and morphisms between them. Then in \S\ref{subsec:Linf_from_SFT} we discuss the some key relevant algebraic structures arising symplectic field theory, setting up the formalism for various
important notions such as ellipsoidal cobordism maps, Maurer--Cartan elements, and stationary descendant augmentations.

\subsection{$\Li$ algebras and all that}\label{subsec:Linf}

We begin by briefly recalling some relevant notions from $\Li$ algebra and setting our conventions.
Given a $\Z$-graded vector space $V$ over $\Q$, let $\odot^{k}V = 
\underbrace{V \odot \cdots \odot V}_k$ denote its $k$-fold supersymmetric tensor power.
This is a quotient space of the ordinary $k$-fold tensor power $\otimes^k V = \underbrace{V \otimes \cdots \otimes V}_k$ by the permtuation group $\Sigma_k$ acting with Koszul signs, e.g. we have $v_1\odot v_2 \odot v_3 = (-1)^{|v_2||v_3|}v_1 \odot v_3 \odot v_2$ and so on.
Let $\ovl{S}V = \bigoplus_{k=1}^\infty \odot^{k}V$ denote the (reduced) symmetric tensor coalgebra on $V$. 
The grading of an element $v_1 \odot \cdots \odot v_k \in \ovl{S}V$ is $\sum_{i=1}^k |v_i|$, where $|v_i|$ is the grading of $v_i$ as an element of the graded vector space $V$.

Recall that $\ovl{S}V$ has a natural coproduct, given by 
\begin{align*}
\Delta(v_1 \odot \cdots \odot v_k) = \sum_{i=1}^{k-1} \sum_{\sigma \in \shuf(i,k-i)} \di(\sigma)(v_{\sigma(1)} \odot \cdots \odot v_{\sigma(i)}) \otimes (v_{\sigma(i+1)} \odot \dots \odot v_{\sigma(k)}),
\end{align*}
where $\shuf(i,k-i)$ is the subset of permutations $\sigma \in \Sigma_k$ satisfying $\sigma(1) < \cdots < \sigma(i)$ and $\sigma(i+1) < \cdots < \sigma(k)$.
Here the Koszul-type sign 
$\di(\sigma) := \di(V,\sigma;v_1,\dots,v_k)$ is given by
\begin{align*}
\di(V,\sigma;v_1,...,v_k) = (-1)^{\{|v_i| |v_j|\;:\; 1 \leq i < j \leq k,\; \sigma(i) > \sigma(j)\}}.
\end{align*}

\begin{definition}
  An $\Li$ algebra over $\Q$ consists of:
\begin{itemize}
  \item a $\Z$-graded $\Q$-vector space $V$
  \item a degree $1$ linear map $\wh{\ell}: \ovl{S}V \ra \ovl{S}V$
\end{itemize}
such that:
\begin{itemize}
  \item $\wh{\ell}$ is a coderivation, i.e. it satisfies the coLeibniz rule
  \begin{align*}
  \Delta \circ \wh{\ell} = (\1 \otimes \wh{\ell}) \circ \Delta + (\wh{\ell} \otimes \1) \circ \Delta.
  \end{align*}
  \item $\wh{\ell} \circ \wh{\ell} = 0$.
\end{itemize}
\end{definition}

\begin{rmk}
Our grading convention here is slightly nonstandard, as it is typical to replace $V$ with its downward degree shift $sV$, in order to emphasize the skew-symmetric nature of the $\Li$ operations. Note, however, that the Koszul signs mean supersymmetric operations can behave either symmetrically or skew-symmetrically, depending on the grading parities of the elements in question.
\end{rmk}

Alternatively, the coderivation $\wh{\ell}: \ovl{S}V \ra \ovl{S}V$ is uniquely determined by the degree $1$ maps $\ell^1,\ell^2,\ell^3,\dots$, where $\ell^k: \odot^kV \ra V$ is given by restricting and projecting $\wh{\ell}$.
Indeed, a sequence of degree $1$ maps $\ell^k: \odot^k V \ra V$ for $k \in \Z_{\geq 1}$ extend uniquely to a degree $1$ coderivation $\wh{\ell}: \ovl{S}V \ra \ovl{S}V$ via the formula
\begin{align*}
 \wh{\ell}(v_1 \odot ... \odot v_k) = \sum_{i=1}^k\sum_{\sigma \in \shuf(i,k-i)}\di(\sigma,V;v_1,...,v_k)\ell^k(v_{\sigma(1)}\odot ... \odot v_{\sigma(i)}) \odot v_{\sigma(i+1)} \odot ... \odot v_{\sigma(k)}.
\end{align*}
The equation $\wh{\ell} \circ \wh{\ell} = 0$ is then equivalent to the $\Li$ algebra 
structure equations for $\ell = (\ell^1,\ell^2,\ell^3,\dots),$
the first few of which are
\begin{gather*}
\ell^1(\ell^1(x)) = 0,\\
\ell^1(\ell^2(x,y)) + \ell^2(\ell^1(x),y) + (-1)^{|x||y|}\ell^2(\ell^1(y),x) = 0,
\end{gather*}
etc.
\begin{rmk}
In the sequel, we will go back and forth between these two perspectives on $\Li$ algebras, with similar considerations applying for $\Li$ homomorphisms and so on.
\end{rmk}

\begin{definition}
   Given $\Li$ algebras $V,W$, an $\Li$ homomorphism $\Phi: V \ra W$ consists of:
\begin{itemize}
  \item a degree $0$ linear map $\wh{\Phi}: \ovl{S}V \ra \ovl{S}W$
\end{itemize}
such that:
\begin{itemize}
  \item $\wh{\Phi}$ is a coalgebra homomorphism
  \item $\wh{\Phi} \circ \wh{\ell}_V = \wh{\ell}_W \circ \wh{\Phi}$.
\end{itemize}
 \end{definition}

Similar to above, the coalgebra homomorphism $\wh{\Phi}: \ovl{S}V \ra \ovl{S}W$ is uniquely determined by the degree $0$ maps $\Phi^1,\Phi^2,\Phi^3,\dots$, where $\Phi^k: \odot^kV \ra W$ is given by restricting and projecting $\wh{\Phi}$, and we have the extension formula
  \begin{align*}
\wh{\Phi}(v_1\odot ... \odot v_k) 
&:= \sum_{\substack{s \geq 1\\ k_1,\dots,k_s \geq 1,\\k_1 + ... + k_s = k}}\sum_{\sigma \in \Sigma_k} \frac{\di(\sigma,V;v_1,...,v_k)}{s!k_1!...k_s!}(\Phi^{k_1} \odot ... \odot \Phi^{k_s})(v_{\sigma(1)}\odot ... \odot v_{\sigma(n)})
\\ &=\sum_{\substack{s \geq 1\\1 \leq k_1 \leq \cdots \leq k_s \\ k_1 + \cdots + k_s = k}}\sum_{\sigma \in \shufbar(k_1,\dots,k_s)} \di(\sigma,V;v_1,...,v_k)(\Phi^{k_1} \odot ... \odot \Phi^{k_s})(v_{\sigma(1)}\odot ... \odot v_{\sigma(k)}),
  \end{align*}
with $\shufbar(k_1,\dots,k_s)$ defined as follows.
\begin{notation}
   Let $\shuf(k_1,\dots,k_s)$ denote the subset of permutations $\sigma \in \Sigma_{k_1+\cdots+ k_s}$ satisfying 
   \begin{gather*}
  \sigma(1) < \cdots < \sigma(k_1)\\
  \sigma(k_1+1) < \cdots < \sigma(k_1+k_2)\\
   \cdots \\
   \sigma(k_1+\cdots+k_{s-1}+1) < \cdots < \sigma(k_1+\cdots+k_s),
   \end{gather*}
   and let 
   $\shufbar(k_1,\dots,k_s) \subset \shuf(k_1,\dots,k_s)$ denote the subset of permutations $\sigma$ satisfying 
   \begin{align*}
   (\sigma(1),\cdots,\sigma(k_1)) < (\sigma(k_1+1),\dots,\sigma(k_1+k_2)) < \cdots < (\sigma(k_1+\cdots+k_{s-1}+1),\cdots,\sigma(k_1+\cdots+k_s)),
   \end{align*}
   where we order tuples of the same length lexographically and we declare longer tuples to be greater than shorter tuples.
\end{notation}

The equation $\wh{\Phi} \circ \wh{\ell}_V = \wh{\ell}_W \circ \wh{\Phi}$ is equivalent to the $\Li$ homomorphism equations for $\Phi = (\Phi^1,\Phi^2,\Phi^3,\dots)$, the first few of which are
\begin{gather*}
 \Phi^1 (\ell^1_V(x)) = \ell^1_W(\Phi^1(x)),\\
 \Phi^2(\ell^1_V(x),y) + \Phi^2(x,\ell^1_V(y)) + \Phi^1(\ell^2_V(x,y)) = \ell^2_W(\Phi^1(x),\Phi^1(y)) + \ell^1_W(\Phi^2(x,y)),
\end{gather*}
etc.

We can also compose $\Li$ homomorphisms $\Phi: V \ra W$ and $\Psi: W \ra Q$ in a natural way, such that $\wh{\Psi \circ \Phi} = \wh{\Psi} \circ \wh{\Phi}$.
The identity $\Li$ homomorphism $\1: V \ra V$ by definition has first term $(\1)^1: V \ra V$ as the identity map and $(\1)^k: \odot^kV \ra V$ trivial for $k \geq 2$.

Given an $\Li$ algebra $(V,\wh{\ell})$, its {\bf bar complex} $\bar V$ is the chain complex $(\ovl{S}V,\wh{\ell})$.
We denote the homology of the bar complex by $H\bar(V)$.
By the above discussion, an $\Li$ homomorphism $\Phi: V \ra W$ induces a chain map $\wh{\Phi}: \bar V \ra \bar W$ and an induced linear map $[\wh{\Phi}]: H\bar V \ra H\bar W$.

For quantitative applications, we usually consider $\Li$ algebras $V$ which are additionally equipped with a filtration.
An (increasing) $\R$-filtration on a vector space $V$ is a collection of subspaces $\calF_{\leq r}(V) \subset V$, $r \in \R$, such that $V = \bigcup\limits_{r \in \R}V$ and $r < r'$ implies $\calF_{\leq r}(V) \subset \calF_{\leq r'}(V)$. 
Typically we will have $\calF_{\leq r}(V) = \{0\}$ for $r < 0$.
If $V$ and $W$ are $\R$-filtered vector spaces, we will say that a linear map $f: V \ra W$ is filration-preserving if $f(\calF_{\leq r}(V)) \subset \calF_{\leq r}(W)$ for all $r \in \R$.

An $\R$-filtration on $V$ induces one on $\ovl{S}V$ by declaring $v_1 \odot \cdots \odot v_k \in \calF_{\leq r}(\ovl{S}V)$ if and only if we have $v_1 \in \calF_{\leq r_1}(V),\dots, v_k \in \calF_{\leq r_k}(V)$ for some $r_1,\dots,r_k \in \R$ satisfying $r_1 + \cdots + r_k \leq r$.

\begin{definition}
  A {\bf filtered $\Li$ algebra} is an $\Li$ algebra $(V,\wh{\ell})$ together with an $\R$-filtration on $V$ such that $\wh{\ell}: \ovl{S}V \ra \ovl{S}V$ is filtration-preserving. Similarly, an $\Li$ homomorphism $\Phi: V \ra W$ between filtered $\Li$ algebras $V,W$ is {\bf filtered} if the induced map $\wh{\Phi}: \ovl{S}V \ra \ovl{S}W$ is filtration preserving.
\end{definition}
\NI We will often specify a filtration for an $\Li$ algebra by specifying a canonical basis $v_1,v_2,v_3,\dots \in V$ and corresponding ``action values'' $\calA(v_1),\calA(v_2),\calA(v_3),\dots \in \R$.
The filtration is then taken to be such that $\calF_{\leq r}(V)$ is the span of all basis elements having action at most $r$.

\subsection{$\Li$ structures from symplectic field theory}\label{subsec:Linf_from_SFT}
We now discuss some relevant holomorphic curve invariants arising from symplectic field theory, along with their main structural properties. 
General background on symplectic field theory can be found e.g. in \cite{EGH2000,wendl_SFT_notes,BEHWZ}.
We refer the reader to \cite{HSC} for a more detailed discussion of how the invariants of this paper fit within a larger SFT context (including a standard disclaimer about virtual perturbation schemes).

\subsubsection{Filtered $\Li$ algebra assoicated to an ellipsoid}
Following \cite{HSC}, for each Liouville domain $X$ we have an associated filtered $\Li$ algebra $\chlin(X)$ whose underlying vector space is spanned by the good\footnote{See e.g. \cite[Def. 11.6]{wendl_SFT_notes}.} Reeb orbits of $\bdy X$.

In the case of an ellipsoid $E(\veca)$, the $\Li$ operations vanish for degree reasons, and all Reeb orbits are good, so we can describe the invariant $C_\veca := \chlin(E(\veca))$ purely formally as follows:
\begin{definition}
  For $\veca = (a_1,\dots,a_n) \in \R_{> 0}^n$, let $C_\veca$ be the abelian filtered $\Li$ algebra over $\Q$ with generators $\orb_1^\veca,\orb_2^\veca,\orb_3^\veca,\dots$ such that 
  \begin{itemize}
    \item  $|\orb_k^\veca| = -2-2k$
    \item $\calA(\orb_k^\veca)$ is the $k$th smallest element in the multiset 
    $\{ja_i\;|\; i \in \{1,\dots,n\}, j \in \Z_{\geq 1}\}$.    
  \end{itemize}
\end{definition}
\NI In the case $n = 2$ we will use the shorthand $C_a := C_{(1,a)}$.
When discussing stationary descendants, it will also be useful to define an (unfilterd) $\Li$ algebra $C_o$, roughly viewed as $\chlin$ of the empty set, as follows: 
\begin{definition}
Let $C_o$ denote the abelian $\Li$ algebra over $\Q$ with generators $\dg_1,\dg_2,\dg_3,\dots$, with $|\dg_k| = -2-2k$.  
\end{definition}

\begin{rmk}
For $k \in \Z_{\geq 1}$ we have $\cz(\orb^\veca_k) = n-1+2k$, where $\cz$ denotes the Conley--Zehnder index as measured by a global trivialization of the contact distribution $\bdy E(\veca)$, 
so the above grading convention corresponds to $|\ga| = n-3-\cz(\ga)$ (see \cite{HSC}).
\end{rmk}

\subsubsection{$\Li$ cobordism maps between ellipsoids}\label{subsubsec:cob_maps_btw_ell}

Given a symplectic embedding $\iota: E(\veca) \hooksymp E(\veca')$, we get an 
associated filtered $\Li$ homomorphism
$\Xi^{\veca'}_{\veca}: C_{\veca'} \ra C_\veca$, defined roughly by counting genus zero asymptotically cylindrical curves with many positive ends and one negative end in the symplectic completion of the corresponding complementary cobordism.
More precisely, defining the compact symplectic cobordism $E^{\veca'}_{\veca} := E(\veca') \setm \iota(\intE(\veca))$,
then with respect to the bases $\{\orb_k^{\veca'}\}$,$\{\orb_k^\veca\}$ 
 the structure coefficients of $\Xi^{\veca'}_\veca$ are given by
\begin{align*}
\langle (\Xi^{\veca'}_{\veca})^k(\orb^{\veca'}_{i_1},\dots,\orb^{\veca'}_{i_k}), \orb^\veca_j\rangle
= \begin{cases}
  \numtil \ovl{\calM}_{E^{\veca'}_\veca}(\orb^{\veca'}_{i_1},\dots,\orb^{\veca'}_{i_k};\orb^\veca_j) & \text{if}\;\;\; j = \sum_{s=1}^k i_s + k-1 \\ 0 & \text{otherwise}.
\end{cases} 
\end{align*}

Note that a priori $\Xi^{\veca'}_\veca$ depends on the symplectic embedding $\iota$, but 
if we ignore filtrations then it only depends on $\veca,\veca'$.
Indeed, we can always find a symplectic embedding $\iota: E(\eps \veca) \hooksymp E(\veca')$ for some $\eps > 0$ sufficiently small, and moreover any two such embeddings are Hamiltonian isotopic after possibly shrinking $\eps$.
Applying a Hamiltonian isotopy to $\iota$ potentially modifies $\Xi^{\veca'}_{\veca}$ by an $\Li$ homotopy, but all such $\Li$ homotopies are trivial since $C_\veca,C_{\veca'}$ are abelian.
 
We will work throughout with a fixed choice of almost complex structure $J_{\bdy E(\veca)}$ on $\R \times \bdy E(\veca)$ (see Definition~\ref{def:J_E_etc}), and we also fix any other relevant perturbation data for $\bdy E(\veca)$.
In particular, we can assume that $\Xi^\veca_{\eps \veca}: C_{\veca} \ra C_{\eps \veca}$ is the identity $\Li$ homomorphism, after identifying $\orb_{k}^\veca$ with $\orb_{k}^{\eps \veca}$ for all $k \in \Z_{\geq 1}$. 
 This means that replacing $\veca$ with $\eps \veca$ is immaterial for the purposes of computing the map $\Xi^{\veca'}_{\veca}$, which justifies:
\begin{notation}
  For $\veca,\veca' \in \R_{>0}^n$, put $E^{\veca'}_\veca := E(\veca') \setm \iota(\intE(\eps\veca))$, where $\iota: E(\eps \veca) \hooksymp E(\veca')$ is a symplectic embedding for some $\eps > 0$. 
\end{notation}
\NI Let us emphasize, however, that the existence of a {\em filtered} $\Li$ homomorphism $\Xi^{\veca'}_{\veca}: C_{\veca'} \ra C_{\veca}$ depends on the existence of a symplectic embedding $E(\veca) \hooksymp E(\veca')$ (not just $E(\eps \veca) \hooksymp E(\veca')$).

\subsubsection{Closed manifolds with ellipsoidal negative ends and their Maurer--Cartan elements}
\label{subsubsec:MC_elts}

Let $M^{2n}$ be a closed symplectic manifold, and consider the ellipsoidal complement $M_\veca$ for soime $\veca \in \R_{>0}^n$.
Counting rigid planes in $\wh{M}_\veca$ in a fixed homology class $A \in H_2(M)$
defines a distinguished element $\mc_{M,A}^\veca \in C_\veca$, whose coefficients with respect to the basis $\{\orb^\veca_k\}$ are given by
\begin{align*}
\langle \mc_{M,A}^\veca, \orb_k^\veca\rangle = 
\begin{cases}
\numtil \ovl{\calM}_{M_\veca,A}(\orb_k^\veca) & \text{if}\;\;\; k = c_1(A)-1\\
0 & \text{otherwise}.
\end{cases}
\end{align*}
Note that in terms of the ellipsoidal superpotential we have $\mc_{M,A}^\veca = \wtTcount_{M,A}^\veca \,\orb^\veca_{c_1(A)-1}$.

  Morally, the infinite sum $\sum\limits_{A \in H_2(M)} \mc_{M,A}^\veca$ is a Maurer--Cartan element in $C_\veca$, called the Cieliebak--Latschev element in \cite{HSC,chscI}. In fact, the Maurer--Cartan equation $\sum_{k=1}^\infty \tfrac{1}{k!}\ell^k(\mc^{\odot k}) = 0$ 
  in $C_\veca$ is trivial since $C_\veca$ is abelian, but this infinite sum is a priori ill-defined without working over a suitable Novikov completion. We will avoid this issue by simply working in one homology class at a time.

We also define a bar complex cycle $\exp_A(\mc^\veca_M) \in \bar C_\veca$ as follows:
\begin{align*}
\exp_A(\mc_M^\veca) := \sum_{\substack{k \geq 1\\A_1,\dots,A_k \in H_2(M)\\A_1 + \cdots + A_k = A}} \tfrac{1}{k!}\mc^\veca_{M,A_1} \odot \cdots \odot \mc^\veca_{M,A_k}.
\end{align*}
Note that $\wtTcount_{M,A}^\veca$ can only be nonzero if $c_1(A) \geq 2$, and hence $k$ has an a priori upper bound of $c_1(A)/2$.
It is then a consequence of the SFT compactness theorem that the sum on the right hand side is in fact finite, since for any given $\ell \in \Z_{\geq 2}$ and $J \in \calJ(M_\veca)$ there can only be finitely many homology classes $B \in H_2(M)$ with $c_1(B) = \ell$ for which $\ovl{\calM}_{M,B}^J(\orb^\veca_{\ell-1})$ is nonempty.

\subsubsection{Augmentation maps from stationary descendants}\label{subsubsec:aug_from_des}

We first briefly recall the construction of gravitational descendant Gromov--Witten invariants, restricting to the case of a single marked point carrying a point constraint.
Let $M^{2n}$ be a closed symplectic manifold and let $A \in H_2(M)$ be a second homology class.
For $J \in \calJ(M)$, we denote by $\calM_{M,A,1}^J$ the moduli space of $J$-holomorphic spheres $u: S^2 \ra M$ with one marked point. 
We have a natural evaluation map $\ev: \calM_{M,A,1}^J \ra M$, and this extends to the compactification $\ovl{\calM}_{M,A,1}^J$ by stable maps.
There is a natural complex line bundle over $\calM_{M,A,1}^J$ whose fiber over a given curve is the cotangent space to the domain at the marked point.
This bundle can be showed to extend naturally to $\ovl{\calM}_{M,A,1}^J$,
and we denote the first Chern class of this extension by $\psi \in H^2(\ovl{\calM}^J_{M,A,1})$.

The stationary descendant Gromov--Witten invariant $N_{M,A}\lll \psi^k \pt\rrr \in \Q$ is given by 
\begin{align*}
N_{M,A}\lll \psi^k \pt\rrr :=  \int\limits_{[\ovl{\calM}^J_{M,A,1}]}\ev^*(\pd(\pt)) \cup (\underbrace{\psi \cup \cdots \cup \psi}_k),
\end{align*}
where $\pd(\pt) \in H^{2n}(M)$ is Poincar\'e dual of the point class.
Note that $N_{M,A}\lll \psi^k \pt \rrr$ is nonzero only if $k = c_1(A) - 2$.

We can similarly define stationary descendant counts for rational punctured curves in symplectic ellipsoids; related counts have been considered e.g. in \cite{Fabert_descendants_2011,fabert2011string,tonk}.
Fix a collection of Reeb orbits $\orb^\veca_{i_1},\dots,\orb^\veca_{i_k}$ in $\bdy E(\veca)$, and let $J \in \calJ(E(\veca))$ be a generic almost complex structure which agrees with our fixed choice $J_{\bdy E(\veca)}$ on the cylindrical end.
Let $\ovl{\calM}_{E(\veca),1}^J(\orb^\veca_{i_1},\dots,\orb^{\veca}_{i_k})$ denote the compactified moduli space 
of genus zero asymptotically cylindrical curves in $\wh{E}(\veca)$ with positive asymptotics $\orb^\veca_{i_1},\dots,\orb^\veca_{i_k}$ and one additional marked point.
The standard SFT compatification in general produces codimension one boundary strata, and hence some care is needed to define descendants since we expect only virtual fundamental chains rather than virtual fundamental classes.
However, we circumvent this issue for ellipsoids by working with the modified SFT compactification discussed in \S\ref{subsec:moduli_spaces}, in which symplectization levels are only taken modulo $\C^*$, so that all boundary strata of $\ovl{\calM}_{E(\veca),1}^J(\orb^\veca_{i_1},\dots,\orb^{\veca}_{i_k})$ have expected (real) codimension at least two.
By analogy with the closed curve case, we then define
\begin{align*}
N_{E(\veca)}(\orb^\veca_{i_1},\dots,\orb^\veca_{i_k})\lll \psi^m \pt\rrr :=  \int\limits_{[\ovl{\calM}^J_{E(\veca),1}(\orb^\veca_{i_1},\dots,\orb^{\veca}_{i_k})]}\ev^*(\pd(\pt)) \cup (\underbrace{\psi \cup \cdots \cup \psi}_m).
\end{align*}

We will package these counts into an $\Li$ augmentation\footnote{Here use the term ``$\Li$ augmentation'' somewhat informally as a shorthand for an $\Li$ homomorphism with target $C_o$.} $\aug_\veca: C_\veca \ra C_o$ whose structure coefficients are given by:
\begin{align*}
\langle \aug_\veca^k(\orb^\veca_{i_1},\dots,\orb^{\veca}_{i_k}),\dg_j\rangle = 
\begin{cases}
  N_{E(\veca)}(\orb^\veca_{i_1},\dots,\orb^\veca_{i_k})\lll \psi^{j-1} \pt \rrr & \text{if}\;\;\; j = i_1 + \cdots + i_k + k-1\\
  0 & \text{otherwise}.
\end{cases}
\end{align*}

\subsubsection{Compatibilities between structures}\label{subsubsec:compatibilities}

For any $\veca,\veca',\veca'' \in \R_{> 0}^n$, the compact symplectic cobordism $E^{\veca''}_{\veca}$ is given (up to rescaling) by concatenating $E^{\veca'}_\veca$ and $E^{\veca''}_{\veca'}$ along their common contact boundary.
This gives rise to the compatibility relation:
\begin{align}\label{eq:Xi_Xi_compat}
\Xi^{\veca'}_{\veca} \circ \Xi^{\veca''}_{\veca'} = \Xi^{\veca''}_{\veca}.
\end{align}
As before, a priori one only expects the two sides to be $\Li$ homotopic, but this is the same as equality since all relevant $\Li$ algebras are abelian.

For a closed symplectic manifold $M^{2n}$,
neck stretching along the embedded contact hypersurface $\bdy E(\eps \veca)$ induces various compatibilities between the structures discussed above.
For a homology class $A \in H_2(M)$, the corresponding Maurer--Cartan elements $\exp_A(\mc^\veca_M) \in \bar C_\veca$ and $\exp_A(\mc_M^{\veca'}) \in \bar C_{\veca'}$ satisfy:
\begin{align}\label{eq:Xi_mc_compatibility}
\wh{\Xi}^{\veca'}_{\veca}(\exp_A(\mc_M^{\veca'})) = \exp_A(\mc_M^{\veca}).
\end{align}
Similarly, for the stationary descendant augmentations, we have:
\begin{align}\label{eq:aug_Xi_compatibility}
\aug_{\veca} \circ \Xi^{\veca'}_\veca = \aug_{\veca'}
\end{align}
as $\Li$ homomorphisms $C_{\veca'} \ra C_o = \Q\langle \dg_1,\dg_2,\dg_3,\dots\rangle$.

Finally, the compatibility between the Maurer--Cartan elements and descendant augmentations is as follows.
Define
\begin{align*}
\mc^o_{M,A} := N_{M,A}\lll \psi^{k-1} \pt\rrr \dg_k \in C_o,
\end{align*}
where $k = c_1(A) - 1$,
and put
\begin{align*}
\exp_A(\mc^o_M) := \sum_{\substack{k \geq 1\\A_1,\dots,A_k \in H_2(M)\\A_1 + \cdots + A_k = A}} \tfrac{1}{k!}\mc^o_{M,A_1} \odot \cdots \odot \mc^o_{M,A_k}.
\end{align*}
Then we have
\begin{align}\label{eq:aug_hat_exp_m}
\wh{\aug}_\veca(\exp_A(\mc_M^\veca)) = \exp_A(\mc_M^o)
\end{align}
as elements of $\bar C_o = \Q[\dg_1,\dg_2,\dg_3,\dots]$.
Informally, this says that the ellipsoidal superpotential together with stationary descendants in the ellipsoid give closed curve descendants.

\section{SFT admissibility of the standard complex structure}\label{sec:J_std}

In this section we set up a framework to explicitly analyze asymptotically cylindrical curves in spaces such as $\R \times \bdy E(\veca)$ and $\wh{E}(\veca)$.
More precisely, we show that there exists an SFT admissible almost complex structure $J_{\bdy E(\veca)}$ on $\R \times \bdy E(\veca)$ and a diffeomorphism $\R \times \bdy E(\veca) \ra \C^n \setm \{\vec{0}\}$ which pushes forward $J_{\bdy E(\veca)}$ to the standard complex structure $i$. 
Similarly, there is an SFT admissible almost complex structure $J_{E(\veca)}$ on $\wh{E}(\veca)$ and a diffeomorphism $\wh{E}(\veca) \ra \C^n$ which pushes forward $J_{E(\veca)}$ to $i$. 
The upshot is that SFT computations in these spaces can be reduced to questions about proper $i$-holomorphic curves in $\C^n$ or $\C^n \setm \{\vec{0}\}$, which in turn can be described quite explicitly (we leverage this in \S\ref{sec:descendants}).

The main technical work is to construct the aforementioned diffeomorphisms.
We begin in \S\ref{subsec:strictly_iconv} by describing the basic setup and collecting some facts about $i$-convex hypersurfaces. We then construct the diffeomorphism $\R \times \bdy E(\veca) \ra \C^n \setm \{\vec{0}\}$ in \S\ref{prop:Phi_a_exists}.
The result is nonobvious but given by explicit formulas.
We then modifiy this construction in \S\ref{subsec:past_symp} in order to construct the diffeomorphism $\wh{E}(\veca) \ra \C^n$ (there is also an analogous construction for $\wh{M}_\veca$).
The basic idea here is to modify the previous diffeomorphism for $\R \times \bdy E(\veca)$ by a slow cutoff function in such a way that tameness is preserved.
The argument here is rather involved but it uses essentially only elementary techniques from differential geometry, and since this appears to be a new fundamental result about ellipsoids we include a careful proof.
Finally, in \S\ref{subsec:asymp_corr} we explain in this framework how to read off the asymptotic Reeb orbits of a curve in $\R \times \bdy E(\veca)$ or $\wh{E}(\veca)$.

\subsection{Symplectizations of strictly $i$-convex hypersurfaces}\label{subsec:strictly_iconv}

Let $\Sigma \subset \C^n$ be a smooth hypersurface which bounds a compact domain, and let $Z$ be a vector field defined near $\Sigma$, such that $iZ$ is tangent to $\Sigma$.
Define the field of complex tangencies $\xi_\C \subset T\Sigma$ by $\xi_\C := T\Sigma \cap iT\Sigma$.
These data determine a one-form $\beta \in \Omega^1(\Sigma)$ satisfying $\ker \beta = \xi_\C$ and the normalization condition $\beta(iZ|_\Sigma) = 1$. 
Following \cite[\S2.3]{cieliebak2012stein}, the hypersurface $\Sigma$ is {\bf strictly $i$-convex} if we have $d\beta(v,iv) > 0$ for all nonzero $v \in \xi_\C$. In this case $\beta$ is a contact form, and in particular there is an induced Reeb vector field $R_\beta$, characterized by $d\beta(R_\beta,-) = 0$ and $\beta(R_\beta) = 1$.
\begin{rmk}\label{rmk:conv_implies_i_conv}
If $\Sigma$ is strictly geometrically convex, then it is strictly $i$-convex (see \cite[\S2.5]{cieliebak2012stein}).
\end{rmk}

\begin{lemma}\label{lem:iZ_is_Reeb}
  If $Z$ is holomorphic, then we have $R_\beta = iZ$.
\end{lemma}
\begin{proof}
Since $iZ$ is also holomorphic, it induces a flow on $\Sigma$ which preserves $\xi_\C$, i.e. $iZ|_{\Sigma}$ is a contact vector field for $(\Sigma,\xi_\C)$. Since $R_\beta$ is also a contact vector field for $(\Sigma,\xi_\C)$ and we have $\beta(iZ) \equiv \beta(R_\beta)$, we must have $iZ = R_\beta$.
\end{proof}

Let $\alpha := \la_\std|_{\Sigma}$ denote the standard contact form on $\Sigma$, where ${\la_\std = \tfrac{1}{2}\sum_{i=1}^n (x_idy_i - y_idx_i)}$. We denote the corresponding Reeb vector field on $\Sigma$ by $R_\alpha$.

\begin{lemma}\label{lem:convert_i}
Assume that $\Sigma$ is strictly $i$-convex, and suppose that we have the following:
\begin{itemize}
  \item a complete nonvanishing holomorphic vector field $Z$ on $\C^n \setm \{\vec{0}\}$ which is outwardly transverse to $\Sigma$, such that all orbits intersect $\Sigma$
  \item a diffeomorphism $G: \Sigma \ra \Sigma$ such that $G^*\alpha = \beta$.
\end{itemize}
Then there is some $J \in \calJ(\Sigma,\alpha)$\footnote{Recall that $\calJ(\Sigma,\alpha)$ denotes the space of SFT admissible almost complex structures on the symplectization $\R \times \bdy \Sigma$, with respect to the contact form $\alpha$. In particular, these are invariant under translations in the $\R$ factor.} and a diffeomorphism $\Phi: \R \times \Sigma \ra \C^n \setm \{\vec{0}\}$ which pushes forward $J$ to $i$.
\end{lemma}
\NI In particular, if the above hypotheses hold then $J$-holomorphic asymptotically cylindrical curves in the symplectization $\R \times \bdy \Sigma$ are images of $i$-holomorphic curves in $\C^n \setm \{\vec{0}\}$ under the map $\Phi^{-1}$.

\begin{proof}[Proof of Lemma~\ref{lem:convert_i}]
Define the map $\Phi: \R \times \Sigma \ra \C^n \setm \{\vec{0}\}$ by $\Phi(r,p) = \fl_Z^r(G^{-1}(p))$, where $\fl_Z^t$ denotes the time $t$ flow of $Z$.
Put $J := \Phi^*i$.
Since $\Sigma$ is strictly $i$-convex, we have $d\beta(v,iv) > 0$ for nonzero $v \in \xi_\C$. 
Moreover, since $i$ is integrable, we have $d\beta(iv,iw) = d\beta(v,w)$ for all $v,w \in \xi_\C$ (see \cite[Lem. 2.6]{cieliebak2012stein}).
This means that $i$ is $d\beta$-compatible on $\xi_\C$, and hence $J$ is $d\alpha$-compatible on $\ker\alpha$.
Since $\Phi^*Z = \bdy_r$, $\Phi|_{\Sigma}^*R_\beta = G_*R_\beta = R_\alpha$, and $iZ = R_\beta$ by Lemma~\ref{lem:iZ_is_Reeb}, along $\Sigma$ we have $J\bdy_r = R_\alpha$.
Finally, since $i$ is invariant under the flow of $Z$, $J$ is invariant under the flow of $\bdy_r$, i.e. it is invariant under translations in the $\R$ factor.
\end{proof}

\subsection{Symplectizations of ellipsoids}

We now specialize the above discussion to the case of ellipsoids. Fix $\veca = (a_1,\dots,a_n) \in \R_{>0}^n$.

\begin{prop}\label{prop:Phi_a_exists}
There is a diffeomorphism $\Phi_\veca: \R \times \bdy E(\veca) \ra \C^n\setm \{\vec{0}\}$ such that $\Phi_\veca^*i \in \calJ(\bdy E(\veca))$.
\end{prop}
\begin{proof}
 The following lemma provides choices of $Z$ and $G$ such that the hypotheses of Lemma~\ref{lem:convert_i} are satisfied. 
\end{proof}

\begin{lemma}
  For $\Sigma = \bdy E(\veca)$ and $Z = Z_\veca := 2\pi\sum_{i=1}^n \tfrac{1}{a_i}(x_i\bdy_{x_i} + y_i \bdy_{y_i})$, there is a diffeomorphism $G: \bdy E(\veca) \ra \bdy E(\veca)$ such that $G^* \alpha = \beta$.
\end{lemma}

\NI We will also sometimes write $G_\veca, \alpha_\veca,\beta_\veca$ and so on when we wish to emphasize the dependence on the ellipsoid parameters $\veca$.

\begin{proof}
  
As a preliminary step, consider the one-form 
\begin{align*}
\delta := \sum_{i=1}^n \tfrac{1}{a_i}(x_idy_i - y_idx_i) \in \Omega^1(\bdy E(\veca)).
\end{align*}
\begin{claim}
 We have $\ker \delta = \xi_\C$. 
\end{claim}
\begin{proof}
Recall that we have $\bdy E(\veca) = H_\veca^{-1}(1)$, where $H_\veca: \C^n \ra \R$ is the function $H_\veca(z_1,\dots,z_n) = \sum_{i=1}^n \pi |z_i|^2/a_i$.
  Using the standard metric on $\C^n$, for $p \in \Sigma$, $\ker \delta_p$ is the orthogonal complement of the two-dimensional subspace 
\begin{align*}
\left\langle \sum_{i=1}^n \tfrac{1}{a_i}(x_i\bdy_{x_i} + y_i\bdy_{y_i}), \sum_{i=1}^n \tfrac{1}{a_i} (x_i\bdy_{y_i} - y_i\bdy_{x_i}) \right\rangle \subset T_p\C^n.
\end{align*}
Since this is a complex subspace, so is its orthogonal complement.  
\end{proof}

We therefore have
\begin{align*}
\beta = \frac{\delta}{\delta(iZ)} = \frac{\sum_{i=1}^n \tfrac{1}{a_i}(x_idy_i - y_idx_i)}{2\pi\sum_{i=1}^n \frac{1}{a_i^2}(x_i^2+y_i^2)}.
\end{align*}
In polar coordinates, we have
\begin{align*}
\alpha = \tfrac{1}{2}\sum_{i=1}^n r_i^2d\theta_i,\;\;\;\;\;
\beta = \frac{\sum_{i=1}^n \tfrac{1}{a_i} r_i^2d\theta_i}{2\pi\sum_{i=1}^n \tfrac{1}{a_i^2}r_i^2}.
\end{align*}
  Define a map $G = G_\veca: \Sigma \ra \Sigma$ by
\begin{align*}
G(r_1,\dots,r_n,\theta_1,\dots,\theta_n) = (\wt{r}_1(r_1,\dots,r_n),\dots,\wt{r}_n(r_1,\dots,r_n),\theta_1,\dots,\theta_n),
\end{align*}
where
\begin{align*}
\wt{r}_i(r_1,\dots,r_n) = 
\frac{r_i}{\sqrt{\pi a_i \sum_{j=1}^n \tfrac{1}{a_j^2} r_j^2}}.
\end{align*}
Observe that $G$ does indeed map $\Sigma$ to itself: 
\begin{align*}
\pi \sum_{i=1}^n \wt{r}_i^2/a_i = \frac{\sum_{i=1}^n r_i^2/a_i^2}{\sum_{j=1}^n r_j^2/a_j^2} = 1,
\end{align*}
and moreover we have $G^*\alpha = \beta$.
\end{proof}

The inverse map $G_\veca^{-1}: \Sigma \ra \Sigma$ is given explicitly by
\begin{align*}
G_\veca^{-1}(\wt{r}_1,\dots,\wt{r}_n,\theta_1,\dots,\theta_n) = (r_1(\wt{r}_1,\dots,\wt{r}_n),\dots,r_n(\wt{r}_1,\dots,\wt{r}_n),\theta_1,\dots,\theta_n),   
\end{align*}
where
\begin{align*}
r_i = \frac{\wt{r}_i\sqrt{a_i}}{\sqrt{\pi\sum_{i=1}^n \wt{r}_i^2}}.
\end{align*}
We also have
\begin{align*}
\fl_{Z_\veca}^t(\vecz) = (e^{2\pi t/a_1}z_1,\dots,e^{2\pi t/a_n}z_n)
\end{align*}
for $\vecz = (z_1,\dots,z_n) \in \C^n$.
Therefore, we have the following formula for $\Phi_\veca$:
\begin{align*}
\Phi_\veca(r,\vecz) = \fl^r_{Z_\veca}(G_\veca^{-1}(\vecz)) = \frac{(e^{2\pi r/a_1}\sqrt{a_1}z_1,\dots,e^{2\pi r/a_n}\sqrt{a_n}z_n)}{\sqrt{\pi \sum_{i=1}^n |z_i|^2}},
\end{align*}
for $r \in \R$ and $\vecz \in \bdy E(\veca)$.

\begin{rmk}
Note that $G_\veca: \bdy E(c,\dots,c) \ra \bdy E(c,\dots,c)$ is the identity map for any $c \in \R_{> 0}$.
Also, for $\veca_0 := (\underbrace{4\pi,\dots,4\pi}_n)$, we have $Z_{\veca_0} = Z_\std$, where $Z_\std = \tfrac{1}{2}\sum_{i=1}^n (x_i\bdy_{x_i} + y_i\bdy_{y_i})$ is the Liouville vector field associated to the standard Liouville one-form $\la_\std$.
\end{rmk}

\subsection{Extending past the symplectization}\label{subsec:past_symp}

We now explain how to adapt the preceding discussion to symplectic manifolds of the form $\wh{E}(\veca)$ and $\wh{M}_\veca$.
The idea here is to modify the diffeomorphism $\Phi_\veca$ by a careful cutoff function.

We first introduce some relevant notation.

\begin{definition}\label{def:B_veca_C_veca}
For $R \in \R_{> 0}$ sufficiently large, put $B^{2n}_\veca(R) := B^{2n}(R) \setm \intE(\veca)$, and let $\wh{B}_\veca^{2n}(R) = B_\veca^{2n}(R) \cup (\R_{\leq 0} \times \bdy E(\veca))$
denote its symplectic completion {\em at the negative end}.

Similarly, put $\C^n_\veca := \C^n \setm \intE(\veca)$, and let
$\wh{\C}^n_\veca = \C^n_\veca \cup (\R_{\leq 0} \times \bdy E(\veca))$ denote its symplectic completion at the negative end.
\end{definition}

\NI Let $\calJ(B_\veca^{2n}(R))$ denote the corresponding space of SFT admissible almost complex structures on $\wh{B}_\veca^{2n}(R)$, and let $\calJ(\C^n_\veca)$ denote the space of SFT admissible almost complex structures on $\wh{\C}^n_\veca$.
We define the tame analogues $\calJ_\tame(B_\veca^{2n}(R))$ and $\calJ_\tame(\C_\veca^{n})$ similarly.

\begin{rmk}
Note that the notation in Definition~\ref{def:B_veca_C_veca} is slightly nonstandard since we complete only at the negative end, but there should be minimal risk of confusion.
Observe that, using the flow of $Z_\std$, there are natural identifications 
\begin{align*}
\R \times \bdy E(\veca) \approx \wh{E}(\veca) \setm \{\vec{0}\}
\;\;\;\;\;\text{and}\;\;\;\;\;
\R \times \bdy E(\veca) \approx \wh{\C}_\veca^n
\end{align*}
which respect the associated Liouville one-forms.
However, our conventions are such that almost complex structures in $\calJ(\bdy E(\veca))$ must be globally translation invariant, whereas almost complex structures in $\calJ(E(\veca))$ (resp. $\calJ(\C^n_\veca$)) are only required to be translation invariant on the positive (resp. negative) end.
\end{rmk}

Now let $F_\veca: \R \times \bdy E(\veca) \ra \C^n \setm \{\vec{0}\}$ denote the diffeomorphism given by
\begin{align*}
F_\veca(r,\vecz) = \fl_{Z_\std}^r(\vecz) = e^{r/2} \cdot \vecz,
\end{align*}
for $r \in \R$ and $\vecz \in \bdy E(\veca)$.
Note that we have $F_\veca^* \la_\std = e^r \alpha_\veca \in \Omega^1(\R \times \bdy E(\veca))$.
For $p \in \C^n \setm \{\vec{0}\}$, let $\rr_\veca(p) \in \R$ be the unique real number such that $\fl_{Z_\std}^{-\rr(p)}$ lies in $\bdy E(\veca)$, i.e. such that 
$H_\veca(\fl_{Z_\std}^{-\rr_\veca(p)}) = 1$.
Explicitly, we have $\rr_\veca(\vecz) = \ln H_\veca(\vecz)$. 
Then $F_\veca^{-1}: \C^n \setm \{\vec{0}\} \ra \R \times \bdy E(\veca)$ can be written as 
\begin{align*}
F_\veca^{-1}(\vecz) = (\rr_\veca(\vecz),\fl_{Z_\std}^{-\rr_\veca(\vecz)}(\vecz)) = (\ln H_\veca(\vecz),\vecz/\sqrt{H_\veca(\vecz)}).
\end{align*}

\begin{prop}\label{prop:Q_exists} \hfill

\begin{enumerate}
  \item 
There is a diffeomorphism $Q_{\veca}^+: \wh{E}(\veca) \ra \C^n$ such that 
\begin{itemize}
  \item $(Q_{\veca}^+)^*i \in \calJ_\tame(E(\veca))$ 
  \item the composition 
\begin{tikzcd}
  \C^n \setm \{0\} & \R \times \bdy E(\veca) \subset \wh{E}(\veca) & \C^n
  \arrow["F_\veca^{-1}", from=1-1, to=1-2]
  \arrow["Q_\veca^+", from=1-2, to=1-3]
\end{tikzcd}
 is the identity near $\vec{0}$.
\end{itemize}

 \item
  There is a diffeomorphism $Q^-_\veca: \wh{\C}^n_\veca \ra \C^n \setm \{\vec{0}\}$ such that
\begin{itemize}
\item $(Q^-_\veca)^* i \in \calJ_\tame(\C^n_\veca)$
\item the composition 
\begin{tikzcd}
  \C^n \setm \{0\} & \R \times \bdy E(\veca) \approx \wh{\C}^n_\veca & \C^n\setm \{0\}
  \arrow["F_\veca^{-1}", from=1-1, to=1-2]
  \arrow["Q_\veca^-", from=1-2, to=1-3]
\end{tikzcd}
 is the identity on $\C^n \setm B^{2n}(R)$, for some $R \in \R_{> 0}$ sufficiently large.

\end{itemize} 
\end{enumerate}

\end{prop}

\begin{cor}\label{cor:M_a_corresp}
Let $M$ be a closed symplectic manifold, let $J \in \calJ(M)$ be integrable near a point $p_0 \in M$,
and assume there are complex coordinates near $p_0$ in which the symplectic form is ${\omega_\std = \sum_{i=1}^n dx_i \wedge dy_i}$.
Then for all $\veca \in \R_{>0}^n$ there is a diffeomorphism $Q: \wh{M}_\veca \ra M \setm \{p_0\}$ such that $Q^*J \in \calJ(M_\veca)$.
\end{cor}

\begin{proof}
Let $M'$ denote the symplectic manifold $M$, but with symplectic form $\omega_{M'} = C \cdot \omega_M$, where $\omega_M$ denotes the symplectic form on $M$.
By our assumptions, for $C > 0$ sufficiently large we can find a biholomorphic embedding $\zeta: B^{2n}(R) \hookrightarrow M'$ such that $\zeta(0) = p_0$ and $\zeta^* \omega_{M'} = \omega_\std$.
Let $\iota_\veca: E(\veca) \hooksymp M'$ denote the symplectic embedding given by restricting $\zeta$ to the ellipsoid $E(\veca)$.
Note that the same map can be viewed as a symplectic embedding
$\iota_{\delta \veca}: E(\delta \veca) \hooksymp M$, for $\delta := \tfrac{1}{C}$.
Put $M'_\veca = M' \setm \iota_\veca(\intE(\veca))$, and let $\wh{M}'_\veca = M'_\veca \cup (\R_{\geq 0} \times \bdy E(\veca))$ denote its symplectic completion (at the negative end).
Similarly, put $M_\veca = M \setm \iota_{\delta\veca}(\intE(\delta\veca))$, and let $\wh{M}_\veca := M_\veca \cup (\R_{\geq 0} \times \bdy E(\delta\veca))$ denote its symplectic completion.

By Proposition~\ref{prop:Q_exists}(2), there is a diffeomorphism $Q_\veca^-: \wh{B}_\veca^{2n}(R) \ra B^{2n}(R) \setm \{p_0\}$ satisfying $(Q_\veca^-)^*i \in \calJ(B^{2n}_\veca(R))$ which is the identity near the boundary.
We define a diffeomorphism $Q: \wh{M}'_\veca \ra M \setm \{p\}$ which on the image of $\zeta$ is given by $\zeta \circ Q_\veca^- \circ \zeta^{-1}$, and we extend by the identity to the rest of $\wh{M}'_\veca$. 
Then we have $Q^*J \in \calJ(M'_\veca)$.

Finally, we can equally view $Q^*J$ as an element of $\calJ(M_\veca)$, except that on the negative end it satisfies $\bdy_r \mapsto \delta R_\alpha$ rather than $\bdy_r \mapsto R_\alpha$.
This discrepancy is rectified by further pulling back $Q^*J$ by a diffeomorphism $\wh{M}_\veca \ra \wh{M}_\veca$ which takes the form $r \mapsto \delta r$ on the negative on $\R_{\leq 0} \times \bdy E(\delta\veca)$ and is the identity on $M_\veca$ minus a small interpolation region.

\end{proof}

\begin{proof}[Proof of Proposition~\ref{prop:Q_exists}]

We will prove (1), the proof of (2) being nearly identical.
  Put 
\begin{align*}
\wt{\Phi}_\veca := \Phi_\veca \circ F_\veca^{-1}: \C^n \setm \{0\} \ra \C^n \setm \{0\}.
\end{align*}
Not that this is the identity map when $\veca = \veca_0 = (4\pi,\dots,4\pi)$.
For $\vecz = (z_1,\dots,z_n) \in \C^n \setm \{\vec{0}\}$, we have explicitly
\begin{align*}
\wt{\Phi}_\veca(\vecz) = \Phi_\veca(\ln H_\veca(\vecz), \vecz/\sqrt{H_\veca(\vecz)}) &= \Phi_\veca(\ln H_\veca(\vecz), \vecz) \\ &= \frac{(H_\veca(\vecz)^{2\pi/a_1}\sqrt{a_1} z_1,\dots,H_\veca(\vecz)^{2\pi/a_n}\sqrt{a_n} z_n)}{\sqrt{\pi} ||\vecz||}.
\end{align*}

Fix $\veca$, and put $\rr(\vecz) := \rr_\veca(\vecz) = \ln H_\veca(\vecz)$.   
Let $\vfa: \R \ra \R^n_{> 0}$ be a smooth function such that $\vfa(s) = \veca$ for all $s \geq 0$ and 
$\vfa(s) = \veca_0$ for all $s \leq -R$, where $R > 0$ will be chosen to be sufficiently large.
We will assume $\vfa(s)$ is of the form $\vfa = (1-\tau(s))\veca_0 + \tau(s)\veca$,
where $\tau: \R \ra [0,1]$ is a smooth map satisfying
\begin{itemize}
  \item $\tau(s) = 0$ for $s \leq -R$
  \item $\tau(s) = 1$ for $s \geq 0$
  \item $0 \leq \tau'(s) \leq 2/R$ for all $s \in \R$.
\end{itemize}

Now define $\wt{Q}: \C^n \setm \{\vec{0}\} \ra \C^n \setm \{\vec{0}\}$ by
$\wt{Q}(\vecz) := \wt{\Phi}_{\vfa(\rr(\vecz))}(\vecz)$.
Note that we have $\wt{Q}(\vecz) = \vecz$ when $\vfa(\rr(\vecz)) = \veca_0$, i.e. when $\rr(\vecz) \leq -R$, i.e. for $||\vecz||$ sufficiently small.
Put also ${Q := \wt{Q} \circ F_\veca: \R \times \bdy E(\veca) \ra \C^n \setm \{\vec{0}\}}$, so that we have $Q(r,\vecz) = \Phi_\veca(r,\vecz)$ when $r \geq 0$, and hence $Q^*i$ takes the desired form on the positive end $\R_{\geq 0} \times \bdy E(\veca)$. 
It therefore suffices to establish the following.
\begin{sublemma}\label{sublem:slow_implies_diffeo}
For $R$ sufficiently large, $\wt{Q}$ is a diffeomorphism such that the almost complex structure $\wt{Q}^*i$ on $\C^n \setm\{\vec{0}\}$ is tamed by $\omega_\std$.
\end{sublemma}

\begin{proof}[Proof of Sublemma~\ref{sublem:slow_implies_diffeo}]
For $(r_0,p_0) \in \R \times \bdy E(\veca)$, and $\vecb \in \R_{>0}^n$, let us introduce the shorthand $\Phi_{\vecb}^\veca$ for the composition
\[\begin{tikzcd}
  \R \times \bdy E(\veca) & \C^n \setm \{\vec{0}\} & \R \times \bdy E(\vecb) & \C^n \setm \{\vec{0}\}
  \arrow["F_\veca", from=1-1, to=1-2]
  \arrow["F_\vecb^{-1}", from=1-2, to=1-3]
  \arrow["\Phi_{\vecb}", from=1-3, to=1-4].
  \arrow["\Phi^\veca_\vecb", from=1-1, to=1-4,bend right=13,swap]
\end{tikzcd}\]
In particular, for any $(r_0,p_0) \in \R \times \bdy E(\veca)$ we can write
\begin{align*}
Q(r_0,p_0) = \wt{\Phi}_{\vfa(r_0)}(F_\veca(r_0,p_0)) = \Phi_{\vfa(r_0)}^{\veca}(r_0,p_0).
\end{align*}
Note that $(\Phi_\vecb^\veca)^*i$ is tamed by $d(e^r\alpha_\veca)$ for any $\vecb \in \R_{>0}^n$,
since $\Phi_\vecb^*i$ is tamed by $d(e^r\alpha_\vecb)$.
 Moreover, for any ${v \in T_{p_0}\bdy E(\veca) \subset T_{(r_0,p_0)}(\R \times \bdy E(\veca))}$ (i.e. $v$ is a tangent vector with no $\bdy_r$ component), 
we have $dQ(v) = d\Phi^\veca_{\vfa(r_0)}(v)$,
and in particular for 
any nonzero $v \in \ker \alpha_\veca$ we have
\begin{align*}
d(e^r\alpha_\veca)(v,(Q^*i)v) = d(e^r\alpha_\veca)(v,((\Phi^\veca_{\vfa(r_0)})^*i)  v) > 0.
\end{align*}

For general $v \in T_{(r_0,p_0)}(\R \times \bdy E(\veca))$, the main idea is that, for $R$ large, $dQ(v)$ is a small perturbation of $d\Phi^{\veca}_{\vfa(r_0)}(v)$, and hence $d(e^r\alpha_\veca)(v,(Q^*i)v)$ is a small perturbation of $d(e^r\alpha_\veca)(v,((\Phi^\veca_{\vfa(r_0)})^*i)  v)$ which is still positive.
Here ``small'' should be interpreted with respect to a suitable cylindrical metric.
More precisely, consider the maps 
\begin{align*}
\cQ := F^{-1}_{\veca_0}\circ Q: \R \times \bdy E(\veca) \ra \R \times \bdy E(\veca_0)
\end{align*}
and
\begin{align*}
\cPhi_\vecb^\veca := F^{-1}_{\veca_0}\circ \Phi_{\vecb}^\veca: \R \times \bdy E(\veca) \ra \R \times \bdy E(\veca_0),
\end{align*}
for $\vecb \in \R_{>0}^n$ and with $\veca_0 = (4\pi,\dots,4\pi)$ as before.

Let $i_\vecb$ denote the pullback of the standard complex structure $i$ under the map $F_\vecb: \R \times \bdy E(\vecb) \ra \C^n \setm \{\vec{0}\}$.
In the following, we measure norms of vectors in $T(\R \times \bdy E(\vecb))$ by fixing a choice of Riemannian metric which is $i_\vecb$-invariant and invariant under translations in the $\R$ factor. 
Observe that, in order prove Sublemma~\ref{sublem:slow_implies_diffeo}, it suffices to show that $\cQ^*i_{\veca_0}$ is tamed by $d(e^r\alpha_\veca)$.

Let $L_{\veca_0,\veca}$ denote the line segment connecting $\veca_0$ and $\veca$ in $\R^n$.
The following can be established by long but elementary derivative computation based on the explicit formulas above:
\begin{claim}\label{claim:error_bounds}
  There is a constant $C > 0$ such that,
   for any $(r_0,p_0) \in \R \times \bdy E(\veca)$, $v \in T_{(r_0,p_0)}(\R \times \bdy E(\veca))$ and $\vecb \in L_{\veca_0,\veca}$, we have 
\begin{align*}
\tfrac{1}{C}||v|| \leq ||d\cPhi_\vecb^\veca(v)|| \leq C ||v||.
\end{align*}
Moreover, for any given $\delta > 0$, for $R$ sufficiently large we have
\begin{align*}
||d\cQ(v) - d\cPhi_{\vfa(r_0)}^\veca(v)|| \leq \delta ||v||.
\end{align*}
\end{claim}

\begin{claim}
For $R$ as in Claim~\ref{claim:error_bounds}, we have
\begin{align*}
||\cQ^* i_{\veca_0} - (\cPhi^{\veca}_{\veca_0})^* i_{\veca_0})|| \leq C\delta + D\delta(C+\delta),
\end{align*}
where $D := C(\tfrac{1}{C}-\delta)^{-1}$.
\end{claim}

\begin{proof}
  For $(r_0,p_0) \in \R \times \bdy E(\veca)$ and $v \in T_{(r_0,p_0)}(\R \times \bdy E(\veca))$, we have
\begin{align*}
(\cQ^*i_{\veca_0})v = d\cQ^{-1}( i_{\veca_0} d\cQ(v)) = d\cQ^{-1}( i_{\veca_0} d\cPhi^{\veca}_{\vfa(r_0)}(v)+i_{\veca_0}v')),
\end{align*}
where $||v'|| \leq \delta ||v||$.
Using Lemma~\ref{lem:lin_alg} below, we then have 
\begin{align*}
(\cQ^*i_{\veca_0})v &= d(\cPhi^\veca_{\vfa(r_0)})^{-1}( i_{\veca_0} d\cPhi^{\veca}_{\vfa(r_0)}(v)+i_{\veca_0}v')) + v''
\end{align*}
where $v''$ satisfies
\begin{align*}
||v''|| \leq D\delta|| i_{\veca_0} d\cPhi^{\veca}_{\vfa(r_0)}(v)+i_{\veca_0}v'|| \leq D\delta(C + \delta)||v||
\end{align*}
for $D := C (\tfrac{1}{C} - \delta)^{-1}$.
That is, we have
\begin{align*}
(\cQ^*i_{\veca_0})v = ((\cPhi^{\veca}_{\veca_0})^* i_{\veca_0})v +  d(\cPhi^\veca_{\vfa(r_0)})^{-1} (i_{\veca_0}v') + v'' = ((\cPhi^{\veca}_{\veca_0})^* i_{\veca_0})v +  v''',
\end{align*}
where
\begin{align*}
||v'''|| = ||d(\cPhi^\veca_{\vfa(r_0)})^{-1} (i_{\veca_0}v') + v''|| \leq C\delta ||v|| + D\delta(C+\delta)||v||.
\end{align*}
\end{proof}

To complete the proof of Sublemma~\ref{sublem:slow_implies_diffeo}, we must show that taking $\delta$ sufficiently small in Claim~\ref{claim:error_bounds} gives $d(e^r\alpha_\veca)(v,(\cQ^*i_{\veca_0})v) > 0$, or equivalently
$$(dr \wedge \alpha_\veca + d\alpha_\veca)(v,(\cQ^*i_{\veca_0})v) > 0,$$ for any nonzero $v \in T(\R \times \bdy E(\veca))$.
Observe that we can find  $K > 0$ such that we have
\begin{align*}
(dr \wedge \alpha_\veca + d\alpha_\veca)(v,((\cPhi^\veca_{\vecb})^*i_{\veca_0})v) \geq K ||v||^2\;\;\;\;\;\text{and}\;\;\;\;\; |(dr \wedge \alpha_\veca + d\alpha_\veca)(v,w)| \leq K||v||\cdot||w||
\end{align*}
for all $\vecb \in L_{\veca_0,\veca}$ and all $v \in T(\R \times \bdy E(\veca))$.
Then, for $v \in T_{(r_0,p_0)}(\R \times \bdy E(\veca))$, we have
\begin{align*}
(dr \wedge \alpha_\veca + d\alpha_\veca)(v,(\cQ^*i_{\veca_0})v) &=
(dr \wedge \alpha_\veca + d\alpha_\veca)(v,((\cPhi^\veca_{\vfa(r_0)})^*i_{\veca_0})v) \\&\;\;\;\;\; + (dr \wedge \alpha_\veca + d\alpha_\veca)(v, (\cQ^*i_{\veca_0} - (\cPhi^\veca_{\vfa(r_0)})^*i_{\veca_0})v)
\\ &\geq K ||v||^2 - K(C\delta + D\delta(C+\delta))||v||^2.
\end{align*}
For $\delta > 0$ sufficiently small we can assume 
$K - K(C\delta + D\delta(C+\delta)) > 0$, which gives the desired result.

\end{proof}

\end{proof}

\begin{lemma}\label{lem:lin_alg}
Let $V$ and $W$ be normed real vector spaces, and let $F,G: V \ra W$ be invertible linear operators satisfying
\begin{align*}
\tfrac{1}{C}||v|| \leq ||F(v)|| \leq C||v|| \;\;\;\text{and}\;\;\; ||F(v) - G(v)|| \leq \delta ||v||
\end{align*}
for some $C,\delta > 0$ with $\delta < \tfrac{1}{C}$ and all $v \in V$.
Then we have
\begin{align*}
||F^{-1} - G^{-1}|| \leq C\delta (\tfrac{1}{C} - \delta)^{-1}.
\end{align*}
\begin{proof}
Note that we have $||F^{-1}|| \leq C$, and for all $v \in V$ we have
\begin{align*}
||G(v)|| &\geq ||F(v)|| - ||G(v) - F(v)||
\\&\geq \tfrac{1}{C}||v|| - \delta ||v||,
\end{align*}
and hence $||G^{-1}||  \leq (\tfrac{1}{C} - \delta)^{-1}$.
Then for all $w \in W$ we have
\begin{align*}
||F^{-1}(w) - G^{-1}(w)|| &\leq ||F^{-1}|| \cdot ||w - F(G^{-1}(w))||
\\ &\leq C \cdot ||G(G^{-1}(w)) - F(G^{-1}(w))|| 
\\&\leq C \delta ||G^{-1}(w)|| 
\\&\leq C\delta (\tfrac{1}{C} - \delta)^{-1}||w||.
\end{align*}

\end{proof}

\end{lemma}

\subsection{Asymptotics correspondence}\label{subsec:asymp_corr}

As usual fix $\veca = (a_1,\dots,a_n) \in \R_{>0}^n$, and assume for simplicity that the components of $\veca = (a_1,\dots,a_n)$ are rationally independent.
In order to apply the above results, we work with fixed almost complex structures on $\R \times E(\veca)$ and $\wh{E}(\veca)$:
\begin{definition}\label{def:J_E_etc}
Put $J_{\bdy E(\veca)} := \Phi_{\veca}^*i \in \calJ(\bdy E(\veca))$ and $J_{E(\veca)} := (Q^+_\veca)^*i \in \calJ(E(\veca))$, where $\Phi_\veca$ is provided by Proposition~\ref{prop:Phi_a_exists} and $Q^+_\veca$ is provided by Proposition~\ref{prop:Q_exists}(1). 
\end{definition}

Now let $\Sigma$ be a punctured Riemann sphere, and let $u: \Sigma \ra \C^n \setm \{\vec{0}\}$ be a proper $i$-holomorphic map. 
By postcomposing with the map $\Phi_\veca^{-1}$ from Proposition~\ref{prop:Phi_a_exists}, we get an asymptotically cylindrical $J_{E(\veca)}$-holomorphic curve $\Phi_\veca^{-1} \circ u: \Sigma \ra \wh{E}(\veca)$, whose Reeb orbit asymptotics we seek to understand.

Given a puncture $w$ of $\Sigma$, let $\vecv = (v_1,\dots,v_n)$ denote the respective pole orders (i.e. negative vanishing orders) of the components of $u$ at $w$.
Since $u$ is a proper map, we must have either
\begin{enumerate}
  \item $\vecv \in \Z^n \setm \Z^n_{\leq 0}$, in which case $||u(z)||$ is unbounded as $z \ra w$, or
  \item $\vecv \in \Z_{< 0}^n$, in which case $u(z) \ra \vec{0}$ as $z \ra w$.
\end{enumerate}
Concretely, if $z$ is a local complex coordinate for $\Sigma$ centered at the puncture $w$,
then for $z$ near $0$, to highest order $u$ is of the form
\begin{align*}
u(z) = (C_1 z^{-v_1},\dots,C_n z^{-v_n})
\end{align*}
for some $C_1,\dots,C_n \in \C$.

Let $\pi_{\bdy E(\veca)}: \R \times \bdy E(\veca) \ra \bdy E(\veca)$ denote projection to the second factor.
For $\vecz \in \C^n \setm \{\vec{0}\}$, let $\rr_\veca(\vecz)$ be the unique real number such that $\fl_{Z_\veca}^{-\rr_\veca(\vecz)}(\vecz) \in \bdy E(\veca)$, or equivalently $H_\veca(\fl_{Z_\veca}^{-\rr_\veca(\vecz)}(\vecz)) = 1$.
Then, putting $s = |z|^{-1}$, we have
\begin{align*}
(\pi_{\bdy E(\veca)} \circ \Phi_\veca^{-1} \circ u)(z) = G_\veca(\fl_{Z_\veca}^{-\rr_\veca(u(z))}(u(z))),
\end{align*}
where
\begin{align}\label{eq:sum_to_1}
H_\veca(\fl^{-\rr_\veca(u(z))}_{Z_\veca}(u(z))) = \pi \sum_{i=1}^n e^{-4\pi \rr_\veca(u(z))/a_i} |C_i|^2 s^{2v_i}/a_i = 1.
\end{align}

Inspecting \eqref{eq:sum_to_1}, note that we must have
$\lim\limits_{s \ra \infty}\rr_\veca(u(z)) = \infty$ if $\vecv \in \Z^n \setm \Z_{\leq 0}^n$,
and $\lim\limits_{s \ra \infty}\rr_\veca(u(z)) =  -\infty$ if $\vecv \in \Z_{<0}^n$.
Furthermore, for $i = 1,\dots,n$, the term $e^{-4\pi \rr_\veca(u(z))/a_i} |C_i|^2 s^{2v_i}/a_i$ must be bounded from above for $s$ large.
Therefore we must have
$e^{-4\pi \rr_\veca(u(z))/a_i} s^{2v_i} < const$, i.e. 
$-4\pi \rr_\veca(u(z))/a_i + 2v_i \ln(s) < const$, i.e. 
$\rr_\veca(u(z)) > \tfrac{a_iv_i}{2\pi}\ln(s) + const$.
Since not every summand in \eqref{eq:sum_to_1} can limit to $0$ as $s \ra \infty$, we must have 
$\rr_\veca(u(z)) \approx \tfrac{a_jv_j}{2\pi}\ln(s)$ for some $j \in \{1,\dots,n\}$.
Then evidently we must have $j = i_M$, where $i_M$ is the index $1 \leq i \leq n$ for which $a_iv_i$ is maximal.

Then we have 
\begin{align*}
\fl_{Z_\veca}^{-\rr_\veca(u(z))}(u(z)) \approx (s^{-v_{i_M}a_{i_M}/a_1}u_1(z),\dots,s^{-v_{i_M}a_{i_M}/a_n}u_n(z)),
\end{align*}
and taking the modulus of each component gives
\begin{align*}
(|C_1|s^{v_1 - v_{i_M}a_{i_M}/a_1},\dots, |C_n|s^{v_n - v_{i_M}a_{i_M}/a_n}).
\end{align*}
Note that we have $v_i - v_{i_M}a_{i_M}/a_i < 0$ for $i \neq i_M$, in which case the $i$th component of $\fl_{Z_\veca}^{-\rr_\veca(u(z))}(u(z))$ limits to zero as $s \rightarrow \infty$, and the same is also true for $G_\veca(\fl_{Z_\veca}^{-\rr_\veca(u(z))}(u(z)))$.
In other words, as $z$ approaches the puncture $w$, 
$(\Phi_\veca^{-1} \circ u)(z)$ is positively asymptotic to $\nu_{i_M}^{v_{i_M}}$, i.e. the $v_{i_M}$-fold cover of the Reeb orbit given by the intersection of $\bdy E(\veca)$ with the $i_M$th complex axis.

Let us introduce the following:
\begin{notation} 
  Given $\vecv = (v_1,\dots,v_n) \in \Z^n \setm \Z_{\leq 0}^n$, put $\orb^{\veca}_{\vecv} := \nu^{v_{i_M}}_{i_M}$, where $i_M$ is the index $1 \leq i \leq n$ for which $a_iv_i$ is maximal. Similarly, for $\vecv \in \Z_{<0}^n$ put $\orb_\veca^\veca := \nu_{i_m}^{|v_{i_m}|}$, where $i_m$ is the index $1 \leq i \leq n$ for which $a_i |v_i|$ is minimal.

\end{notation}

\NI Summarizing the above discussion, we have:

\begin{prop}\label{prop:asymp_corr_pos_and_neg}
Let $\Sigma$ be a punctured Riemann sphere, and let $u: \Sigma \ra \C^n \setm \{\vec{0}\}$ be a proper $i$-holomorphic map. Let $w$ be a puncture of $\Sigma$, and let $\vecv = (v_1,\dots,v_n)$ denote the pole orders (i.e. negative vanishing orders) of the components of $u$ at $w$.
Then, at the puncture $w$, the curve $\Phi_\veca^{-1} \circ u: \Sigma \ra \R \times \bdy E(\veca)$ is:
\begin{itemize}
  \item positively asymptotic to the Reeb orbit $\orb^\veca_{\vecv}$ if $\vecv \in \Z^n \setm \Z_{\leq 0}^n$
  \item negatively asymptotic to the Reeb orbit $\orb^\veca_{\vecv}$ if $\vecv \in \Z_{< 0}^n$. 
\end{itemize}
\end{prop}

The same argument applies equally to $i$-holomorphic curves in $\C^n$, in which case the resulting curve asymptotically cylindrical curve in $\wh{E}(\veca)$ has only positive punctures:

\begin{prop}\label{prop:asymp_corr_pos_only}
Let $\Sigma$ be a punctured Riemann sphere, and let $u: \Sigma \ra \C^n$ be a proper $i$-holomorphic map. Let $w$ be a puncture of $\Sigma$, and let $\vecv = (v_1,\dots,v_n)$ denote the denote the pole orders of the components of $u$ at $w$.
Then, at the puncture $w$, the curve $(Q_\veca^+)^{-1} \circ u: \Sigma \ra \wh{E}(\veca)$ is positively asymptotic to the Reeb orbit $\orb^\veca_{\veca}$.
\end{prop}

\begin{rmk}
Using Corollary~\ref{cor:M_a_corresp}, it is possible to set up a correspondence between singularities
of closed pseudoholomorphic curves in a closed symplectic manifold $M$ and negative ends of asymptotically cylindrical curves in $\wh{M}_\veca$.
In particular, in the case $\dim_R(M) = 4$, this makes it possible to reformulate $\Tcount_{M,A}^a$ in many cases as a certain count of pseudoholomorphic curves in $M$ with a prescribed cusp singularity at a point $p_0$. 
This will be treated in detail in the forthcoming work \cite{CSEN}.
\end{rmk}

\section{Stationary descendant counts in ellipsoids}\label{sec:descendants}

\subsection{The higher dimensional lattice path $\Ga^\veca$}\label{subsec:lattice_path}

We define the higher dimensional analogue of the lattice path from \S\ref{sec:prelude} as follows.
\begin{definition}
For $\veca = (a_1,\dots,a_n) \in \R_{>0}^n$ and $k \in \Z_{\geq 0}$, we put
\begin{align*}
\Ga_k^\veca := \min\limits_{\substack{(v_1,\dots,v_n) \in \Z_{\geq 0}^n\\ v_1 + \cdots + v_n= k}}
\max\limits_{1 \leq i \leq n} a_iv_i
\end{align*}
if the components of $\veca$ are rationally independent, otherwise we put use the convention
$\Ga^{(a_1,\dots,a_n)}_k := \Ga^{(a_1+\delta_1,\dots,a_n+\delta_n)}_k$, where $0 < \delta_1 < \cdots < \delta_n$ are sufficiently small and rationally independent.
\end{definition}
\NI In the case $n=2$, one can easily check that $\Ga^a := \Ga^{(1,a)}$ agrees with the definition in \S\ref{sec:prelude}.

\sss

For $1 \leq i < j \leq n$, let $\pi_{\{i,j\}}: \Z_{\geq 0}^n \ra \Z_{\geq 0}^2$ denote the corresponding projection map.
Let $L_{a_j/a_i}$ denote the line in $\R^2$ passing through $(0,-1)$ and $(a_j/a_i,0)$,
and let $L_{a_j/a_i} + (-1,1)$ denote the line passing through $(-1,0)$ and $(a_j/a_i-1,1)$.
\begin{lemma}
The lattice path $\Ga^\veca = (\Ga^\veca_0,\Ga^\veca_1,\Ga^\veca_2,\dots)$ is the unique unit step lattice path such that $\pi_{\{i,j\}}(\Ga^\veca)$ lies on or above $L_{a_j/a_i}$ and below $L_{a_j/a_i}+(-1,1)$\footnote{As formulated, this lemma actually holds even if the components of $\veca$ are not rationally independent.} for all $1 \leq i < j \leq n$.
\end{lemma}
\NI In particular, note that the projected lattice path $\pi_{\{i,j\}}(\Ga^\veca)$ has the same image in $\Z_{\geq 0}^2$ as the lattice path $\Ga^{(a_i,a_j)}$ for all $1 \leq i < j \leq n$.
\begin{proof}
Observe that $\Ga^\veca_k = \vecv = (v_1,\dots,v_n)$ if and only if 
$||\vecv'||^*_\veca > ||\vecv||^*_\veca$ whenever $\vecv' = \vecv + e_j - e_i$ or $\vecv' = \vecv - e_j + e_i$ for some $1 \leq i < j \leq n$.
This is equivalent to having $a_j(v_j + 1) > \max\limits_{1 \leq i \leq n} a_iv_i$ for each $1 \leq j \leq n$, or equivalently $a_j(v_j + 1) > a_iv_i$ for all $1 \leq i,j \leq n$.
For $i < j$, $\pi_{\{i,j\}}(\vecv)$ lies above $L_{a_j/a_i}$ if and only if $a_j(v_j+1) > a_iv_i$, and $\pi_{\{i,j\}}(\vecv)$ lies below $L_{a_j/a_i} + (-1,1)$ if and only if $a_i(v_i+1) > a_jv_j$.
\end{proof}

\sss

Recall that for $\vecv \in \Z^n \setm \Z_{\leq 0}^n$ we put $\orb^\veca_{\vecv} = \nu_{i_M}^{v_{i_M}}$, where $i_M$ is the index $1 \leq i \leq n$ for which $a_i v_i$ is maximal, and where $\nu_1,\dots,\nu_n$ denote the simple Reeb orbits in $\bdy E(\veca)$. 
In particular, the action of $\orb^\veca_{\vecv}$ is given by $$\calA(\orb^\veca_{\vecv}) = \max\limits_{1 \leq i \leq n}a_iv_i = a_{i_M}v_{i_M}.$$
Observe that we have $\orb^\veca_{\Ga^\veca_k} = \orb^\veca_k$ for any $k
 \in \Z_{\geq 1}$.
In fact, $\Ga^\veca_k$ is in a sense the maximal tuple which corresponds to the orbit $\orb^\veca_k$:
\begin{lemma}\label{lem:Ga_maximal}
 Suppose that $\vecv \in \Z^n \setm \Z_{\leq 0}^n$ satisfies $\orb_{\vecv}^\veca = \orb^\veca_k$ for some $k \in \Z_{\geq 1}$. Then each component of $\vecv$ is at most the corresponding component of $\Ga^\veca_k$.
\end{lemma}
\begin{proof}
Put $\Ga_k^\veca = (w_1,\dots,w_n) \in \Z_{\geq 0}^n$,
which by definition minimizes $\max\limits_{1 \leq i \leq n} a_i w_i$ subject to $\sum\limits_{i = 1}^n w_i = k$, and put $\vecv = (v_1,\dots,v_n)$.
Suppose by contradiction that we have $v_j > w_j$ for some $j \in \{1,\dots,n\}$. 
Note that we have 
\begin{align*}
a_j(w_j+1) > \max\limits_{1 \leq i \leq n}a_i w_i ,
\end{align*}
since otherwise $(w_1,\dots,w_n)$ would not be minimal.
Then we have
\begin{align*}
\max\limits_{1 \leq i \leq n}a_iw_i = \calA(\orb_k^\veca) = \calA(\orb^\veca_{\vecv}) = \max\limits_{1 \leq i \leq n}a_i v_i \geq a_j v_j \geq a_j (w_j + 1) > \max\limits_{1 \leq i \leq n} a_i w_i,
\end{align*}
which is a contradiction.
\end{proof}

\begin{rmk}
Lemma~\ref{lem:Ga_maximal} can be viewed as a higher dimensional manifestation of the SFT winding number bound (see \cite[\S3]{wendl_SFT_notes} and the references therein), which arises the analysis of eigenfunctions of asymptotic operators associated to Reeb orbits. In dimension $3$ this bounds the winding of a positive end of a pseudoholomorphic curve about its asymptotic Reeb orbit $\ga$ in terms of the Conley--Zehnder index of $\ga$ (see also \cite{Hlect} for the relationship with embedded contact homology).
\end{rmk}

\subsection{Main descendant formula}

The main result of this section is the following, which computes stationary descendant punctured curve counts in ellipsoids:
\begin{thm}\label{thm:main_descendant}
For any $\veca \in \R_{> 0}^n$ and $i_1,\dots,i_k \in \Z_{\geq 1}$, we have
\begin{align*}
  \aug_\veca^k(\orb^\veca_{i_1},\dots,\orb^\veca_{i_k}) = \frac{q_{i_1+\cdots+i_k+k-1}}{(\sum_{s=1}^k\Ga^\veca_{i_s})!}.
  \end{align*}

\end{thm}
\NI Recall that for $\vecv = (v_1,\dots,v_n) \in \Z_{\geq 0}^n$ we put $\vecv\,! := v_1!\cdot \cdots\cdot v_n!$.
This theorem forms the basis for the enumerative formulas given in the next section.
It is noteworthy that the formula takes such a concise form. In contrast, if we replace the stationary descendant condition with e.g. the analogous local tangency constraint then even nonvanishing of any given count is quite nontrivial.

\sss

Before giving the details in the next subsection, let us briefly sketch the idea of the proof.
By the results of the previous section, we can reduce the computation of punctured curve counts in $\wh{E}(\veca)$ with prescribed Reeb orbit asymptotics $\orb^\veca_{i_1},\dots,\orb^\veca_{i_k}$ to a question about $i$-holomorphic curves in $(\C^*)^n$.
The lattice terms $\Ga^\veca_{i_1},\dots,\Ga^{\veca}_{i_k}$ naturally appear here in this reformulation.
Moreover, given a Riemann sphere $\Sigma$ with $k$ punctures, there is a unique proper $i$-holomorphic map $u: \Sigma \ra (\C^*)^n$ which passes through a prescribed point $p_0 \in (\C^*)^n$ and such that small loops around the punctures trace out prescribed homology classes in $H_1((\C^*)^n)$.
This allows us to reduce the computation of stationary descendants in an ellipsoid to a standard intersection theory problem on the Deligne--Mumford space $\ovl{\calM}_{0,k+1}$ of stable nodal spheres with $k+1$ (ordered) marked points.
Finally, the factorial term $(\Ga^\veca_{i_1} + \cdots + \Ga^\veca_{i_k})!$ comes out naturally after taking care of combinatorial data arising from the ordering of certain marked points.

\begin{rmk}
    Theorem~\ref{thm:main_descendant} has the following heuristic interpretation in terms of tropical geometry.
The stationary descendant condition $\lll \psi^{k}\pt\rrr$ corresponds to a tropical curve having a vertex with valency $k+2$ which passes through a specified point in $\R^n$. 
Each Reeb orbit asymptotic to $\orb^\veca_{i_s}$ corresponds to a ray in direction $\Ga^\veca_{i_s}$. Then $N_{E(\veca)}(\orb^\veca_{i_1},\dots,\orb^\veca_{i_k})\lll \psi^{c_1(A)-2}\pt\rrr$ is carried by a single tropical curve with:
\begin{itemize}
  \item a unique vertex, which has valency $c_1(A)-2$ and passes through a prescribed point in $\R^n$
  \item rays in the directions $\Ga^\veca_{i_1},\dots,\Ga^\veca_{i_k}$
  \item for $i = 1,\dots,n$, rays in the direction $-e_i$, the number of which is the $i$th component of $\Ga^\veca_{i_1}+\cdots+ \Ga^\veca_{i_k}$ (note that balancing holds).
\end{itemize}
The term $\tfrac{1}{(\Ga^\veca_{i_1}+\cdots+ \Ga^\veca_{i_k})!}$ then comes from dividing out by the number of orderings of the latter rays.

Tropical descendant curves have been considered in the context of closed toric varieties by various authors -- see e.g. \cite{mandel2020descendant,markwig2009tropical,cavalieri2021counting} and the references therein.
To our knowlege, tropical methods have not been explicitly utilized to study punctured curves in ellipsoids, let alone more general star-shaped domains.
Theorem~\ref{thm:main_descendant} could be viewed as an SFT computation inspired by tropical geometry, but our proof utilizes explicit maps between moduli spaces and does not rely on any general correspondence theorem for tropical descendant curves.
Since our computation depends on the results in \S\ref{sec:J_std}, which are formulated only for ellipsoids, it is interesting to try to extend Theorem~\ref{thm:main_descendant} to more general shapes.
\end{rmk}

\subsection{Proof of Theorem~\ref{thm:main_descendant}}\label{subsec:proof_of_main_des}

\subsubsection{Curves in $(\C^*)^n$ of fixed toric degree}\label{subsubsec:toric_degree}

Using the results of \S\ref{sec:J_std}, we will explicitly analyze $\ovl{\calM}^{J_E}_{E(\veca)}(\orb_{i_1}^\veca,\dots,\orb_{i_k}^\veca)\lll p_0\rrr$ in terms of $i$-holomorphic curves in $(\C^*)^n$.
Let us first introduce some relevant formalism.

\begin{definition}
An {\bf ordered toric degree} (of dimension $n$) is a tuple $\bd = (\vecv_1,\dots,\vecv_\vk)$ for some $\vk \in \Z_{\geq 2}$, with $\vecv_1,\dots,\vecv_\vk \in \Z^n \setm \{\vec{0}\}$ satisfying $\sum_{i=1}^\vk \vecv_i = \vec{0}$.

An {\bf unordered toric degree} is an ordered toric degree $\ovl\bd = (\vecv_1,\dots,\vecv_\vk)$ considered only up to permuting the ordering of the tuple.
\end{definition}

\begin{definition}
  Given an ordered toric degree $\bd = (\vecv_1,\dots,\vecv_\vk)$, $\calM_{(\C^*)^n}^i(\bd)$ denotes the moduli space of proper $i$-holomorphic maps $u: \Sigma \ra (\C^*)^n$, where $\Sigma$ is a Riemann sphere with $k$ ordered punctures, such that a small oriented simple loop around the $i$th puncture lies in the homology class $\vecv_i \in \Z^n \approx H_1((\C^*)^n)$ for $i = 1,\dots,\vk$.
\end{definition}
\NI If $\Sigma$ is a Riemann sphere with $k$ ordered punctures and $u: \Sigma \ra (\C^*)^n$ is a proper $i$-holomorphic map, then $u$ has a well-defined ordered toric degree $\bd(u)$ such that $u$ represents an element of $\calM^i_{(\C^*)^n}(\bd)$, and this can be read off by looking at the homology classes of loops around the punctures.
If the punctures of $\Sigma$ are unordered, then $u$ still has a well defined unordered toric degree $\ovlbd(u)$.

We define the moduli space $\calM_{(\C^*)^n}^i(\bd)\lll \pn \rrr$ similarly, but with the domain equipped with one marked point which maps to the distinguished basepoint $\pn := (1,\dots,1) \in (\C^*)^n$.
There is a natural forgetful map $\calM^i_{(\C^*)^n}(\bd)\lll \pn\rrr \ra \calM_{0,1+\vk}$
which remembers only the domain of a curve.\footnote{Technically this map also replaces each puncture with a marked point but in this paper the distinction between punctures and marked points is purely formal.}
\begin{lemma}\label{lem:Mdelta_to_DM}
Let $\bd = (\vecv_1,\dots,\vecv_\vk)$ be an ordered toric degree, and let $\Sigma$ represent an element of $\calM_{0,\vk+1}$.
Then there is a unique $i$-holomorphic map $f_\Sigma: \Sigma \ra (\C^*)^n$ of ordered toric degree $\bd$ which passes through $\pn$.
In particular, the map $\calM^i_{(\C^*)^n}(\bd)\lll \pn\rrr \ra \calM_{0,1+\vk}$ is a diffeomorphism.
\end{lemma}

\begin{proof}
Let $w_0,\dots,w_\vk$ denote the punctures of $\Sigma$.
Put $v_i = (v_i^1,\dots,v_i^n)$ for $i = 1,\dots,\vk$.
Recall that we have $\sum_{i=1}^\vk \vecv_i = \vec{0}$, so in particular $\sum_{i=1}^\vk v_i^j = 0$.
Then there is a unique holomorphic map $f_\Sigma^j: \Sigma \ra \C^*$ which has a pole of order $v^j_i$ (i.e. a zero of order $-v^j_i$) at the $i$th puncture for $i = 1,\dots,\vk$, and such that $f(w_0) = 1$.
Explicitly, after identifying $\Sigma$ with $\C \cup \{\infty\}$ we have
\begin{align*}
f_\Sigma^j(z) = C_j (z-w_1)^{-v^j_1}\cdots (z-w_\vk)^{-v^j_\vk},
\end{align*}
where $C_j = (w_0-w_1)^{v^j_1}\cdots (w_0-w_\vk)^{v^j_\vk}$.
Taken together, these give a unique proper holomorphic map
\begin{align*}
f_\Sigma = (f_\Sigma^1,\dots,f_\Sigma^n): \Sigma \ra (\C^*)^n
\end{align*}
satisfying $f_\Sigma(w_0) = \pn$.
\end{proof}

\subsubsection{Closed curve stationary descendants}\label{subsubsec:closed_curves}

Before discussing the ellipsoid case, we first explain how to compute closed curve stationary descendants in Fano toric varieties.

\begin{thm}\label{thm:closed_des}
  Let $M$ be a smooth complex Fano toric variety associated to a Delzant polytope $P$, and fix a homology class $A \in H_2(M)$.
  Let $D_1,\dots,D_m$ denote the toric boundary divisors, and assume that we have $d_i := A \cdot [D_i] > 0$ for $i = 1,\dots,m$. Then we have
  \begin{align*}
  N_{M,A}\lll \psi^{c_1(A)-2} \pt \rrr = (d_1!\cdots d_m!)^{-1}.
  \end{align*}
\end{thm}
\NI This result is well-known, but we give a proof will serve as the basis for the ellipsoidal case. 
Note that this computation also serves as the base case of the recursion in Theorem~\ref{thm:main_rec_intro}.

\begin{example}
In the case of complex projective space we have $N_{\CP^n,d[L]}\lll \psi^{d(n+1)-2}\pt\rrr = (d!)^{-(n+1)}$.  
\end{example}

\sss

We work throughout with the fixed toric integrable almost complex structure $J_M$ on $M$, so that there is a (non-proper) holomorphic inclusion $(\C^*)^n \subset M$. 
We take $p_0 \in M$ to be the image of the point $\pn := (1,\dots,1) \in (\C^*)^n$ under this inclusion.

Recall that $\calM^{J_M}_{M,A}\lll p_0\rrr$ denotes the moduli space of genus zero $J_M$-holomorphic curves in $M$ which lie in homology class $A$ and have one marked point required to pass through $p_0$.
We define $\calM^{J_M}_{M,A}\lll p_0,\Ddiv_1^{\times d_1},\dots,\Ddiv_m^{\times d_m}\rrr$ in the same way, except that curves are equipped with $\sum_{i=1}^m d_i$ additional marked points, such that the first $d_1$ map to $\Ddiv_1$, the next $d_2$ map to $\Ddiv_2$, and so on.
As usual, we denote the compactification by stable maps by $\ovl{\calM}^{J_M}_{M,A}\lll p_0,\Ddiv_1^{\times d_1},\dots,\Ddiv_m^{\times d_m}\rrr$.

  There is a natural map $\ovl{\calM}^{J_M}_{M,A}\lll p_0,\Ddiv_1^{\times d_1},\dots,\Ddiv_m^{\times d_m}\rrr \ra \ovl{\calM}_{M,A}^{J_M}\lll p_0\rrr$ given by forgetting all but the first marked point, and then contracting any resulting unstable components (which are necessarily constant). 
  There is also a natural map $\ovl{\calM}_{M,A}^{J_M}\lll p_0,\Ddiv_1^{\times d_1},\dots,\Ddiv_m^{\times d_m}\rrr \ra \ovl{\calM}_{0,1+\sum_{i=1}^m d_i}$ given by forgetting the map and remembering only the domain, and again contracting any resulting
  unstable components. 

\sss

Let $\bd = (\vecv_1,\dots,\vecv_\vk)$ be an ordered toric degree, and let $u: \Sigma \ra (\C^*)^n$ be a proper $i$-holomorphic map with ordered toric degree $\bd$.
Then $u$ uniquely extends to an $i$-holomorphic map $u': \Sigma' \ra M$, where $\Sigma'$ denotes the marked sphere given by filling in the punctures of $\Sigma$.
Conversely, any $J_M$-holomorphic curve in $M$ which is not entirely contained in $\Ddiv_1 \cup \cdots \cup \Ddiv_m$ (e.g. this is guaranteed by a point constraint $\lll p_0\rrr$) arises in this way.

For a puncture of $\Sigma$ corresponding to $\vecv \in \Z^n \setm \{\vec{0}\}$, the corresponding marked point $w \in \Sigma'$ has prescribed intersections with the toric divisors $\Ddiv_1,\dots,\Ddiv_m$ as follows.
Let $\nn_1,\dots,\nn_m \in \Z^n$ denote the primitive outward normal vectors to the codimension $1$ facets of the moment polytope $P$ of $M$. 
Since $P$ is Delzant, we can uniquely write $\vecv = c_1\nn_1 + \cdots + c_m\nn_m$ for some $c_1,\dots,c_m \in \Z_{\geq 0}$.
Then $u'$ has contact order $c_i$ with $\Ddiv_i$ at the marked point $w$.

In particular, putting $\bd_A := (\nn_1^{\times d_1},\dots,\nn_m^{\times d_m})$, there is a well-defined map $\calM_{(\C^*)^n}^i(\bd_{A}) \ra \calM^{J_M}_{M,A}\lll p_0,\Ddiv_1^{\times d_1},\dots,\Ddiv_m^{\times d_m}\rrr$.

\begin{lemma}\label{lem:delta_A}
The map $\calM_{(\C^*)^n}^i(\bd_{A}) \ra \calM^{J_M}_{M,A}\lll p_0,\Ddiv_1^{\times d_1},\dots,\Ddiv_m^{\times d_m}\rrr$ is a diffeomorphism. 
\end{lemma}
\begin{proof}
  Let $C$ represent an element of $\calM^{J_M}_{M,A}\lll p_0,\Ddiv_1^{\times d_1},\dots,\Ddiv_m^{\times d_m}\rrr$, and let $\bd$ be its ordered toric degree $\bd$.
  Note that $C$ has at least $\sum_{i=1}^m d_i$ geometric intersections with $\Ddiv_1 \cup \cdots \cup \Ddiv_m$, while by positivity of intersections this number is also at most $A \cdot (\sum_{i=1}^m[\Ddiv_i]) = \sum_{i=1}^m d_i$, so it must be exactly $\sum_{i=1}^m d_i$. Then each intersection is simple, and we must have $\bd = \bd_A$.
The inverse map $\calM^{J_M}_{M,A}\lll p_0,\Ddiv_1^{\times d_1},\dots,\Ddiv_m^{\times d_m}\rrr \ra \calM_{(\C^*)^n}^i(\bd_{A})$ is given by removing all but the first marked point of $C$ and viewing is at a proper $i$-holomorphic map into $(\C^*)^n$ of ordered toric degree $\bd_A$.
\end{proof}

Combining Lemma~\ref{lem:Mdelta_to_DM} and Lemma~\ref{lem:delta_A}, we have:
\begin{cor}\label{cor:right_map_diffeo}
  The map $\calM^{J_M}_{M,A}\lll p_0,\Ddiv_1^{\times d_1},\dots,\Ddiv_m^{\times d_m}\rrr \ra \calM_{0,1+\sum_{i=1}^m d_i}$ is a diffeomorphism.
\end{cor}

\begin{lemma}\label{lem:psi_compat}
The diagram
\begin{equation}\label{eq:closed_des_diag}
\begin{tikzcd}
  \ovl{\calM}_{M,A}^{J_M}\lll p_0\rrr & \ovl{\calM}_{M,A}^{J_M}\lll p_0,\Ddiv_1^{\times d_1},\dots,\Ddiv_m^{\times d_m}\rrr & \ovl{\calM}_{0,1+\sum_{i=1}^md_i}
  \arrow[from=1-2, to=1-1]
  \arrow[from=1-2, to=1-3]
\end{tikzcd}
\end{equation}
is compatible with $\psi$ classes, in the sense that the induced maps on cohomology send the $\psi$ classes in $H^2(\ovl{\calM}_{M,A}^{J_M}\lll p_0\rrr)$ and $H^2(\ovl{\calM}_{0,1+\sum_{i=1}^md_i})$ to the $\psi$ class in $H^2(\ovl{\calM}^{J_M}_{M,A}\lll p_0,\Ddiv_1^{\times d_1},\dots,\Ddiv_m^{\times d_m}\rrr)$.
\end{lemma}
\begin{proof}
We have a commutative diagram of pullback squares:
\begin{equation}\label{eq:univ_fam_diag}
\begin{tikzcd}
\ovl{\calU}_{M,A}^{J_M}\lll p_0\rrr & \ovl{\calU}_{M,A}^{J_M}\lll p_0,\Ddiv_1^{\times d_1},\dots,\Ddiv_m^{\times d_m}\rrr & \ovl{\calU}_{0,1+\sum_{i=1}^md_i}\\
  \ovl{\calM}_{M,A}^{J_M}\lll p_0\rrr & \ovl{\calM}^{J_M}_{M,A}\lll p_0,\Ddiv_1^{\times d_1},\dots,\Ddiv_m^{\times d_m}\rrr & \ovl{\calM}_{0,1+\sum_{i=1}^md_i}
  \arrow[from=1-2, to=1-1]
  \arrow[from=1-2, to=1-3]
  \arrow[from=2-2, to=2-1]
  \arrow[from=2-2, to=2-3]
  \arrow[from=1-1, to=2-1, bend right=20]
  \arrow[from=1-2, to=2-2, bend right=20]
  \arrow[from=1-3, to=2-3, bend right=20] 
  \arrow[to=1-1, from=2-1, bend right=20]
  \arrow[to=1-2, from=2-2, bend right=20]
  \arrow[to=1-3, from=2-3, bend right=20] 
\end{tikzcd},
\end{equation}
where the spaces in the top rows denote universal familes over the corresponding spaces in the bottom rows.
Concretely, these can be realized by adding one additional unconstrained marked point.
The downward arrows then correspond to forgetting the extra marked point, while the upward arrows are the sections associated to the marked point carrying the constraint $\lll p_0\rrr$.
For each of the spaces in the bottom row, the associated $\psi$ class is the Euler class of the complex line bundle given by the conormal to the corresponding section.
In particular, these line bundles themselves form a diagram of pullback squares, and the result follows by naturality of Chern classes.
\end{proof}

\begin{proof}[Proof of Theorem~\ref{thm:closed_des}]

Since $M$ is Fano, each of the spaces in diagram ~\eqref{eq:closed_des_diag} is a thin compactification in the sense of Definition~\ref{def:thin_comp}, so we can compute the degree of the maps in ~\eqref{eq:closed_des_diag} by counting the points in a generic fiber, and for this it suffices to look just at the strata of maximal dimension.\footnote{For more general $M$ we may have superabundant curves, which necessitate working with virtual fundamental classes.}
Since a generic curve in $\calM_{M,A}^{J_M}\lll p_0\rrr$ has only simple intersections with $\Ddiv_1 \cup \cdots \cup \Ddiv_m$, the degree of the left map in ~\eqref{eq:closed_des_diag} is $d_1!\cdots d_m!$, coming from the ways of ordering the intersection points with each $\Ddiv_i$.
Meanwhile, the degree of the right map in 
~\eqref{eq:closed_des_diag} is $1$ by Corollary~\ref{cor:right_map_diffeo}.
Finally, by Lemma~\ref{lem:psi_compat} we have 
\begin{align*}
(d_1!\cdots d_m!)\int\limits_{[\ovl{\calM}_{M,A}^{J_M}\lll p_0\rrr]}\psi^{c_1(A)-2} = \int\limits_{[\ovl{\calM}_{M,A}^{J_M}\lll p_0,\Ddiv_1^{d_1},\dots,\Ddiv_m^{\times d_m}\rrr]}\psi^{c_1(A)-2} = \int\limits_{[\ovl{\calM}_{0,1+\sum_{i=1}^m d_i}]}\psi^{c_1(A)-2},
\end{align*}
i.e. 
\begin{align}\label{eq:reduction_to_DM}
N_{M,A}\lll \psi^{c_1(A)-2}\pt\rrr = (d_1!\cdots d_m!)^{-1} \int\limits_{[\ovl{\calM}_{0,1+c_1(A)}]} \psi^{c_1(A)-2},
\end{align}
and the last expression is $(d_1!\cdots d_m!)^{-1}$ by Lemma~\ref{lem:DM_psi_comp} below.
\end{proof}
\begin{lemma}\label{lem:DM_psi_comp}
For any $k \in \Z_{\geq 1}$ we have
\begin{align*}
 \int\limits_{[\ovl{\calM}_{0,k+2}]} \psi^{k-1} = 1.
 \end{align*}

\end{lemma}
\begin{proof}
  This is a basic fact about intersection theory on $\ovl{\calM}_{0,k+2}$ (a simple special case of the Witten conjecture -- see \cite{kock2001notes}), which can be seen easily with our formalism as follows. 
We apply ~\eqref{eq:reduction_to_DM} in the case $M = \CP^k$ and the line class $A = [L]$.
In this case $m = k+1$ and we have $d_1 = \cdots = d_{k+1} = 1$, and hence
\begin{align*}
N_{\CP^k,[L]}\lll \psi^{k-1}\pt\rrr = \int\limits_{[\ovl{\calM}_{0,k+2}]} \psi^{k-1}.
\end{align*}
Observe that 
$\ovl{\calM}^{J_\std}_{\CP^k,[L]}\lll p_0 \rrr = \calM^{J_\std}_{\CP^k,[L]}\lll \pt\rrr$ is naturally identified with $\CP^{k-1}$, since a line passing through $p_0$ is determined by its tangent space at $p_0$.
Moreover, under this identification the class $\psi \in H_2(\ovl{\calM}^{J_\std}_{\CP^k,[L]}\lll p_0 \rrr)$ is the hyperplane class $[H] \in H^2(\CP^{k-1})$, so we have
\begin{align*}
N_{\CP^k,[L]}\lll \psi^{k-1}\pt\rrr = \int\limits_{[\CP^{k-1}]} [H]^{k-1} = 1.
\end{align*}
\end{proof}

\subsubsection{Setup for punctured curves}

We now discuss the case of punctured curves in ellipsoids.
As usual, fix $\veca = (a_1,\dots,a_n)$ with rationally independent components.
In order to apply the results of \S\ref{sec:J_std}, we work with the fixed integrable almost complex structures on $\R \times \bdy E(\veca)$ and $\wh{E}(\veca)$ as defined in Definition~\ref{def:J_E_etc}.
We will typically use the shorthands $J_E = J_{E(\veca)}$ and $J_{\bdy E} = J_{\bdy E(\veca)}$ unless we wish to explicitly emphase the dependence on $\veca$.

\sss

Fix $i_1,\dots,i_k \in \Z_{\geq 1}$ for some $k \in \Z_{\geq 1}$. Recall that $\calM^{J_E}_{E(\veca)}(\orb_{i_1}^\veca,\dots,\orb_{i_k}^\veca)\lll p_0\rrr$ denotes the moduli space of $J_E$-holomorphic maps $u: \Sigma \ra \wh{E}(\veca)$, with $\Sigma$ a Riemann sphere with $1$ marked point mapping to $p_0$ and $k$ ordered punctures positively asymptotic to $\orb^\veca_{i_1},\dots,\orb^\veca_{i_k}$ respectively.
We denote by $\ovl{\calM}^{J_E}_{E(\veca)}(\orb_{i_1}^\veca,\dots,\orb_{i_k}^\veca)\lll p_0\rrr$ its SFT-type compactification, with all symplectization levels considered modulo $\C^*$ (c.f. the discussion in \S\ref{subsubsec:aug_from_des}).
Our goal is to show that the stationary descendant invariant
\begin{align*}
N_{E(\veca)}(\orb^\veca_{i_1},\dots,\orb^\veca_{i_k})\lll \psi^{i_1+\cdots+i_k+k-2} \pt\rrr = \int\limits_{[\ovl{\calM}^{J_E}_{E(\veca)}(\orb_{i_1}^\veca,\dots,\orb_{i_k}^\veca)\lll p_0\rrr]}\psi^{i_1+\cdots+i_k+k-2}
\end{align*}
is equal to 
$\frac{1}{(\Ga^\veca_{i_1}+\cdots+\Ga^\veca_{i_k})!}$.

\begin{notation}
For $i = 1,\dots,n$, let $\Ddiv_i$ denote the divisor $\{z_i = 0\} \subset \C^n$. By slight abuse of notation, we also denote its image under $(Q_\veca^+)^{-1}: \C^n \ra \wh{E}(\veca)$ by $\Ddiv_i$. 
\end{notation}

Put $(d_1,\dots,d_n) := \sum_{s=1}^k\Ga^\veca_{i_s}$.
Note that we have $\sum_{s=1}^n d_s = \sum_{s=1}^k i_s$.
We define the moduli space $\calM_{E(\veca)}^{J_E}(\orb^\veca_{i_1},\dots,\orb^\veca_{i_k})\lll p_0,\Ddiv^{\times d_1},\dots,\Ddiv^{\times d_n} \rrr$ in the same way as $\calM_{E(\veca)}^{J_E}(\orb^\veca_{i_1},\dots,\orb^\veca_{i_k})\lll p_0 \rrr$, but with curves now equipped with $d_1+\cdots+d_n$ additional ordered marked points, such that the first $d_1$ map to $\Ddiv_1$, the next $d_2$ map to $\Ddiv_2$, and so on.
This also has SFT-type compactification by pseudoholomorphic buildings which we denote
$\ovl{\calM}_{E(\veca)}^{J_E}(\orb^\veca_{i_1},\dots,\orb^\veca_{i_k})\lll p_0,\Ddiv^{\times d_1},\dots,\Ddiv^{\times d_n} \rrr$, where again the symplectization level are taken modulo $\C^*$ (note the extra constraints are still well-defined since the divisors $\Ddiv_1,\dots,\Ddiv_n$ are preserved by the $\C^*$ action).

Analogous to \eqref{eq:closed_des_diag}, 
we have the following diagram of maps:
\begin{equation}\label{eq:punc_des_diag}
\begin{tikzcd}
  \ovl{\calM}_{E(\veca)}^{J_E}(\orb^\veca_{i_1},\dots,\orb^\veca_{i_k})\lll p_0\rrr & \ovl{\calM}_{E(\veca)}^{J_E}(\orb^\veca_{i_1},\dots,\orb^\veca_{i_k})\lll p_0,\Ddiv_1^{\times d_1},\dots,\Ddiv_n^{\times d_n}\rrr & \ovl{\calM}_{0,1+k+\sum_{i=1}^n d_i}
  \arrow[from=1-2, to=1-1]
  \arrow[from=1-2, to=1-3]
\end{tikzcd}.
\end{equation}
Here the first map is given by forgetting the extra marked points and contracting any unstable components. 
Similarly, the second map remembers only the domain, again contracting any unstable components (in particular it forgets the decomposition into building levels).

Our goal is to adapt the proof given in \S\ref{subsubsec:closed_curves} for the closed case, with \eqref{eq:punc_des_diag} taking the place of \eqref{eq:closed_des_diag}, taking note of the following features in the punctured case:
\begin{itemize}
  \item we have fixed Reeb orbit asymptotics $\orb^\veca_{i_1},\dots,\orb^\veca_{i_k}$ instead of a fixed homology class $A \in H_2(M)$
  \item the divisors $\Ddiv_1,\dots,\Ddiv_n$ are now noncompact 
  \item moduli spaces of punctured curves in $\wh{E}(\veca)$ are compactified via pseudoholomorphic buildings.
\end{itemize}
The main first point above requires the most consideration, while the second and third do not pose any particular new difficulties.

\subsubsection{Component of maximal dimension}

Using the results of \S\ref{sec:J_std}, the compactified moduli space 
$\ovl{\calM}_{E(\veca)}^{J_E}(\orb^\veca_{i_1},\dots,\orb^\veca_{i_k})\lll p_0 \rrr$ can be entirely understood in terms of proper $i$-holomorphic curves in $\C^n$ and $\C^n \setm \{\vec{0}\}$ respectively.
Indeed, recall that elements of $\ovl{\calM}_{E(\veca)}^{J_E}(\orb^\veca_{i_1},\dots,\orb^\veca_{i_k})\lll p_0 \rrr$ are pseudoholomorphic buildings consisting of a $J_E$-holomorphic level in $\wh{E}(\veca)$ and some number (possibly zero) of $J_{\bdy E}$-holomorphic levels in $\R \times \bdy E(\veca)$.
Postcomposing with $Q_\veca^+: \wh{E}(\veca) \ra \C^n$ sets up a correspondence between asymptotically cylindrical curves in $\wh{E}(\veca)$ and proper $i$-holomorphic curves in $\C^n$, while postcomposing with $\Phi_\veca: \R \times \bdy E(\veca) \ra \C^n \setm \{\vec{0}\}$ sets up a correspondence between asymptotically cylindrical curves in $\R \times \bdy E(\veca)$ and proper $i$-holomorphic curves in $\C^n \setm \{\vec{0}\}$.
Any proper $i$-holomorphic curve in $\C^n$ or $\C^n \setm \{\vec{0}\}$ which is not entirely contained in $\Ddiv_1 \cup \cdots \cup \Ddiv_n$ (e.g. this is automatic if it passes through $\pn \in (\C^*)^n$) can be equally viewed as a proper $i$-holomorphic map to $(\C^*)^n$, by puncturing the intersection points with $\Ddiv_1 \cup \cdots \cup \Ddiv_n$.

In turn, proper $i$-holomorphic curves in $(\C^*)^n$ of a given toric degree can be fruitfully studied via Lemma~\ref{lem:Mdelta_to_DM}.
In particular, one can check that all of the spaces in \eqref{eq:closed_des_diag} are thin compactifications.

\sss

In order to analyze the moduli space $\calM_{E(\veca)}^{J_E}(\orb^\veca_{i_1},\dots,\orb^\veca_{i_k})\lll p_0\rrr$, it is natural to consider all ordered toric degrees $\bd$ which correspond to curves in $\wh{E}(\veca)$ with fixed asymptotics $\orb^\veca_{i_1},\dots,\orb^\veca_{i_k}$.
Recall that each tuple $\vecv \in \Z_{\geq 0} \setm \Z^n_{\leq 0}$ corresponds to a Reeb orbit $\orb^\veca_{\vecv}$ in $\bdy E(\veca)$.
\begin{definition}
Let $\Delta(\orb^\veca_{i_1},\dots,\orb^\veca_{i_k})$ denote the set of ordered toric degrees $\bd = (\vecv_1,\dots,\vecv_\vk)$ such that
\begin{itemize}
  \item $\vecv_s \in \Z^n \setm \Z_{\leq 0}^n$ with $\orb^\veca_{\vecv_s} = \orb^\veca_{i_s}$ for $s = 1,\dots,k$
  \item $\vecv_s \in \Z^n_{\leq 0}$ for $s = k+1,\dots,\vk$.
\end{itemize}
\end{definition}

In the following, we assume $p_0 \in \wh{E}(\veca)$ is the point mapping to $\pn$ under $Q_\veca^+: \wh{E}(\veca) \ra \C^n$.
There is a natural map
\begin{align*}
\Theta: \bigcup_{\bd \in \Delta(\orb^\veca_{i_1},\dots,\orb^\veca_{i_k})} \calM^i_{(\C^*)^n}(\bd)\lll \pn\rrr \longrightarrow \calM^{J_E}_{E(\veca)}(\orb^\veca_{i_1},\dots,\orb^\veca_{i_k})\lll p_0 \rrr,
\end{align*}
defined as follows.
Let $u: \Sigma \ra (\C^*)^n$ represent an element in $\calM^i_{(\C^*)^n}(\bd)\lll \pn\rrr$.
Since each puncture corresponding to a homology class in $\Z_{\leq 0}^n$ is a removable singularity, we get a unique extension to a proper $i$-holomorphic map $u': \Sigma' \rightarrow \C^n$, where $\Sigma'$ is the result after filling in the last $\vk - k$ punctures of $\Sigma$. 
We then put $\Theta(u) := (Q_\veca^+)^{-1} \circ u'$.

For $\bd \in \Delta(\orb^\veca_{i_1},\dots,\orb^\veca_{i_k})$, we denote the image of $\calM^i_{(\C^*)^n}(\bd)\lll \pn\rrr$ under $\Theta$ by $\im(\Theta_\bd)$.
For $\bd = (\vecv_1,\dots,\vecv_\vk)$, $\im(\Theta_\bd)$ consists of those curves $C$ in $\calM^{J_E}_{E(\veca)}(\orb^\veca_{i_1},\dots,\orb^\veca_{i_k})\lll p_0 \rrr$ 
such that:
\begin{itemize}
  \item for $s = 1,\dots,k$, the linking numbers of the $s$th end of $C$ with $\Ddiv_1,\dots,\Ddiv_n$ are given by the respective components of $\vecv_s$
  \item $C$ intersects $\Ddiv_1\cup \cdots \cup \Ddiv_n$ in precisely $\vk-k$ interior points, and, for $s = k+1,\dots,\vk$, the contact orders with $\Ddiv_1,\dots,\Ddiv_n$ at the $s$th such point (for some ordering) are given by the respective components of $\vecv_s$.
\end{itemize}

It turns out that there is a unique choice of ordered toric degree $\bdmax$ for which $\im(\Theta_{\bd})$ has the maximal dimension.
Indeed, the dimension of $\im(\Theta_\bd)$ is given by
\begin{align*}
\tfrac{1}{2}\dim_\R \im(\Theta_\bd) = \tfrac{1}{2}\dim_\R\calM^i_{(\C^*)^n}(\bd)\lll \pn\rrr = \vk-2.
\end{align*}
Put 
$$\bdmax := (\Gamma^\veca_{i_1},\dots,\Gamma^\veca_{i_k},\underbrace{-e_1,\dots,-e_1}_{d_1},\dots,\underbrace{-e_n,\dots,-e_n}_{d_n}).$$
Note that this is a valid ordered toric degree since $(d_1,\dots,d_m) = \sum_{s=1}^k \Ga^\veca_{i_s}$.
\begin{lemma}\label{lem:bdmax_has_max_dim}
For $\bd \in \Delta(\orb^\veca_{i_1},\dots,\orb^\veca_{i_k})$, we have $\tfrac{1}{2}\dim_\R \im(\Theta_{\bd}) \leq \sum_{s=1}^k {i_s} + k - 2$, with equality if and only if $\bd = \bdmax$ up to permutation.
\end{lemma}
\begin{proof}
  Consider an ordered toric degree $\bd = (\vecv_1,\dots,\vecv_\vk)$ which satisfies $\orb^\veca_{\vecv_s} = \orb^\veca_{i_s}$ for $s = 1,\dots,k$.
Using Lemma~\ref{lem:Ga_maximal}, it is easy to see that $\vk$ is uniquely maximized by the tuple $\bdmax$, whose length is $\sum_{s=1}^k i_s + k$.
\end{proof}

\subsubsection{Completing the proof}

\begin{proof}[Proof of Theorem~\ref{thm:main_descendant}]

It remains to compute the degrees of the maps in \eqref{eq:punc_des_diag}.
Since each of the spaces involved is a thin compactification, we can focus on the maps between open strata and count points in the fiber over a generic point.

We claim that the right map in \eqref{eq:punc_des_diag} has degree $1$.
To see this, note that, similar to Corollary~\ref{cor:right_map_diffeo}, the map on open strata can be written as a composition
\begin{equation}\label{eq:punc_right_map_comp}
  \begin{tikzcd}
  \calM_{E(\veca)}^{J_E}(\orb^\veca_{i_1},\dots,\orb^\veca_{i_k})\lll p_0,\Ddiv_1^{\times d_1},\dots,\Ddiv_n^{\times d_n}\rrr & \calM_{(\C^*)^n}^i(\bdmax)\lll \pn\rrr & \calM_{0,1+k+\sum_{i=1}^n d_i}
  \arrow[from=1-1, to=1-2]
  \arrow[from=1-2, to=1-3]  
  \end{tikzcd}.
\end{equation}
Here the first map postcomposes $u: \Sigma \ra \wh{E}(\veca)$ with $Q_\veca^+$ and then punctures each intersection point with $\Ddiv_1 \cup \cdots \cup \Ddiv_n$, and this is a diffeomorphism, the proof being similar to that of Lemma~\ref{lem:delta_A}.
The second map in \eqref{eq:punc_right_map_comp} is a diffeomorphism by Lemma~\ref{lem:Mdelta_to_DM}.
Therefore the composition \eqref{eq:punc_right_map_comp} is also a diffeomorphism, which implies the claim.

As for the left map in \eqref{eq:punc_des_diag}, we claim that it has degree $d_1!\cdots d_n!$.
Indeed, the map on open strata takes the form
\begin{align}\label{eq:punc_left_map_open}
\calM_{E(\veca)}^{J_E}(\orb^\veca_{i_1},\dots,\orb^\veca_{i_k})\lll p_0,\Ddiv_1^{\times d_1},\dots,\Ddiv_n^{\times d_n}\rrr \longrightarrow \calM_{E(\veca)}^{J_E}(\orb^\veca_{i_1},\dots,\orb^\veca_{i_k})\lll p_0\rrr.
\end{align}
By Lemma~\ref{lem:bdmax_has_max_dim}, a generic element $C$ in $\calM_{E(\veca)}^{J_E}(\orb^\veca_{i_1},\dots,\orb^\veca_{i_k})\lll p_0\rrr$ 
lies in $\im(\Theta_{\bdmax})$, 
and in particular has $d_i$ simple intersections  with $\Ddiv_i$ for $i = 1,\dots, n$.
This means the number of preimages of $C$ under the map \eqref{eq:punc_left_map_open} is $d_1!\cdots d_n!$, coming from the ways of ordering these intersection points.

Finally, the diagram \eqref{eq:punc_des_diag} is compatible with $\psi$ classes, by essentially the same proof as Lemma~\ref{lem:psi_compat}.
Then we have
\begin{align*}
(d_1!\cdots d_n!)\int\limits_{[\ovl{\calM}^{J_E}_{E(\veca)}(\orb_{i_1}^\veca,\dots,\orb_{i_k}^\veca)\lll p_0\rrr]}\psi^{i_1+\cdots+i_k+k-2} &= \int\limits_{[\ovl{\calM}^{J_E}_{E(\veca)}(\orb_{i_1}^\veca,\dots,\orb_{i_k}^\veca)\lll p_0,\Ddiv_1^{\times d_1},\dots,\Ddiv_n^{\times d_n}\rrr]} \psi^{i_1+\cdots+i_k+k-2}\\ &= \int\limits_{[\ovl{\calM}_{0,1+k+\sum_{s=1}^n i_s}]}\psi^{i_1+\cdots+i_k+k-2},
\end{align*}
i.e. 
\begin{align*}
N_{E(\veca)}(\orb^\veca_{i_1},\dots,\orb^\veca_{i_k})\lll \psi^{i_1+\cdots+i_k+k-2} \pt \rrr &= 
(d_1!\cdots d_n!)^{-1} = \frac{1}{(\Ga^\veca_{i_1} + \cdots + \Ga^\veca_{i_k})!},
\end{align*}
where in the last line we used Lemma~\ref{lem:DM_psi_comp}.

\end{proof}

\subsection{Alternative approach via fully rounding}\label{sec:fully_rounding}

In this subsection, which is logically independent from the rest of the paper, we compare Theorem~\ref{thm:main_descendant} with the approach of \cite{chscI}.
We find that a weakened version of Theorem~\ref{thm:main_descendant} can be naturally deduced from the results of \cite{chscI}, albeit without any interpretation in terms of gravitational descendants.
In fact, although the geometric setup in \cite{chscI} is a priori rather different from the present paper, the enumerative implications appear to be closely related.

More precisely, we will use the fully rounding formalism of \cite{chscI} to deduce that the $\Li$ augmentation $\aug_a: C_a \ra C_o$ {\em defined} by 
\begin{align}\label{eq:abstr_aug_def}
\aug_a^k(\orb_{i_1}^a,\dots,\orb_{i_k}^a) := \frac{q_{i_1+\cdots+i_k+k-1}}{(\sum_{s=1}^k \Ga^a_{i_s})!}   
\end{align}
satisfies compatibility equation ~\eqref{eq:aug_Xi_compatibility},
namely $\aug_a \circ \Xi^{a'}_a = \aug_{a'}$ for all $a,a' \in \R_{> 1}$.
 This is nontrivial, since a priori an abstract definition of $\aug_a$ would not enjoy any compatibility with the $\Li$ homomorphisms $\Xi^{a'}_a: C_{a'} \ra C_a$.
For concreteness we restrict the discussion to dimension four. 

  The starting point of \cite{chscI} is to $C^0$-perturb the ellipsoid $E(1,a)$ to its ``fully rounded'' counterpart $\wt{E}(1,a)$ (see \cite[Fig. 5.1]{chscI}), such that the periodic Reeb orbits in $\wt{E}(1,a)$ depend rather transparently on the parameter $a$.
For all $a \in \R_{> 1}$, the $\Li$ algebra $\wt{C}_a := \chlin(\wt{E}(1,a))$ is identified with an $\Li$ algebra $V$, which by definition has basis elements $\alpha_{i,j}$ for $(i,j) \in \Z_{\geq 1}^2$ with $|\alpha_{i,j}| = -1-2i-2j$ and $\beta_{i,j}$ for $(i,j) \in \Z_{\geq 0}^2 \setm \{(0,0)\}$ with $|\beta_{i,j}| = -2-2i-2j$, and with $\Li$ operations given by
\begin{itemize}
  \item $\ell_V^1(\alpha_{i,j}) = j\beta_{i-1,j} - i\beta_{i,j-1}$
  \item $\ell_V^1(\beta_{i,j}) = 0$
  \item $\ell_V^2(\alpha_{i,j},\alpha_{k,l}) = (il - jk)\alpha_{i+k,j+l}$
  \item $\ell_V^2(\alpha_{i,j},\beta_{k,l}) = \ell_V^2(\beta_{k,l},\alpha_{i,j}) = (il-jk)\beta_{i+k,j+l}$, 
  \item $\ell_V^2(\beta_{i,j},\beta_{k,l}) = 0$
  \item $\ell_V^k$ is trivial for $k \geq 3$.
\end{itemize}
There are natural inverse $\Li$ homotopy equivalences $\Phi_a: \wt{C}_a \ra C_a$ and $\Psi_a: C_a \ra \wt{C}_a$ and $\Li$ cobordism maps $\wt{\Xi}^{a'}_a$ which for all $a,a' \in \R_{> 1}$ fit together into a commutative (up to $\Li$ homotopy) diagram
\begin{equation}\label{cd:chsc}
\begin{tikzcd}
  C_{a'} & \wt{C}_{a'} & & \\
  & & V & C_o\\
  C_a & \wt{C}_a & &
  \arrow[swap,"\Psi_{a'}",bend right=20,from=1-1,to=1-2]
  \arrow[swap,"\Phi_{a'}",bend right=20,to=1-1,from=1-2]  
  \arrow[swap,"\Psi_a",bend right=20,from=3-1,to=3-2]
  \arrow[swap,"\Phi_a",bend right=20,to=3-1,from=3-2]  
  \arrow[equal,from=1-2,to=2-3]
  \arrow[equal,from=3-2,to=2-3]  
  \arrow["\wt{\aug}",from=2-3,to=2-4]    
  \arrow["\Xi^{a'}_a",swap,from=1-1,to=3-1]
  \arrow["\wt{\Xi}^{a'}_a",from=1-2,to=3-2] . 
\end{tikzcd}
\end{equation}
In our notation, the map $\Psi^1_a: C_a \ra \wt{C}_a$ sends $\orb_j^a$ to $\beta_{\Gamma^a_j}$ for $j \in \Z_{\geq 1}$, and $\Psi^{k \geq 2}_a$ vanishes on all inputs.
Meanwhile, $\Phi_a$ depends on $a$ in a quite complicated way (a recursive formula is given in \cite{chscI}).

The $\Li$ augmentation $\wt{\aug}$ appearing in \eqref{cd:chsc} is defined by the following proposition, which is our main observation:
\begin{prop}\label{prop:aug_is_aug}
  Define maps $\wt{\aug}^{\,k}: V^{\odot k} \ra C_o$ for $k \in \Z_{\geq 1}$ by:
  \begin{itemize}
     \item 
 $\wt{\aug}^{\,k}(\beta_{i_1,j_1},\dots,\beta_{i_k,j_k}) = \frac{q_{i_1+\cdots+i_k+k-1}}{(i_1+\cdots+i_k)!(j_1+\cdots+j_k)!}$ for any $(i_1,j_1),\dots,(i_k,j_k) \in \Z_{\geq 0}^2 \setm \{(0,0)\}$
     \item $\wt{\aug}^{\,k}$ vanishes on all other inputs (i.e. those involving at least one $\alpha$ basis element).
   \end{itemize}
   Then this gives a valid $\Li$ homomorphism.
\end{prop}
\begin{cor}
The $\Li$ augmentation $\aug_a: C_a \ra C_o$ defined by ~\eqref{eq:abstr_aug_def} satisfies $\aug_a \circ \Xi^{a'}_a = \aug_{a'}$ for all $a,a' \in \R_{> 1}$.
\end{cor}
\begin{proof}
 This follows immediately from commutativity of the diagram \eqref{cd:chsc} and the observation that the composition 
 \[
 \begin{tikzcd}
   C_a & \wt{C}_a & V & C_o
   \arrow["\Psi_a",from=1-1,to=1-2]
   \arrow[equal,from=1-2,to=1-3]   
   \arrow["\wt{\aug}",from=1-3,to=1-4]      
 \end{tikzcd}
 \]
agrees with \eqref{eq:abstr_aug_def}.
\end{proof}

\begin{proof}[Proof of Proposition~\ref{prop:aug_is_aug}]
The $\Li$ homomorphism equations for $\wt{\aug}: V \ra C_o$ are equivalent to the statement that induced map on bar complexes $\wh{\wt{\aug}}: \bar V \ra \bar C_o$ vanishes on the image of $\wh{\ell}_V: \bar V \ra \bar V$.
It suffices to show that $\wh{\wt{\aug}}(\wh{\ell}_V(v_1 \odot \cdots \odot v_k)) = 0 \in \bar C_o$ for any basis elements $v_1,\dots,v_k \in V$.
Note that if $v_1,\dots,v_k$ are all $\beta$ basis elements, then $\wh{\ell}_V(v_1\odot \cdots \odot v_k) = 0$.
Moreover, if two or more of $v_1,\dots,v_k$ are $\alpha$ basis elements, then each summand of $\wh{\ell}_V(v_1\odot \cdots \odot v_k)$ involves at least one $\alpha$ basis element, and hence $\wh{\wt{\aug}}(\wh{\ell}_V(v_1 \odot \cdots \odot v_k)) = 0$.
Therefore it suffices to consider $v_1 \odot \cdots \odot v_k$ of the form
$\alpha_{i_1,j_1}\odot \beta_{i_2,j_2}\odot \cdots \odot \beta_{i_k,j_k}$.
Observe that $\wh{\ell}_V(\alpha_{i_1,j_1}\odot \beta_{i_2,j_2}\odot \cdots \odot \beta_{i_k,j_k})$ is given by
\begin{align*}
(j_1)\left(\beta_{i_1-1,j_1} \odot \beta_{i_2,j_2}\odot \cdots \odot \beta_{i_k,j_k}\right) - (i_1)\left(\beta_{i_1,j_1-1} \odot \beta_{i_2,j_2}\odot \cdots \odot \beta_{i_k,j_k}\right) 
\\ + \sum_{s=2}^k (i_1j_s - j_1i_s)\left(\beta_{i_1+i_s,j_1+j_s} \odot \beta_{i_2,j_2} \odot \cdots \odot \beta_{i_{s-1},j_{s-1}} \odot \beta_{i_{s+1},j_{s+1}}\odot \cdots \odot \beta_{i_k,j_k}\right).
\end{align*}
It is enough to establish $(\pi_1 \circ \wh{\wt{\aug}})(\alpha_{i_1,j_1}\odot \beta_{i_2,j_2}\odot \cdots \odot \beta_{i_k,j_k}) = 0$,
where $\pi_1: \ovl{S}C_o \ra C_o$ is projection onto tensors of word length one.
Finally, we compute $(\pi_1 \circ \wh{\wt{\aug}})(\alpha_{i_1,j_1}\odot \beta_{i_2,j_2}\odot \cdots \odot \beta_{i_k,j_k})$ to be 
\begin{align*}
\frac{j_1}{(i_1+\cdots+i_k-1)!(j_1+\cdots+j_k)!} - \frac{i_1}{(i_1+\cdots+i_k)!(j_1+\cdots+j_k-1)!} + \sum_{s=2}^k \frac{i_1j_s - j_1i_s}{(i_1+\cdots+i_k)!(j_1+\cdots+j_k)!},
\end{align*}
or equivalently
\begin{align*}
\frac{j_1(i_1+\cdots+i_k) - i_1(j_1+\cdots+j_k) + \sum_{s=2}^k (i_1j_s - j_1i_s)}{(i_1+\cdots+i_k)!(j_1+\cdots+j_k)!},  
\end{align*}
which is zero.

\end{proof}

\begin{rmk}
  It seems worth emphasizing that while the proof of Proposition~\ref{prop:aug_is_aug} is elementary and does not require any input from the present paper, formulating the statement itself is rather opaque without knowledge of Theorem~\ref{thm:main_descendant}.
\end{rmk}

\section{Derivation of recursive formula}\label{sec:formulas}

In this section, we combine the results of \S\ref{sec:descendants} with the formalism of \S\ref{sec:Li} in order to derive a recursive formula for $\wtTcount_{M,A}^\veca$ (and hence $\Tcount_{M,A}^\veca$).
The derivation is based on based on the relationship between closed curve descendants in $M$ and punctured curve descendants in ellipsoids.

Our goal is to prove:
\begin{thm}\label{thm:main_recursion}
For any closed symplectic manifold $M$ and homology class $A \in H_2(M)$, we have:
  \begin{align}\label{eq:C_rec_T_tilde}
\wtTcount_{M,A}^\veca = \left(\Ga^\veca_{c_1(A)-1}\right)!\left(N_{M,A}\lll \psi^{c_1(A)-2}\pt\rrr - \sum_{\substack{k \geq 2\\A_1,\dots,A_k \in H_2(M)\\A_1 + \cdots + A_k = A}}  \frac{\wtTcount_{M,A_1}^\veca \cdot \cdots \cdot \wtTcount_{M,A_k}^\veca}{k!(\sum_{s=1}^k \Ga^\veca_{c_1(A_i)-1} )!}\right).
\end{align}
\end{thm}
\NI We note that the right hand side of \eqref{eq:C_rec_T_tilde} is actually a finite sum for the same reasons discussed in \S\ref{subsubsec:MC_elts}.

If $A$ is a homology class for which the sum on the right hand is trivial (e.g. if $c_1(A) = 2$), then ~\eqref{eq:C_rec_T_tilde} gives simply
\begin{align*}
\wt{T}^\veca_{M,A} = \left(\Ga^\veca_{c_1(A)-1}\right)! \cdot N_{M,A}\lll \psi^{c_1(A)-2}\pt \rrr,
\end{align*}
which is the base case of the recursion.

\sss

In the case $M = \CP^2$ with $A = d[L]$ and $\veca = (1,a)$, we put
$\wtTcount_{d}^a := \wtTcount^{(1,a)}_{\CP^2,d[L]}$ and $\Tcount_d^a := \tfrac{1}{\mult(\orb^a_{3d-1})}\wtTcount^a_d$, and \eqref{eq:C_rec_T_tilde} becomes:
\begin{cor}\label{cor:rec_CP2}
For any $d \in \Z_{\geq 1}$ we have:
\begin{align}\label{eq:wtTcount}
\wtTcount_d^a = \left(\Ga^a_{3d-1}\right)!\left( (d!)^{-3} - \sum_{\substack{k \geq 2\\d_1 \geq \cdots \geq d_k \geq 1\\d_1 + \cdots + d_k = d}}  \frac{\wtTcount_{d_1}^a \cdot \cdots \cdot \wtTcount_{d_k}^a}{|\aut(d_1,\dots,d_k)|(\sum_{s=1}^k \Ga^a_{3d_i-1} )!}\right).
\end{align}
\end{cor}
\NI Note that as the base case we have 
\begin{align*}
\wtTcount_1^a = \left(\Ga^a_{2}\right)! = 
\begin{cases}
  1 & \text{if}\;\;\; 1 \leq a < 2 \\
  2 & \text{if}\;\;\; a \geq 2.
\end{cases}
\end{align*}

Further specializing to the case $a \gg 1$, we define $\Tcount_d := \Tcount_d^\infty$ to be $\Tcount_d^a$ for $a$ sufficiently large, and similarly for $\wtTcount_d$.
Note that we have $\Tcount_d = \tfrac{1}{\mult(\orb^a_{3d-1})}\wtTcount_d = \tfrac{1}{3d-1} \wtTcount_d$, and $\Gamma^a_{3d-1} = (3d-1,0)$. Then the recursion in Corollary~\ref{cor:rec_CP2} becomes:
\begin{cor}\label{cor:wtT_d}
\begin{align}\label{eq:C_recursion_for_T_d}
\wtTcount_d = (3d-1)!\left( (d!)^{-3} - \sum_{\substack{k \geq 2\\d_1 \geq \cdots \geq d_k \geq 1\\d_1 + \cdots + d_k = d}}  \frac{\wtTcount_{d_1} \cdot \cdots \cdot \wtTcount_{d_k}}{|\aut(d_1,\dots,d_k)|(3d-k )!}\right).
\end{align}
\end{cor}
\begin{rmk}
 Corollary~\ref{cor:wtT_d} readily recovers the values $\Tcount_1 = 1$, $\Tcount_2 = 1$, $\Tcount_3 = 4$, $\Tcount_4 = 26$, $\Tcount_5 = 217$ and so on which were first computed in \cite{McDuffSiegel_counting}.  
\end{rmk}

\sss

\begin{proof}[Proof of Theorem~\ref{thm:main_recursion}]
Recall that $C_o$ is the abelian $\Li$ algebra with linear generators $\dg_i$ for $i \in \Z_{\geq 1}$.
Let $\pi_1: \ovl{S}C_o \ra C_o$ denote the projection map to tensors of word length one.
Applying $\pi_1$ to both sides of \eqref{eq:aug_hat_exp_m} gives
\begin{align*}
(\pi_1 \circ \wh{\aug}_\veca)(\exp_A(\mc_{M}^\veca)) = \mc^o_{M,A},
\end{align*}
i.e. 
\begin{align*}
\sum_{\substack{k \geq 1\\A_1,\dots,A_k \in H_2(M)\\A_1 + \cdots + A_k = A}} \tfrac{1}{k!}\aug_\veca^k(\mc^\veca_{M,A_1},\dots,\mc^\veca_{M,A_k}) = N_{M,A}\lll \psi^{c_1(A)-2}\pt\rrr \dg_{c_1(A)-1}.
\end{align*}
The left hand side of the above equation can be written as
\begin{align*}
\sum_{\substack{k \geq 1\\A_1,\dots,A_k \in H_2(M)\\A_1 + \cdots + A_k = A}} \tfrac{1}{k!} \left(\prod_{s=1}^k \wtTcount_{M,A_s}^\veca \right)(\aug_\veca)^k(\orb^\veca_{c_1(A_1)-1},\dots,\orb^\veca_{c_1(A_k)-1}),
\end{align*}
and using Theorem~\ref{thm:main_descendant} we can rewrite this as
\begin{align*}
\sum_{\substack{k \geq 1\\A_1,\dots,A_k \in H_2(M)\\A_1 + \cdots + A_k = A}} \tfrac{1}{k!} \left(\prod_{s=1}^k \wtTcount_{M,A_s}^\veca \right) \frac{\dg_{c_1(A)-1}}{(\sum_{s=1}^k \Ga^\veca_{c_1(A_s)-1} )!},
\end{align*}
i.e. we have 
\begin{align*}
\sum_{\substack{k \geq 1\\A_1,\dots,A_k \in H_2(M)\\A_1 + \cdots + A_k = A}}  \frac{\prod_{s=1}^k \wtTcount_{M,A_s}^\veca }{k!(\sum_{s=1}^k \Ga^\veca_{c_1(A_i)-1} )!}  =  N_{M,A}\lll \psi^{c_1(A)-2}\pt\rrr.
\end{align*}
Peeling off the term corresponding to $k=1$ then gives ~\eqref{eq:C_rec_T_tilde}.

\end{proof}

\begin{rmk} \hfill
\begin{enumerate}

\item It was shown in \cite[Cor. 2.3.9]{McDuffSiegel_counting} by a geometric argument that $\Tcount_d$ is a positive integer for all $d \in \Z_{\geq 1}$. It is interesting to ask whether this can be deduced from \eqref{eq:C_rec_T_tilde} by purely combinatorial methods.

\item The following observation is due to Qiuyu Ren based on \eqref{eq:C_recursion_for_T_d}. Define the generating function $F(x) = 1 + \sum_{d=1}^\infty \wtTcount_d x^d$.
Then the $x^d$ coefficient of $F(x)^{3d}$ is $\tfrac{(3d)!}{(d!)^3}$.
Moreover, the sequence $\wtTcount_1,\wtTcount_2,\wtTcount_3,\dots$ appears as A353195 in the Online Encyclopedia of Integer Sequences (see \url{https://oeis.org/A353195}).

\end{enumerate}
\end{rmk}

\section{Infinitesimal symplectic cobordisms}\label{sec:inf_cobs}

\subsection{Jump phenomenon}

The results in \S\ref{sec:formulas} allow for explicit computation of $\Tcount_{M,A}^\veca$ for any closed symplectic manifold $M$, homology class $A \in H_2(M)$, and vector $\veca \in \R_{>0}^n$.
However, due to the presence of terms of both positive and negative sign, it is still nontrivial to extract large scale information from these formulas.
In particular, it is an interesting open question whether Conjecture~\ref{conj:p_q_d} can be directly extracted from this formalism. 

We now turn our attention to studying how $\Tcount_{M,A}^\veca$ behaves as a function of $\veca$. For concreteness we restrict the discussion to the case $M = \CP^2$, leaving the adaptation to general $M$ to the interested reader.
As a starting point, computer experiments suggest the following:
\begin{conjecture}\label{conj:T_d^a_nondecreasing}
 The count $\Tcount_{d}^a \in \R$ is nondecreasing as a function of $a \in \R_{>1}$. 
\end{conjecture}
\NI As we observe below, verifying this conjecture would suffice to prove Conjecture~\ref{conj:p_q_d}.

Remarkably, $\Tcount_d^a$ is a piecewise constant function of $a$, and it can only jump at certain specific values.
Indeed, inspecting Corollary~\ref{cor:rec_CP2}, we have:
\begin{cor}
The count $\wtTcount_d^a$ only depends on the parameter $a$ via the pairs
$\Ga^a_{2},\Ga^a_{5},\dots,\Ga^a_{3d-1}$.
\end{cor}
\NI In other words, if $\Ga^a_{3k-1} = \Ga^{a'}_{3k-1}$ for all $k = 1,\dots,d$, then we must have $\wtTcount^a_d = \wtTcount^{a'}_d$.

Observe that $\Ga^{a^-}_k \neq \Ga^{a^+}_k$ if and only if $\Ga^a_k$ lies on the line $L_a$ passing through $(0,-1)$ and $(a,0)$.
In particular, in this case $a$ must be rational, and we must have $\Ga^{a^-}_k = \Ga^{a^+}_k + (-1,1)$. This situation is illustrated Figure~\ref{fig:lattice_path_movie}.
\begin{figure}
    \centering
    \subfigure[]{\includegraphics[width=0.32\textwidth]{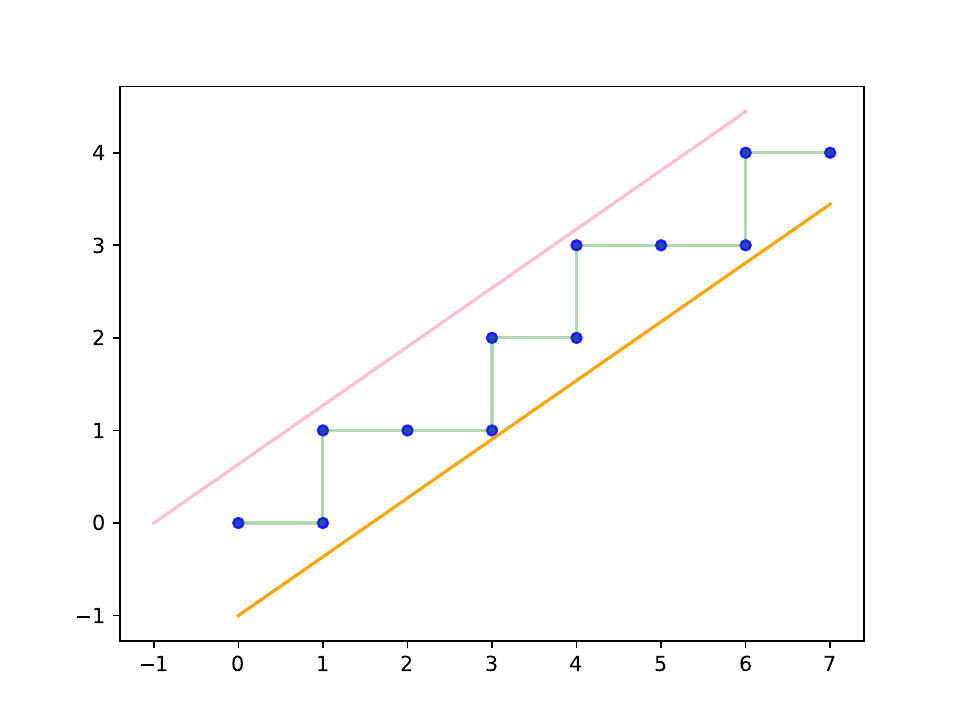}} 
    \subfigure[]{\includegraphics[width=0.32\textwidth]{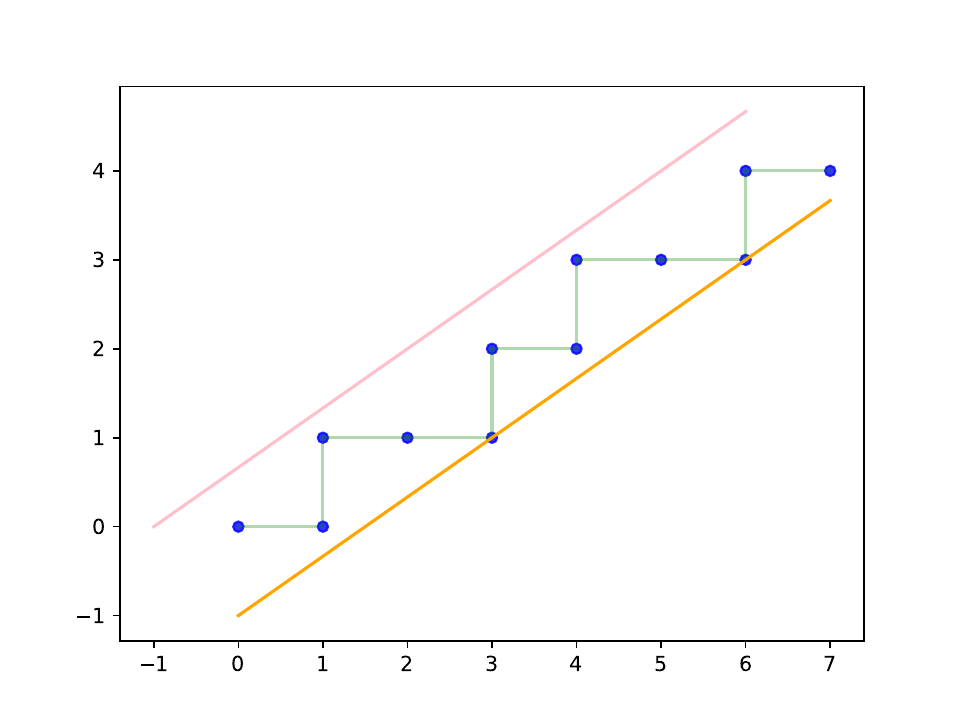}} 
    \subfigure[]{\includegraphics[width=0.32\textwidth]{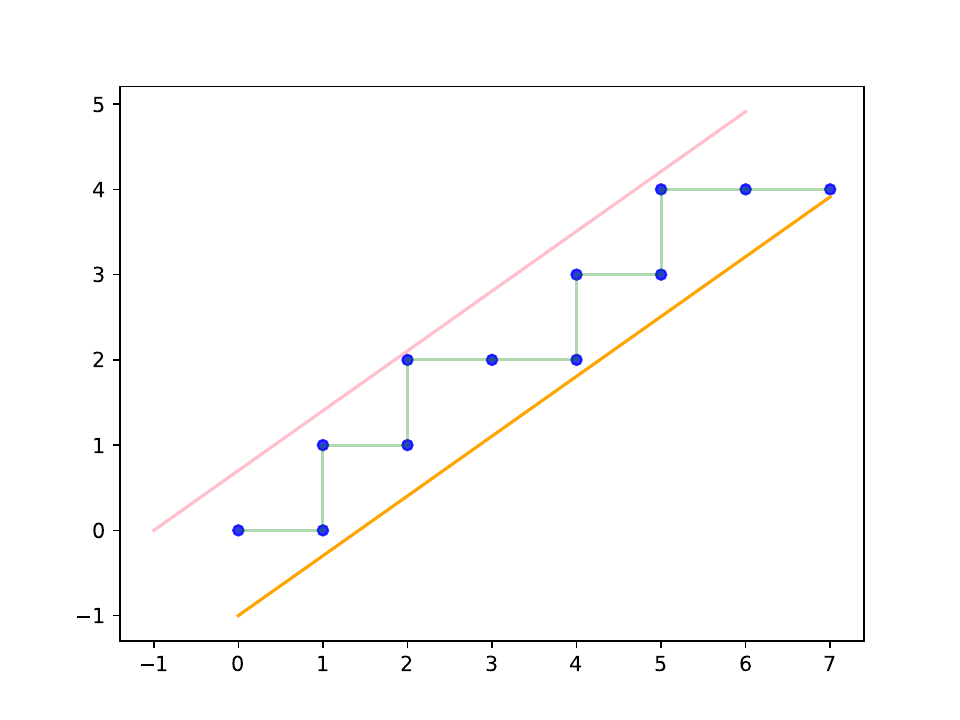}}
    \caption{The lattice paths $\Ga^{(3/2)^+}$, $\Ga^{3/2}$ and $\Ga^{(3/2)^-}$. Note that the lattice point $\Ga^a_4$ is forced to move up and to the left as $a$ passes from $(3/2)^+$ to $(3/2)^-$.}
    \label{fig:lattice_path_movie}
\end{figure}

More precisely, put $(i,j) = \Ga_k^{a}$. This point lies on $L_a$ if and only if $a = \frac{i}{j+1}$, and hence:
\begin{lemma}
 Given $k \in \Z_{\geq 1}$ and $a \in \R_{> 0}$, we have $\Ga_k^{a^-} = \Ga_k^{a^+}$ unless
\begin{align*}
 a \in \jump_k := \left\{\tfrac{i}{j+1}\;|\; i,j \in \Z_{\geq 0},\; i+j = k\right\} 
 = \left\{\tfrac{k}{1},\tfrac{k-1}{2},\,\dots,\tfrac{1}{k}\right\}.
\end{align*}
\end{lemma}

\begin{cor}\label{cor:jumps_of_wtTcount}
For $d \in \Z_{\geq 1}$, $\wtTcount_d^a$ is piecewise constant as a function of $a$, with discontinuity points (``jumps'') contained in   
$\bigcup\limits_{k=1}^d \jump_{3k-1}$.
\end{cor}
\begin{rmk}
  In fact, it is not difficult to show that $\wtTcount_d^a$ is constant for $a \in (1,2)$, i.e. there are no jumps in this interval.  
\end{rmk}

\begin{example}\label{ex:deg_5_jumps}
  According to the above discussion, the potential jumps of $\wtTcount_5^a$ with $a \geq 2$ are:
  \begin{align*}
   \left\{2,\tfrac{11}{4},3,\tfrac{7}{2},4,5,\tfrac{13}{2},8,11,14\right\}.
  \end{align*}
  In fact, we have
\begin{center}
\begin{tabular}{c|c|c|c|c|c|c} 
 $a \in $ & $(1,5)$ & $(5,13/2)$ & $(13/2,8)$ & $(8,11)$ & $(11,14)$ & $(14,\infty)$ \\ 
 \hline
 $\wtTcount_5^a$ & $0$ & $2$ & $13$ & $113$ & $217$ & $3038$
\end{tabular}.
\end{center}
Note that $\Tcount_5^{(13/2)^+} = \tfrac{1}{13}\cdot\wtTcount_5^{(13/2)^+} = 1$ and $\Tcount_5^{\infty} = \tfrac{1}{14}\cdot\wtTcount_5^{\infty} = 217$.
 \end{example}

\begin{rmk}
If $a = \tfrac{i}{j+1}$ with $i+j = k-1$, then $\Ga_k^{a^-} = \Ga_k^{a^+} = (i,j+1)$, but the interpretation of $\Ga_k$ as a Reeb orbit changes.
Namely, letting $\sht^m$ (resp. $\lng^m$) denote the $m$-fold cover of the short (resp. long) simple Reeb orbit in $\bdy E(1,a)$, we have $\orb_k^{a^+} = \sht^i$ and $\orb_k^{a^-} = \lng^{j+1}$, with covering multiplicities $i$ and $j+1$ respectively.

\end{rmk}

\subsection{Decomposing cobordisms into infinitesimal pieces}

Recall that for $d \in \Z_{\geq 1}$ and $a \in \R_{> 1}$ we have 
\begin{align*}
\wtTcount_d^a = \tfrac{1}{d!}\langle (\Xi^{1^+}_a)^d(\odot^d \orb^{1^+}_2),\orb^a_{3d-1}\rangle,
\end{align*}
where the extra $d!$ comes from the ordering of the positive punctures in the definition of $\Li$ cobordism map $\Xi^{1^+}_a$ from $C_{1^+}$ to $C_a$.
Observe that the right hand side naturally glues under decompositions of cobordisms, i.e. given $1 < a_1 < \cdots < a_M < a$ we have
\begin{align*}
\Xi^{1^+}_a = \Xi^{a_M}_{a}\circ \Xi^{a_{M-1}}_{a_M} \circ  \cdots \circ \Xi^{a_1}_{a_2} \circ \Xi^{1^+}_{a_1}.
\end{align*}
In fact, we can take a maximal decomposition into elementary pieces, which we call {\bf infinitesimal cobordism maps}, of the form $\Xi^{a^-}_{a^+}: C_{a^-} \ra C_{a^+}$ for $a \in \R_{> 1}$.
In fact, we have $(\Xi^{a^-}_{a^+})^d = 0$ for all $d \in \Z_{\geq 2}$ and all $a \in \R_{> 1}$ outside of a finite set $1 < a_1 < \cdots < a_M$.
We therefore have 
\begin{align*}
\Xi^{1^+}_a = \Xi^{a_M^+}_{a} \circ \Xi^{a_M^-}_{a_M^+} \circ\cdots\circ \Xi^{a_2^-}_{a_2^+} \circ \Xi^{a_1^+}_{a_2^-}\circ \Xi^{a_1^-}_{a_1^+} \circ \Xi^{1^+}_{a_1^-},
\end{align*}
where the $\Li$ maps $\Xi^{1^+}_{a_1^-},\Xi^{a_1^+}_{a_2^-},\dots,\Xi^{a_{M-1}^+}_{a_M^-},\Xi^{a_M^+}_{a}$ have no $k$-ary term for $k \geq 2$.
Since the latter maps are essentially trivial, all of the subtleties of the counts $\Tcount_d^a$ is reduced to understanding the infinitesimal cobordism maps $\Xi^{a_i^-}_{a_i^+}: C_{a^-} \ra C_{a^+}$.

Let us introduce the following shorthand:
\begin{notation}
For $a \in \R_{> 1}$ and $i_1,\dots,i_k \in \Z_{\geq 1}$, put
\begin{align*}
\infcount^{a^-}_{a^+}(i_1,\dots,i_k) &:= \langle (\Xi^{a^-}_{a^+})^k(\orb_{i_1}^{a^-},\dots,\orb_{i_k}^{a_-}),\orb^{a^+}_{i_1+\cdots+i_k+k-1} \rangle. 
\end{align*}
\end{notation}

\begin{rmk}
By the results in \S\ref{sec:formulas}, if $\infcount^{a^-}_{a^+}(i_1,\dots,i_k) \neq 0$ for some $i_1,\dots,i_k \in \Z_{\geq 1}$ with $k \geq 2$ then $a$ must lie in $\bigcup\limits_{s = 1}^{i_1+\cdots+i_k+k-1}\jump_s$.
Note that by energy considerations we must also have
\begin{align*}
\sum_{j=1}^k\calA(\Ga^{a}_{i_j}) \geq \calA(\Ga^{a}_{i_1+\cdots+i_k+k-1}).
\end{align*}
\end{rmk}

\subsection{Jump formulas}\label{subsec:jump_formulas}

We now compute the infinitesimal curve counts $\infcount^{a^-}_{a^+}(i_1,\dots,i_k)$ for any $a \in \R_{> 1}$ and $i_1,\dots,i_k \in \Z_{\geq 1}$. We first give derive simple explicit formulas for counts of cylinders (i.e. $k=1$) and pairs of pants (i.e. $k=2$), and then we give a recursive formula for the general case.

Recall that, by ~\eqref{eq:aug_Xi_compatibility}, for $a,b \in \R_{> 0}$ we have $\aug_{a} = \aug_{b} \circ \Xi^{a}_{b}$ as $\Li$ homomorphisms $C_{a} \ra C_o$, and hence $\Xi^{a}_{b} = \eta_{b} \circ \aug_{a}$,
i.e. for $i_1,\dots,i_k \in \Z_{\geq 1}$ we have
\begin{align}\label{eq:eta_circ_aug}
(\Xi^{a}_{b})^k(\orb_{i_1}^{a},\dots,\orb_{i_k}^{a}) &= \sum_{\substack{s \geq 1\\1 \leq k_1 \leq \cdots \leq k_s \\ k_1 + \cdots + k_s = k}}\sum_{\sigma \in \shufbar(k_1,\dots,k_s)} 
\eta^s_{b} \circ (\aug_{a}^{k_1}\odot \cdots \odot \aug_{a}^{k_s})(\orb^{a}_{i_{\sigma(1)}},\dots,\orb^{a}_{i_{\sigma(k)}}).
\end{align}

\subsubsection{Jump counts for cylinders}

For $i \in \Z_{\geq 1}$ we have
\begin{align*}
\Xi^{a}_{b}(\orb^a_i) = (\eta^1_b\circ \aug^1_a)(\orb_i^{a}) = \tfrac{(\Ga^b_i)!}{(\Ga^a_i)!} \cdot \orb^b_i,
\end{align*}
and hence $\infcount^{a^-}_{a^+}(i) = \frac{(\Ga^{a^+}_i)!}{(\Ga^{a^-}_i)!}$.
This means that the jump $\infcount^{a^-}_{a^+}(i)$ is $1$ unless $a$ lies in $\jump_i$, i.e. $a = \alpha/(\beta+1)$ for some $\alpha,\beta \in \Z_{\geq 0}$ with $\alpha + \beta = i$.
In that case we have $\Ga^{a^-}_i = (\alpha,\beta)$ and $\Ga^{a_+}_i = (\alpha-1,\beta+1)$, and hence  
\begin{align*}
\infcount^{a^-}_{a^+}(i) = \frac{\alpha!\beta!}{(\alpha-1)!(\beta+1)!} = \frac{\alpha}{\beta+1} = a.
\end{align*}
In summary:
\begin{lemma}
 For $a \in \R_{>0}$ and $i \in \Z_{\geq 1}$ we have
\begin{align*}
\infcount^{a^-}_{a^+}(i) = 
\begin{cases}
  a & \text{if}\;\;\;a \in \jump_i \\
  0 & \text{otherwise}.
\end{cases}
\end{align*}
\end{lemma}

\subsubsection{Jump counts for pairs of pants}

For $i,j \in \Z_{\geq 1}$ we have 
\begin{align*}
(\Xi^{a^-}_{a^+})^2(\orb^{a^-}_i,\orb^{a^-}_j) &= \eta_{a^+}^2(\aug_{a^-}^1(\orb_i^{a^-}),\aug_{a^-}^1(\orb_j^{a^-})) + \eta_{a^+}^1(\aug_{a^-}^2(\orb_i^{a^-},\orb_j^{a^-}))
\\&= \frac{\eta_{a^+}^2(\dg_i,\dg_j)}{(\Ga^{a^-}_i)!(\Ga^{a^-}_j)!} + \frac{\eta_{a^+}^1(\dg_{i+j+1})}{(\Ga^{a^-}_i+\Ga^{a^-}_j)!}.
\end{align*}
From the definition of $\eta_a$ and $\Li$ homomorphism relations we also have:
\begin{align*}
\eta_{a^+}^2(\dg_i,\dg_j) &= -\eta^1_{a^+}(\aug^2_{a^+}(\eta_{a^+}^1(\dg_i),\eta_{a^+}^1(\dg_j)))
\\&= -(\Ga^{a^+}_i)!(\Ga^{a^+}_j)! \cdot \eta^1_{a^+}(\aug^2_{a^+}(\orb_i^{a^+},\orb_j^{a^+}))
\\ &= -\frac{(\Ga^{a^+}_i)!(\Ga^{a^+}_j)!}{(\Ga_i^{a^+}+\Ga_j^{a^+})!} \cdot \eta^1_{a^+}(\dg_{i+j+1})
\\ &= -\frac{(\Ga^{a^+}_i)!(\Ga^{a^+}_j)!}{(\Ga_i^{a^+}+\Ga_j^{a^+})!} \cdot (\Ga^{a^+}_{i+j+1})! \cdot \orb^{a^+}_{i+j+1},
\end{align*}
and therefore:
\begin{lemma}
\begin{align*}
\infcount^{a^-}_{a^+}(i,j)  = 
-\frac{(\Ga^{a^+}_i)!(\Ga^{a^+}_j)!(\Ga^{a^+}_{i+j+1})!}{(\Ga_i^{a^+}+\Ga_j^{a^+})!(\Ga^{a^-}_i)!(\Ga^{a^-}_j)!}  + \frac{(\Ga^{a^+}_{i+j+1})!}{(\Ga^{a^-}_i+\Ga^{a^-}_j)!}.
\end{align*}
\end{lemma}

\begin{rmk}
A natural approach to proving Conjecture~\ref{conj:T_d^a_nondecreasing} would be to show that the jumps 
$\jump^{a^-}_{a^+}(i_1,\dots,i_k)$ are always nonnegative.
Unfortunately this is not the case, as for example we have
$\infcount^{(5/4)^-}_{(5/4)^+}(2,8) = -1/4$ (though it is noteworthy that all of the terms appearing in the examples below are indeed nonnegative).
\end{rmk}

\subsubsection{General recursive formula for jump counts}

To derive a general recursive formula in the spirit of Theorem~\ref{thm:main_recursion}, observe that by ~\eqref{eq:aug_Xi_compatibility} we have
\begin{align*}
\aug_{a^-}^k(\orb_{i_1}^{a^-},\dots,\orb_{i_k}^{a_-}) &= \sum_{\substack{s \geq 1\\1 \leq k_1 \leq \cdots \leq k_s \\ k_1 + \cdots + k_s = k}}\sum_{\sigma \in \shufbar(k_1,\dots,k_s)} 
\aug^s_{a^+} \circ ((\Xi^{a^-}_{a^+})^{k_1}\odot \cdots \odot (\Xi^{a^-}_{a^+})^{k_s})(\orb^{a^-}_{i_{\sigma(1)}},\dots,\orb^{a^-}_{i_{\sigma(k)}})
\\ &= \aug^1_{a^+}((\Xi^{a^-}_{a^+})^k(\orb^{a^-}_{i_1},\dots,\orb^{a^-}_{i_k})) 
\;\;\;\;\;+ \\&\;\;\;\;\;\sum_{\substack{s \geq 2\\1 \leq k_1 \leq \cdots \leq k_s \\ k_1 + \cdots + k_s = k}}\sum_{\sigma \in \shufbar(k_1,\dots,k_s)} 
\aug^s_{a^+} \circ ((\Xi^{a^-}_{a^+})^{k_1}\odot \cdots \odot (\Xi^{a^-}_{a^+})^{k_s})(\orb^{a^-}_{i_{\sigma(1)}},\dots,\orb^{a^-}_{i_{\sigma(k)}}).
\end{align*}

Applying $\eta_{a^+}^1$ to the above, we have
\begin{align*}
\eta^1_{a^+}(\aug^k_{a^-}(\orb^{a^-}_{i_1},\dots,\orb^{a^-}_{i_k}))
&= (\Xi^{a^-}_{a^+})^k(\orb^{a^-}_{i_1},\dots,\orb^{a^-}_{i_k}) 
\;\;\;\;\;+ \\&\;\;\;\;\;\sum_{\substack{s \geq 2\\1 \leq k_1 \leq \cdots \leq k_s \\ k_1 + \cdots + k_s = k}}\sum_{\sigma \in \shufbar(k_1,\dots,k_s)} 
\eta^1_{a^+} \circ \aug^s_{a^+} \circ ((\Xi^{a^-}_{a^+})^{k_1}\odot \cdots \odot (\Xi^{a^-}_{a^+})^{k_s})(\orb^{a^-}_{i_{\sigma(1)}},\dots,\orb^{a^-}_{i_{\sigma(k)}}),
\end{align*}
and hence:
\begin{thm}
For $a \in \R_{> 0}$ and $i_1,\dots,i_k \in \Z_{\geq 1}$ we have
\begin{align*}
\infcount^{a^-}_{a^+}(i_1,\dots,i_k) = \frac{(\Ga^{a^+}_{i_1+\cdots+i_k+k-1})!}{(\Ga^{a^-}_{i_1}+\cdots+\Ga^{a^-}_{i_k})!} - \sum_{\substack{s \geq 2\\1 \leq k_1 \leq \cdots \leq k_s \\ k_1 + \cdots + k_s = k}}\sum_{\sigma \in \shufbar(k_1,\dots,k_s)} \frac{(\Ga^{a^+}_{i_1+\cdots+i_k+k-1})!}{(\sum_{j=1}^s  \Ga^{a^+}_{i^j_1+\cdots+i^j_{k_j}+k_j-1})!}\prod_{r=1}^s\infcount^{a^-}_{a^+}(i^r_1,\dots,i^r_{k_r}),
\end{align*}
where $(i^1_1,\dots,i^1_{k_1})$ denotes the first $k_1$ elements of $(i_{\sigma(1)},\dots,i_{\sigma(k_1)})$, $(i^2_1,\dots,i^2_{k_2})$ denotes the next $k_2$ elements, and so on.
\end{thm}

\subsection{Examples}\label{subsec:jump_examples}

\begin{example}
Figure~\ref{fig:T_3_inf_cob} illustrates the computation of $\wtTcount_3 = \wtTcount_3^\infty$ via infinitesimal cobordism maps. 
In this case there is just a single nonzero term, although there are $3$ possible ways to feed the ordering inputs into the diagram.
Note that this gives $\Tcount_3 = 4$ since $\mult(\orb_8^\infty) = \mult(\sht^8) = 8$, which is consistent with the computation in \cite{McDuffSiegel_counting}.
\end{example}

\begin{figure}
  \includegraphics[scale=1.2]{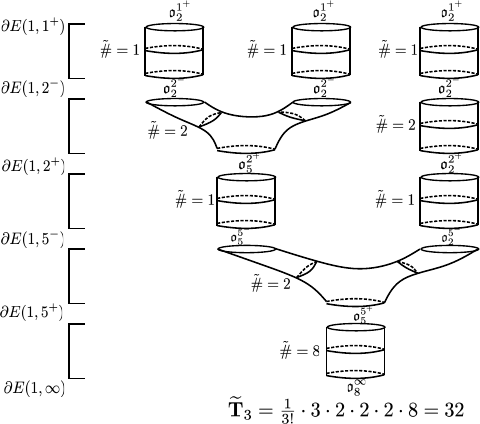}
  \caption{Computing $\wtTcount_3$ via infinitesimal cobordisms.}
  \label{fig:T_3_inf_cob} 
\end{figure}

\begin{example}
Figure~\ref{fig:T_5_13_2_inf_cob} shows the computation of $\wtTcount_{5}^{(13/2)^+}$ via jumps.
Note that this gives $\Tcount_5^{(13/2)^+} = 1$ since $\mult(\orb^{(13/2)^+}_{14}) = \mult(\sht^{13}) = 13$.
This is consistent with the computation from \cite{CGH} that $\Tcount_d^{(p/q)^+} = 1$
whenever $d = \tfrac{p+q}{3}$ and $p/q$ is an ``odd index Fibonacci ratio'', i.e. $p/q = 2,5,13/2,34/5$, etc.
According to Example~\ref{ex:deg_5_jumps}, this computation in fact remains valid for $a \in (13/2,8)$.
\end{example}

\begin{figure}
  \includegraphics[scale=1.2]{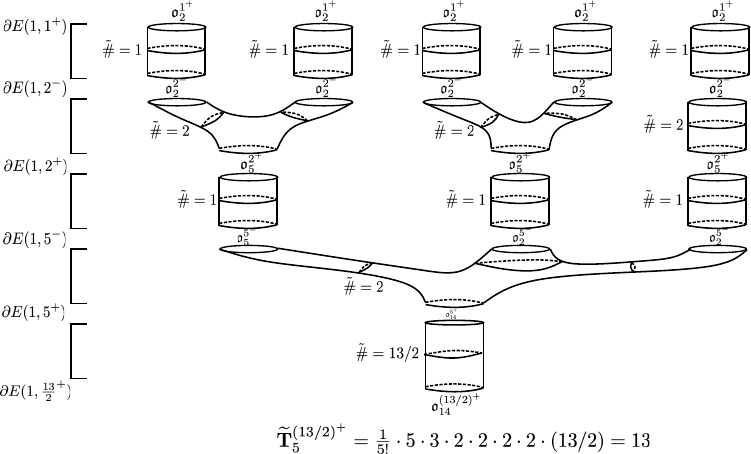}
  \caption{Computing $\wtTcount_5^{(13/2)^+}$ via infinitesimal cobordisms.}
  \label{fig:T_5_13_2_inf_cob} 
\end{figure}

\begin{rmk}
 Let $\tfrac{p_1}{q_1},\tfrac{p_2}{q_2},\tfrac{p_3}{q_3},\dots$ denote the ratios of odd index Fibonacci ratios, i.e. we put $\fib_1 = \fib_2 = 1$, $\fib_{i+2} = \fib_i + \fib_{i+1}$, and $\tfrac{p_i}{q_i} = \tfrac{\fib_{2i+3}}{\fib_{2i-1}}$.
Recall that $\lim\limits_{i \ra \infty} \tfrac{p_i}{q_i} = \tau^4 = \tfrac{3\sqrt{5}+7}{2}$, where $\tau = \tfrac{1+\sqrt{5}}{2}$ is the golden ratio. 
Note that Conjecture~\ref{conj:T_d^a_nondecreasing} would imply Conjecture~\ref{conj:p_q_d} for $p/q$ in the set 
$$\calS := \left\{\tfrac{p_i+k}{q_i-k}\;|\; i \in \Z_{\geq 1}, k \in \{0,\dots,q_i-1\}\right\}.$$
It is not difficult to show that $\calS$ is dense in $[\tau^4,\infty)$,
and hence is sufficient for computing the stabilized ellipsoid embedding function $f_{\CP^2 \times \C^N}$ for $N \in \Z_{\geq 1}$. 
\end{rmk}

\bibliographystyle{math}
\bibliography{biblio}

\end{document}